\documentclass[reqno]{amsart}
\usepackage{mathrsfs}
\usepackage{color}
\usepackage{amsmath}
\usepackage{amsfonts}
\usepackage{amssymb}
\usepackage{graphicx}
\usepackage{hyperref}

 \newtheorem{Theorem}{Theorem}[section]
 \newtheorem{Corollary}[Theorem]{Corollary}
 \newtheorem{Lemma}[Theorem]{Lemma}
 \newtheorem{Proposition}[Theorem]{Proposition}
 
 \newtheorem{Definition}[Theorem]{Definition}
\newtheorem{Problem}[Theorem]{Problem}
 \newtheorem{Conjecture}[Theorem]{Conjecture}

 \newtheorem{Remark}[Theorem]{Remark}

 \numberwithin{equation}{section}

\begin{document}

\title[Concavity property of minimal $L^2$ integrals \uppercase\expandafter{\romannumeral4}]
{Concavity property of minimal $L^2$ integrals with Lebesgue measurable gain \uppercase\expandafter{\romannumeral6} ----- fibrations over products of open Riemann surfaces}

\author{Shijie Bao}
\address{Shijie Bao: School of
Mathematical Sciences, Peking University, Beijing 100871, China.}
\email{bsjie@pku.edu.cn}

\author{Qi'an Guan}
\address{Qi'an Guan: School of
Mathematical Sciences, Peking University, Beijing 100871, China.}
\email{guanqian@math.pku.edu.cn}

\author{Zheng Yuan}
\address{Zheng Yuan: School of
Mathematical Sciences, Peking University, Beijing 100871, China.}
\email{zyuan@pku.edu.cn}

\thanks{}



\date{\today}

\dedicatory{}

\commby{}


\begin{abstract}
In this article, we present characterizations of the concavity property of minimal $L^2$ integrals degenerating to linearity in the case of fibrations over products of open Riemann surfaces.
As applications, we obtain characterizations of the holding of equality in optimal jets $L^2$ extension problem from fibers over products of analytic subsets to fibrations over products of open Riemann surfaces,
which implies characterizations of the equality parts of Suita conjecture and extended Suita conjecture for fibrations over products of open Riemann surfaces.
 \end{abstract}

\maketitle

\section{Introduction}\label{introduction}

The strong openness property of multiplier ideal sheaves \cite{GZSOC} (2-dim \cite{JonssonMustata}) i.e. $\mathcal{I}(\varphi)=\mathcal{I}_+(\varphi):=\mathop{\cup} \limits_{\epsilon>0}\mathcal{I}((1+\epsilon)\varphi)$ (conjectured by Demailly \cite{DemaillySoc})
has opened the door to new types of approximation techniques, which was used in the study of several complex variables, complex algebraic geometry and complex differential geometry
(see e.g. \cite{GZSOC,K16,cao17,cdM17,FoW18,DEL18,ZZ2018,GZ20,berndtsson20,ZZ2019,ZhouZhu20siu's,FoW20,KS20,DEL21}),
where $\varphi$ is a plurisubharmonic function of a complex manifold $M$ (see \cite{Demaillybook}), and the multiplier ideal sheaf $\mathcal{I}(\varphi)$ is defined as the sheaf of germs of holomorphic functions $f$ such that $|f|^2e^{-\varphi}$ is locally integrable (see e.g. \cite{Tian,Nadel,Siu96,DEL,DK01,DemaillySoc,DP03,Lazarsfeld,Siu05,Siu09,DemaillyAG,Guenancia}).

When $\mathcal{I}(\varphi)=\mathcal{O}$, the strong openness property degenerates to the openness property conjectured by Demailly-Koll\'ar \cite{DK01}.
Berndtsson \cite{Berndtsson2} (2-dim by Favre-Jonsson \cite{FavreJonsson}) proved the openness property by establishing an effectiveness result of the openness property.
Stimulated by Berndtsson's effectiveness result, and continuing the proof of the strong openness property \cite{GZSOC},
Guan-Zhou \cite{GZeff} established an effectiveness result of the strong openness property by considering the minimal $L^{2}$ integral on the pseudoconvex domain $D$.

Considering the minimal $L^{2}$ integrals on the sublevel sets of the weight $\varphi$,
Guan \cite{G16} obtained a sharp version of Guan-Zhou's effectiveness result,
and established a concavity property of the minimal $L^2$ integrals on the sublevel sets of the weight $\varphi$ (with constant gain).
The concavity property was applied to study the upper bound of the Bergman kernel i.e. a proof of Saitoh's conjecture for conjugate Hardy $H^2$ kernels \cite{Guan2019},
and equisingular approximations for the multiplier ideal sheaves i.e. the sufficient and necessary condition of the existence of decreasing equisingular approximations with analytic singularities for the multiplier ideal sheaves with weights $\log(|z_{1}|^{a_{1}}+\cdots+|z_{n}|^{a_{n}})$ \cite{guan-20}.

For smooth gain, Guan  \cite{G2018} (see also \cite{GM}) presented the concavity property on Stein manifolds (the weakly pseudoconvex K\"{a}hler case was obtained by Guan-Mi\cite{GM_Sci}).
The concavity property \cite{G2018} (see also \cite{GM}) was applied by Guan-Yuan to deduce an optimal support function related to the strong openness property \cite{GY-support} and an effectiveness result of the strong openness property in $L^p$ \cite{GY-lp-effe}.

For Lebesgue measurable gain, Guan-Yuan \cite{GY-concavity} obtained the concavity property on Stein manifolds (the weakly pseudoconvex K\"{a}hler case was obtained by Guan-Mi-Yuan \cite{GMY-concavity2}).
The concavity property \cite{GY-concavity} was applied by Guan-Yuan to deduce a twisted $L^p$ version of the strong openness property \cite{GY-twisted}.

As the linearity is a degenerate case of concavity,
a natural problem was posed in \cite{GY-concavity3}:

\begin{Problem}[\cite{GY-concavity3}]\label{Q:chara}
How to characterize the concavity property degenerating to linearity?
\end{Problem}

For 1-dim case, Guan-Yuan \cite{GY-concavity} gave an answer to Problem \ref{Q:chara} for single point, i.e. for weights may not be subharmonic (the case of subharmonic weights was answered by Guan-Mi \cite{GM}),
and Guan-Yuan \cite{GY-concavity3} gave an answer to Problem \ref{Q:chara} for finite points.
For products of open Riemann surfaces, Guan-Yuan \cite{GY-concavity4} gave answers to Problem \ref{Q:chara} for products of analytic subsets. 
Recently, Bao-Guan-Yuan \cite{BGY-concavity5} gave an answer to Problem \ref{Q:chara} for fibrations over open Riemann surfaces.

In the present article, we give answers to Problem \ref{Q:chara} for fibrations over products of open Riemann surfaces.

Let $\Omega_j$  be an open Riemann surface, which admits a nontrivial Green function $G_{\Omega_j}$ for any  $1\le j\le n_1$. Let $Y$ be an $n_2-$dimensional weakly pseudoconvex K\"ahler manifold, and let $K_Y$ be the canonical (holomorphic) line bundle on $Y$. Let $M=\left(\prod_{1\le j\le n_1}\Omega_j\right)\times Y$ be an $n-$dimensional complex manifold, where $n=n_1+n_2$. Let $\pi_{1}$, $\pi_{1,j}$ and $\pi_2$ be the natural projections from $M$ to $\prod_{1\le j\le n_1}\Omega_j$, $\Omega_j$ and $Y$ respectively. Let $K_M$ be the canonical (holomorphic) line bundle on $M$.

 Let $Z_j$ be a (closed) analytic subset of $\Omega_j$ for any $j\in\{1,\ldots,n_1\}$, and denote that $Z_0:=\left(\prod_{1\le j\le n_1}Z_j\right)\times Y\subset M$. For any $j\in\{1,\ldots,n_1\}$, let $\varphi_j$ be a subharmonic function on $\Omega_{j}$ such that $\varphi_j(z)>-\infty$ for any $z\in Z_j$. Let $\varphi_Y$ be a plurisubharmonic function on $Y$, and denote that $\varphi:=\sum_{1\le j\le n_1}\pi_{1,j}^*(\varphi_j)+\pi_2^*(\varphi_Y)$. Let $\psi$ be a plurisubharmonic function on $M$ such that $\{\psi<-t\}\backslash Z_0$ is a weakly pseudoconvex K\"ahler manifold for any $t\in\mathbb{R}$ and $\psi(z)=-\infty$ for any $z\in Z_0$.
Let $c$ be a positive function on $(0,+\infty)$ such that $\int_{0}^{+\infty}c(t)e^{-t}dt<+\infty$, $c(t)e^{-t}$ is decreasing on $(0,+\infty)$ and $c(-\psi)$ has a positive lower bound on any compact subset of $M\backslash Z_0$. Let $f$ be a holomorphic $(n,0)$ form on a neighborhood of $Z_0$.
Denote
\begin{equation*}
\begin{split}
\inf\bigg\{\int_{\{\psi<-t\}}|\tilde{f}|^{2}e^{-\varphi}c(-\psi):(\tilde{f}-f,z)\in(\mathcal{O}(K_M)&\otimes\mathcal{I}(\varphi+\psi))_{z}\mbox{ for any $z\in Z_0$} \\&\&{\,}\tilde{f}\in H^{0}(\{\psi<-t\},\mathcal{O}(K_{M}))\bigg\}
\end{split}
\end{equation*}
by $G(t;c)$ (without misunderstanding, we denote $G(t;c)$ by $G(t)$),  where $t\in[0,+\infty)$ and
$|f|^{2}:=\sqrt{-1}^{n^{2}}f\wedge\bar{f}$ for any $(n,0)$ form $f$.

Recall that $G(h^{-1}(r))$ is concave with respect to $r$ \cite{GMY-concavity2}, where $h(t)=\int_{t}^{+\infty}c(s)e^{-s}ds$ for any $t\ge0$. 

In the following section, we present the characterizations of the concavity of $G(h^{-1}(r))$ degenerating to linearity.

\subsection{Main results}
\

 We recall some notations (see \cite{OF81}, see also \cite{guan-zhou13ap,GY-concavity,GMY-concavity2}).
 Let $P_j:\Delta\rightarrow\Omega_j$ be the universal covering from unit disc $\Delta$ to $\Omega_j$.
 we call the holomorphic function $f$ (resp. holomorphic $(1,0)$ form $F$) on $\Delta$ a multiplicative function (resp. multiplicative differential (Prym differential)),
 if there is a character $\chi$, which is the representation of the fundamental group of $\Omega_j$, such that $g^{*}(f)=\chi(g)f$ (resp. $g^{*}(F)=\chi(g)F$),
 where $|\chi|=1$ and $g$ is an element of the fundamental group of $\Omega$. Denote the set of such kinds of $f$ (resp. $F$) by $\mathcal{O}^{\chi}(\Omega_j)$ (resp. $\Gamma^{\chi}(\Omega_j)$).

It is known that for any harmonic function $u$ on $\Omega_j$,
there exists a $\chi_{j,u}$ (called  character associate to $u$) and a multiplicative function $f_u\in\mathcal{O}^{\chi_{j,u}}(\Omega_j)$,
such that $|f_u|=P_j^{*}(e^{u})$.
If $u_1-u_2=\log|f|$, then $\chi_{j,u_1}=\chi_{j,u_2}$,
where $u_1$ and $u_2$ are harmonic functions on $\Omega_j$ and $f$ is a holomorphic function on $\Omega_j$. Let $z_j\in \Omega_j$.
Recall that for the Green function $G_{\Omega_j}(z,z_j)$,
there exist a $\chi_{j,z_j}$ and a multiplicative function $f_{z_j}\in\mathcal{O}^{\chi_{j,z_j}}(\Omega_j)$, such that $|f_{z_j}(z)|=P_j^{*}\left(e^{G_{\Omega_j}(z,z_j)}\right)$ (see \cite{suita72}).

Let $Z_0=\{z_0\}\times Y=\{(z_1,\ldots,z_{n_1})\}\times Y\subset M$.
Let $$\psi=\max_{1\le j\le n_1}\left\{2p_j\pi_{1,j}^{*}(G_{\Omega_j}(\cdot,z_j))\right\},$$ where $p_j$ is positive real number for $1\le j\le n_1$.
Let $w_j$ be a local coordinate on a neighborhood $V_{z_j}$ of $z_j\in\Omega_j$ satisfying $w_j(z_j)=0$. Denote that $V_0:=\prod_{1\le j\le n_1}V_{z_j}$, and $w:=(w_1,\ldots,w_{n_1})$ is a local coordinate on $V_0$ of $z_0\in \prod_{1\le j\le n_1}\Omega_j$. Denote that $E:=\left\{(\alpha_1,\ldots,\alpha_{n_1}):\sum_{1\le j\le n_1}\frac{\alpha_j+1}{p_j}=1\,\&\,\alpha_j\in\mathbb{Z}_{\ge0}\right\}$.
Let $f$ be a holomorphic $(n,0)$ form on $V_0\times Y\subset M$.

We present a characterization of the concavity of $G(h^{-1}(r))$ degenerating to linearity for the case $Z_0=\{z_0\}\times Y$.

\begin{Theorem}
	\label{thm:linear-fibra-single}
	Assume that $G(0)\in(0,+\infty)$.  $G(h^{-1}(r))$ is linear with respect to $r\in(0,\int_{0}^{+\infty}c(t)e^{-t}dt]$  if and only if the  following statements hold:
	
	$(1)$ $f=\sum_{\alpha\in E}\pi_{1}^*\left(w^{\alpha}dw_1\wedge\ldots\wedge dw_{n_1}\right)\wedge \pi_2^*(f_{\alpha})+g_0$ on $V_0\times Y$, where  $g_0$ is a holomorphic $(n,0)$ form on $V_0\times Y$ satisfying $(g_0,z)\in(\mathcal{O}(K_M)\otimes\mathcal{I}(\varphi+\psi))_{z}$ for any $z\in Z_0$ and $f_{\alpha}$ is a holomorphic $(n_2,0)$ form on $Y$ such that $\sum_{\alpha\in E}\int_{Y}|f_{\alpha}|^2e^{-\varphi_Y}\in(0,+\infty)$;
	
	$(2)$ $\varphi_j=2\log|g_j|+2u_j$, where $g_j$ is a holomorphic function on $\Omega_j$ such that $g_j(z_j)\not=0$ and $u_j$ is a harmonic function on $\Omega_j$ for any $1\le j\le n_1$;

    $(3)$ $\chi_{j,z_j}^{\alpha_j+1}=\chi_{j,-u_j}$ for any $j\in\{1,2,...,n\}$ and $\alpha\in E$ satisfying $f_{\alpha}\not\equiv 0$.
\end{Theorem}

Let $c_j(z)$ be the logarithmic capacity (see \cite{S-O69}) on $\Omega_j$, which is locally defined by
$$c_j(z_j):=\exp\lim_{z\rightarrow z_j}(G_{\Omega_j}(z,z_j)-\log|w_j(z)|).$$

\begin{Remark}
	Lemma \ref{l:G1=G2} shows that the above result also holds when we replace that sheaf $\mathcal{I}(\varphi+\psi)$ (in the definition of $G(t)$ and statement $(1)$ in Theorem \ref{thm:linear-fibra-single}) by $\mathcal{I}(\psi)$.
\end{Remark}

\begin{Remark}
	\label{r:fibra-single}When the three statements in Theorem \ref{thm:linear-fibra-single} hold,
$$\sum_{\alpha\in E}c_{\alpha}\left(\wedge_{1\le j\le n_1}\pi_{1,j}^*\left(g_j(P_j)_*\left(f_{u_j}f_{z_j}^{\alpha_j}df_{z_j}\right)\right)\right)\wedge \pi_2^*(f_{\alpha})$$
 is the unique holomorphic $(n,0)$ form $F$ on $M$ such that $(F-f,z)\in(\mathcal{O}(K_{M}))_{z}\otimes\mathcal{I}(\varphi+\psi)_{z}$ for any $z\in Z_0$ and
 \begin{displaymath}
 	\begin{split}
 		G(t)&=\int_{\{\psi<-t\}}|F|^2e^{-\varphi}c(-\psi)\\
 		&=\left(\int_t^{+\infty}c(s)e^{-s}ds\right)\sum_{\alpha\in E}\frac{(2\pi)^{n_1}e^{-\sum_{1\le j\le n_1}\varphi_j(z_{j})}}{\prod_{1\le j\le n_1}(\alpha_j+1)c_{j}(z_j)^{2\alpha_{j}+2}}\int_Y|f_{\alpha}|^2e^{-\varphi_Y}
 	\end{split}
 \end{displaymath}
	 for any $t\ge0$, where $f_{u_j}$ is a holomorphic function on $\Delta$ such that $|f_{u_j}|=P_j^*(e^{u_j})$ for any $j\in\{1,\ldots,n_1\}$, $f_{z_j}$ is a holomorphic function on $\Delta$ such that $|f_{z_j}|=P_j^*\left(e^{G_{\Omega_j}(\cdot,z_j)}\right)$ for any $j\in\{1,\ldots,n_1\}$ and $c_{\alpha}$ is a constant such that $c_{\alpha}=\prod_{1\le j\le n_1}\left(\lim_{z\rightarrow z_j}\frac{w_j^{\alpha_j}dw_j}{g_j(P_j)_*\left(f_{u_j}f_{z_j}^{\alpha_j}df_{z_j}\right)}\right)$ for any $\alpha\in E$. We prove the remark in Section \ref{sec:s-1}.
\end{Remark}

 Let $Z_j=\{z_{j,1},\ldots,z_{j,m_j}\}\subset\Omega_j$ for any  $j\in\{1,\ldots,n_1\}$, where $m_j$ is a positive integer.
Let
$$\psi=\max_{1\le j\le n_1}\left\{\pi_{1,j}^*\left(2\sum_{1\le k\le m_j}p_{j,k}G_{\Omega_j}(\cdot,z_{j,k})\right)\right\},$$
where $p_{j,k}$ is a positive real number.
Let $w_{j,k}$ be a local coordinate on a neighborhood $V_{z_{j,k}}\Subset\Omega_{j}$ of $z_{j,k}\in\Omega_j$ satisfying $w_{j,k}(z_{j,k})=0$ for any $j\in\{1,\ldots,n_1\}$ and $k\in\{1,\ldots,m_j\}$, where $V_{z_{j,k}}\cap V_{z_{j,k'}}=\emptyset$ for any $j$ and $k\not=k'$. Denote that $I_1:=\{(\beta_1,\ldots,\beta_{n_1}):1\le \beta_j\le m_j$ for any $j\in\{1,\ldots,n_1\}\}$, $V_{\beta}:=\prod_{1\le j\le n_1}V_{z_{j,\beta_j}}$ for any $\beta=(\beta_1,\ldots,\beta_{n_1})\in I_1$ and $w_{\beta}:=(w_{1,\beta_1},\ldots,w_{n_1,\beta_{n_1}})$ is a local coordinate on $V_{\beta}$ of $z_{\beta}:=(z_{1,\beta_1},\ldots,z_{n_1,\beta_{n_1}})\in \prod_{1\le j\le n_1}\Omega_j$ satisfying $w_\beta(z_\beta)=0$.

Let $\beta^*=(1,\ldots,1)\in I_1$, and let $\alpha_{\beta^*}=(\alpha_{\beta^*,1},\ldots,\alpha_{\beta^*,n_1})\in\mathbb{Z}_{\ge0}^{n_1}$.
Denote that $E':=\left\{\alpha\in\mathbb{Z}_{\ge0}^{n_1}:\sum_{1\le j\le n_1}\frac{\alpha_j+1}{p_{j,1}}>\sum_{1\le j\le n_1}\frac{\alpha_{\beta^*,j}+1}{p_{j,1}}\right\}$. Let $f$ be a holomorphic $(n,0)$ form on $\cup_{\beta\in I_1}V_{\beta}\times Y$ satisfying $f=\pi_1^*\left(w_{\beta^*}^{\alpha_{\beta^*}}dw_{1,1}\wedge\ldots\wedge dw_{n_1,1}\right)\wedge\pi_2^*\left(f_{\alpha_{\beta^*}}\right)+\sum_{\alpha\in E'}\pi_1^*(w^{\alpha}dw_{1,1}\wedge\ldots\wedge dw_{n_1,1})\wedge\pi_2^*(f_{\alpha})$ on $V_{\beta^*}\times Y$, where $f_{\alpha_{\beta^*}}$ and $f_{\alpha}$ are  holomorphic $(n_2,0)$ forms on $Y$.

We present a characterization of the concavity of $G(h^{-1}(r))$ degenerating to linearity for the case  $Z_j$ is a set of finite points.

\begin{Theorem}
	\label{thm:linear-fibra-finite}Assume that $G(0)\in(0,+\infty)$.  $G(h^{-1}(r))$ is linear with respect to $r\in(0,\int_0^{+\infty} c(s)e^{-s}ds]$ if and only if the following statements hold:

	$(1)$ $\varphi_j=2\log|g_j|+2u_j$ for any $j\in\{1,\ldots,n_1\}$, where $u_j$ is a harmonic function on $\Omega_j$ and $g_j$ is a holomorphic function on $\Omega_j$ satisfying $g_j(z_{j,k})\not=0$ for any $k\in\{1,\ldots,m_j\}$;
	
	$(2)$ There exists a nonnegative integer $\gamma_{j,k}$ for any $j\in\{1,\ldots,n_1\}$ and $k\in\{1,\ldots,m_j\}$, which satisfies that $\prod_{1\le k\leq m_j}\chi_{j,z_{j,k}}^{\gamma_{j,k}+1}=\chi_{j,-u_j}$ and $\sum_{1\le j\le n_1}\frac{\gamma_{j,\beta_j}+1}{p_{j,\beta_j}}=1$ for any $\beta\in I_1$;
	
	$(3)$ $f=\pi_1^*\left(c_{\beta}\left(\prod_{1\le j\le n_1}w_{j,\beta_j}^{\gamma_{j,\beta_j}}\right)dw_{1,\beta_1}\wedge\ldots\wedge dw_{n,\beta_n}\right)\wedge\pi_2^*(f_0)+g_\beta$ on $V_{\beta}\times Y$ for any $\beta\in I_1$, where $c_{\beta}$ is a constant, $f_0\not\equiv0$ is a holomorphic $(n_2,0)$ form on $Y$ satisfying $\int_Y|f_0|^2e^{-\varphi_2}<+\infty$, and $g_{\beta}$ is a holomorphic $(n,0)$ form on $V_{\beta}\times Y$ such that $(g_{\beta},z)\in(\mathcal{O}(K_M)\otimes\mathcal{I}(\varphi+\psi))_{z}$ for any $z\in\{z_\beta\}\times Y$;
	
	$(4)$ $c_{\beta}\prod_{1\le j\le n_1}\left(\lim_{z\rightarrow z_{j,\beta_j}}\frac{w_{j,\beta_j}^{\gamma_{j,\beta_j}}dw_{j,\beta_j}}{g_j(P_{j})_*\left(f_{u_j}\left(\prod_{1\le k\le m_j}f_{z_{j,k}}^{\gamma_{j,k}+1}\right)\left(\sum_{1\le k\le m_j}p_{j,k}\frac{df_{z_{j,k}}}{f_{z_{j,k}}}\right)\right)}\right)=c_0$ for any $\beta\in I_1$, where $c_0\in\mathbb{C}\backslash\{0\}$ is a constant independent of $\beta$, $f_{u_j}$ is a holomorphic function $\Delta$ such that $|f_{u_j}|=P_j^*(e^{u_j})$ and $f_{z_{j,k}}$ is a holomorphic function on $\Delta$ such that $|f_{z_{j,k}}|=P_j^*\left(e^{G_{\Omega_j}(\cdot,z_{j,k})}\right)$ for any $j\in\{1,\ldots,n_1\}$ and $k\in\{1,\ldots,m_j\}$.
\end{Theorem}

Denote that
\begin{equation*}
c_{j,k}:=\exp\lim_{z\rightarrow z_{j,k}}\left(\frac{\sum_{1\le k_1\le m_j}p_{j,k_1}G_{\Omega_j}(z,z_{j,k_1})}{p_{j,k}}-\log|w_{j,k}(z)|\right)
\end{equation*}
 for any $j\in\{1,\ldots,n_1\}$ and $k\in\{1,\ldots,m_j\}$.
\begin{Remark}
	\label{r:fibra-finite}When the four statements in Theorem \ref{thm:linear-fibra-finite} hold,
$$c_0\left(\wedge_{1\le j\le n_1}\pi_{1,j}^*\left(g_j(P_{j})_*\left(f_{u_j}\left(\prod_{k=1}^{m_j}f_{z_{j,k}}^{\gamma_{j,k}+1}\right)\left(\sum_{k=1}^{m_j}p_{j,k}\frac{df_{z_{j,k}}}{f_{z_{j,k}}}\right)\right)\right)\right)\wedge\pi_2^*(f_0)$$
 is the unique holomorphic $(n,0)$ form $F$ on $M$ such that $(F-f,z)\in(\mathcal{O}(K_{M}))_{z}\otimes\mathcal{I}(\varphi+\psi)_{z}$ for any $z\in Z_0$ and
 \begin{displaymath}
 	\begin{split}
 		G(t)&=\int_{\{\psi<-t\}}|F|^2e^{-\varphi}c(-\psi)\\
 		&=\left(\int_{t}^{+\infty}c(s)e^{-s}ds\right)\sum_{\beta\in I_1}\frac{|c_{\beta}|^2(2\pi)^{n_1}e^{-\sum_{1\le j\le n_1}\varphi_j(z_{j,\beta_j})}}{\prod_{1\le j\le n_1}(\gamma_{j,\beta_j}+1)c_{j,\beta_j}^{2\gamma_{j,\beta_j}+2}}\int_Y|f_0|^2e^{-\varphi_Y}
 	\end{split}
 \end{displaymath}
	 for any $t\ge0$. We prove the remark in Section \ref{sec:proof-2}.
\end{Remark}

 Let ${Z}_j=\{z_{j,k}:1\le k<\tilde m_j\}$ be a discrete subset of $\Omega_j$ for any  $j\in\{1,\ldots,n_1\}$, where $\tilde{m}_j\in\mathbb{Z}_{\ge2}\cup\{+\infty\}$.
Let $p_{j,k}$ be a positive number for any $1\le j\le n_1$ and $1\le k<\tilde m_j$ such that  $\sum_{1\le k<\tilde{m}_j}p_{j,k}G_{\Omega_j}(\cdot,z_{j,k})\not\equiv-\infty$ for any $j$.
Let
$$\psi=\max_{1\le j\le n_1}\left\{\pi_{1,j}^*\left(2\sum_{1\le k<\tilde{m}_j}p_{j,k}G_{\Omega_j}(\cdot,z_{j,k})\right)\right\}.$$ Assume that $\limsup_{t\rightarrow+\infty}c(t)<+\infty$.

Let $w_{j,k}$ be a local coordinate on a neighborhood $V_{z_{j,k}}\Subset\Omega_{j}$ of $z_{j,k}\in\Omega_j$ satisfying $w_{j,k}(z_{j,k})=0$ for any $j\in\{1,\ldots,n_1\}$ and $1\le k<\tilde{m}_j$, where $V_{z_{j,k}}\cap V_{z_{j,k'}}=\emptyset$ for any $j$ and $k\not=k'$. Denote that $\tilde I_1:=\{(\beta_1,\ldots,\beta_{n_1}):1\le \beta_j< \tilde m_j$ for any $j\in\{1,\ldots,n_1\}\}$, $V_{\beta}:=\prod_{1\le j\le n_1}V_{z_{j,\beta_j}}$ for any $\beta=(\beta_1,\ldots,\beta_{n_1})\in\tilde I_1$ and $w_{\beta}:=(w_{1,\beta_1},\ldots,w_{n_1,\beta_{n_1}})$ is a local coordinate on $V_{\beta}$ of $z_{\beta}:=(z_{1,\beta_1},\ldots,z_{n_1,\beta_{n_1}})\in \prod_{1\le j\le n_1}\Omega_j$.

Let $\beta^*=(1,\ldots,1)\in I_1$, and let $\alpha_{\beta^*}=(\alpha_{\beta^*,1},\ldots,\alpha_{\beta^*,n_1})\in\mathbb{Z}_{\ge0}^{n_1}$. Denote that $E':=\left\{\alpha\in\mathbb{Z}_{\ge0}^{n_1}:\sum_{1\le j\le n_1}\frac{\alpha_j+1}{p_{j,1}}>\sum_{1\le j\le n_1}\frac{\alpha_{\beta^*,j}+1}{p_{j,1}}\right\}$.
Let $f$ be a holomorphic $(n,0)$ form on $\cup_{\beta\in I_1}V_{\beta}\times Y$ satisfying $f=\pi_1^*\left(w_{\beta^*}^{\alpha_{\beta^*}}dw_{1,1}\wedge\ldots\wedge dw_{n_1,1}\right)\wedge\pi_2^*\left(f_{\alpha_{\beta^*}}\right)+\sum_{\alpha\in E'}\pi_1^*(w^{\alpha}dw_{1,1}\wedge\ldots\wedge dw_{n_1,1})\wedge\pi_2^*(f_{\alpha})$ on $V_{\beta^*}\times Y$, where $f_{\alpha_{\beta^*}}$ and $f_{\alpha}$ are  holomorphic $(n_2,0)$ forms on $Y$.

We present that $G(h^{-1}(r))$ is not linear when there exists $j_0\in\{1,\ldots,n_1\}$ such that $\tilde m_{j_0}=+\infty$ as follows.

\begin{Theorem}
	\label{thm:linear-fibra-infinite}If $G(0)\in(0,+\infty)$ and there exists $j_0\in\{1,\ldots,n_1\}$ such that $\tilde m_{j_0}=+\infty$, then $G(h^{-1}(r))$ is not linear with respect to $r\in(0,\int_0^{+\infty} c(s)e^{-s}ds]$.
\end{Theorem}

 Let ${Z}_j=\{z_{j,k}:1\le k<\tilde m_j\}$ be a discrete subset of $\Omega_j$ for any  $j\in\{1,\ldots,n_1\}$, where $\tilde{m}_j\in\mathbb{Z}_{\ge2}\cup\{+\infty\}$.
Let $p_{j,k}$ be a positive number for any $1\le j\le n_1$ and $1\le k<\tilde m_j$ such that $\sum_{1\le k<\tilde{m}_j}p_{j,k}G_{\Omega_j}(\cdot,z_{j,k})\not\equiv-\infty$ for any $j$.
Let
$$\psi=\max_{1\le j\le n_1}\left\{\pi_{1,j}^*\left(2\sum_{1\le k<\tilde{m}_j}p_{j,k}G_{\Omega_j}(\cdot,z_{j,k})\right)\right\}.$$

Let $M_1\subset M$ be an $n-$dimensional weakly pseudoconvex K\"ahler manifold satisfying that $Z_0\subset M_1$.  Let $f$ be a holomorphic $(n,0)$ form on a neighborhood $U_0\subset M_1$ of $Z_0$.
Replace $M$ in the definition of $G(t)$ by $M_1$.

\begin{Proposition}
	\label{p:M=M_1}If $G(0)\in(0,+\infty)$ and $G(h^{-1}(r))$ is  linear with respect to $r\in(0,\int_0^{+\infty} c(s)e^{-s}ds]$, we have $M_1=M$.
\end{Proposition}

\subsection{Applications}
\

Let $\Omega_j$  be an open Riemann surface, which admits a nontrivial Green function $G_{\Omega_j}$ for any  $1\le j\le n_1$. Let $Y$ be an $n_2-$dimensional weakly pseudoconvex K\"ahler manifold, and let $K_Y$ be the canonical (holomorphic) line bundle on $Y$. Let
$M=\left(\prod_{1\le j\le n_1}\Omega_j\right)\times Y$
 be an $n-$dimensional complex manifold, where $n=n_1+n_2$. Let $\pi_{1}$, $\pi_{1,j}$ and $\pi_2$ be the natural projections from $M$ to $\prod_{1\le j\le n_1}\Omega_j$, $\Omega_j$ and $Y$ respectively. Let $K_{M}$ be the canonical (holomorphic) line bundle on $M$.
Let $Z_j$ be a (closed) analytic subset of $\Omega_j$ for any $j\in\{1,\ldots,n_1\}$, and denote that $Z_0:=\left(\prod_{1\le j\le n_1}Z_j\right)\times Y$. Let $M_1\subset M$ be an $n-$dimensional complex manifold satisfying that $Z_0\subset M_1$, and let $K_{M_1}$ be the canonical (holomorphic) line bundle on $M_1$.

In this section, we present the characterizations of the holding of equality in optimal jets $L^2$ extension problem from $Z_0$ to $M_1$.

Let $Z_0=\{z_0\}\times Y\subset M_1$, where $z_0=(z_1,\ldots,z_{n_1})\in\prod_{1\le j\le n_1}\Omega_j$.
Let $w_j$ be a local coordinate on a neighborhood $V_{z_j}$ of $z_j\in\Omega_j$ satisfying $w_j(z_j)=0$. Denote that $V_0:=\prod_{1\le j\le n_1}V_{z_j}$, and $w:=(w_1,\ldots,w_{n_1})$ is a local coordinate on $V_0$ of $z_0\in \prod_{1\le j\le n_1}\Omega_j$. Let $\Psi\le0$ be a  plurisubharmonic function on $\prod_{1\le j\le n_1}\Omega_j$, and let $\varphi_j$ be a Lebesgue measurable function on $\Omega_j$ such that $\Psi+\sum_{1\le j\le n_1}\tilde\pi_{j}^*(\varphi_j)$ is plurisubharmonic on $\prod_{1\le j\le n_1}\Omega_j$, where $\tilde\pi_j$ is the natural projection from $\prod_{1\le j\le n_1}\Omega_j$ to $\Omega_j$. Let $\varphi_Y$ be a plurisubharmonic function on $Y$.  Denote that
$$\psi:=\max_{1\le j\le n_1}\left\{2p_j\pi_{1,j}^{*}(G_{\Omega_j}(\cdot,z_j))\right\}+\pi_1^*(\Psi)$$
 and $\varphi:=\sum_{1\le j\le n_1}\pi_{1,j}^*(\varphi_j)+\pi_2^*(\varphi_Y)$ on $M$, where $p_j$ is a positive real number for $1\le j\le n_1$.  Denote that $E:=\left\{(\alpha_1,\ldots,\alpha_{n_1}):\sum_{1\le j\le n_1}\frac{\alpha_j+1}{p_j}=1\,\&\,\alpha_j\in\mathbb{Z}_{\ge0}\right\}$ and $\tilde E:=\left\{(\alpha_1,\ldots,\alpha_{n_1}):\sum_{1\le j\le n_1}\frac{\alpha_j+1}{p_j}\ge1\,\&\,\alpha_j\in\mathbb{Z}_{\ge0}\right\}$.
Let
$$f=\sum_{\alpha\in\tilde E}\pi_1^*(w^{\alpha}dw_1\wedge\ldots\wedge dw_{n_1})\wedge \pi_2^*(f_{\alpha})$$
 be a holomorphic $(n,0)$ form on a neighborhood $U_0\subset (V_0\times Y)\cap M_1$ of $Z_0$,  where  $f_{\alpha}$ is a holomorphic $(n_2,0)$ form on $Y$. Let $c_j(z)$ be the logarithmic capacity (see \cite{S-O69}) on $\Omega_j$, which is locally defined by
$$c_j(z_j):=\exp\lim_{z\rightarrow z_j}(G_{\Omega_j}(z,z_j)-\log|w_j(z)|).$$

We obtain a characterization of the holding of equality in optimal jets $L^2$ extension problem for the case $Z_0=\{z_0\}\times Y$.

\begin{Theorem}
\label{thm:exten-fibra-single}
Let $c$ be a positive function on $(0,+\infty)$ such that $\int_{0}^{+\infty}c(t)e^{-t}dt<+\infty$ and $c(t)e^{-t}$ is decreasing on $(0,+\infty)$. Assume that
$$\sum_{\alpha\in E}\frac{(2\pi)^{n_1}e^{-\left(\Psi+\sum_{1\le j\le n_1}\tilde\pi_{j}^*(\varphi_j)\right)(z_0)}\int_Y|f_{\alpha}|^2e^{-\varphi_Y}}{\prod_{1\le j\le n_1}(\alpha_j+1)c_{j}(z_j)^{2\alpha_{j}+2}}\in(0,+\infty).$$
Then there exists a holomorphic $(n,0)$ form $F$ on $M_1$ satisfying that $(F-f,z)\in\left(\mathcal{O}(K_{M_1})\otimes\mathcal{I}\left(\max_{1\le j\le n_1}\left\{2p_j\pi_{1,j}^{*}(G_{\Omega_j}(\cdot,z_j))\right\}\right)\right)_{z}$ for any $z\in Z_0$ and
\begin{displaymath}
	\begin{split}
	&\int_{M_1}|F|^2e^{-\varphi}c(-\psi)\\
	\le&\left(\int_0^{+\infty}c(s)e^{-s}ds\right)\sum_{\alpha\in E}\frac{(2\pi)^{n_1}e^{-\left(\Psi+\sum_{1\le j\le n_1}\tilde\pi_{j}^*(\varphi_j)\right)(z_0)}\int_Y|f_{\alpha}|^2e^{-\varphi_Y}}{\prod_{1\le j\le n_1}(\alpha_j+1)c_{j}(z_j)^{2\alpha_{j}+2}}.	
	\end{split}
\end{displaymath}
	
	Moreover, equality $\inf\big\{\int_{M_1}|\tilde{F}|^2e^{-\varphi}c(-\psi):\tilde{F}\in H^0(M_1,\mathcal{O}(K_{M_1}))\,\&\, (\tilde{F}-f,z)\in(\mathcal{O}\left(K_{M_1})\otimes\mathcal{I}\left(\max_{1\le j\le n_1}\left\{2p_j\pi_{1,j}^{*}(G_{\Omega_j}(\cdot,z_j))\right\}\right)\right)_{z}$ for any $z\in Z_0\big\}=\left(\int_0^{+\infty}c(s)e^{-s}ds\right)\times\sum_{\alpha\in E}\frac{(2\pi)^{n_1}e^{-\left(\Psi+\sum_{1\le j\le n_1}\tilde\pi_{j}^*(\varphi_j)\right)(z_0)}\int_Y|f_{\alpha}|^2e^{-\varphi_Y}}{\prod_{1\le j\le n_1}(\alpha_j+1)c_{j}(z_j)^{2\alpha_{j}+2}}$ holds if and only if the following statements hold:

	$(1)$ $M_1=\left(\prod_{1\le j\le n_1}\Omega_j\right)\times Y$ and $\Psi\equiv0$;

	$(2)$ $\varphi_j=2\log|g_j|+2u_j$, where $g_j$ is a holomorphic function on $\Omega_j$ such that $g_j(z_j)\not=0$ and $u_j$ is a harmonic function on $\Omega_j$ for any $1\le j\le n_1$;

    $(3)$ $\chi_{j,z_j}^{\alpha_j+1}=\chi_{j,-u_j}$ for any $j\in\{1,2,...,n\}$ and $\alpha\in E$ satisfying $f_{\alpha}\not\equiv 0$.
\end{Theorem}

\begin{Remark}
	\label{r:1.7}If $(f_{\alpha},y)\in(\mathcal{O}(K_Y)\otimes\mathcal{I}(\varphi_Y))_y$ for any $y\in Y$ and $\alpha\in \tilde E\backslash E$, the above result also holds when we replace  the ideal sheaf $\mathcal{I}\left(\max_{1\le j\le n_1}\left\{2p_j\pi_{1,j}^{*}(G_{\Omega_j}(\cdot,z_j))\right\}\right)$  by $\mathcal{I}(\varphi+\psi)$. We prove the remark in Section \ref{sec:proof-1.7}.
\end{Remark}

\begin{Remark}
	Let $f$ be a holomorphic $(n,0)$ form on a neighborhood of $Z_0$. It follows from Lemma \ref{l:fibra-decom-2} that there exists a sequence of holomorphic $(n_2,0)$ form $\{f_{\alpha}\}_{\alpha\in\mathbb{Z}_{\ge0}^{n_1}}$ on $Y$ such that $f=\sum_{\alpha\in\mathbb{Z}_{\ge0}^{n_1}}\pi_1^*(w^{\alpha}dw_1\wedge\ldots\wedge dw_{n_1})\wedge\pi_2^*(f_{\alpha})$ on a neighborhood of $Z_0$. In the setting of Theorem \ref{thm:exten-fibra-single}, we assume that $f_{\alpha}\equiv0$ for $\alpha\in\mathbb{Z}_{\ge0}^{n_1}$ satisfying $\sum_{1\le j\le n_1}\frac{\alpha_j+1}{p_j}<1$.
	\end{Remark}

\begin{Remark}
	Let $\tilde\psi=\max_{1\le j\le n_1}\left\{2n_1\pi_{1,j}^*(G_{\Omega_j}(\cdot,z_j))\right\}$. It follows from Lemma \ref{l:0} that $(H_1-H_2,z)\in(\mathcal{O}(K_M)\otimes\mathcal{I}(\tilde\psi))_z$ for any $z\in Z_0$ if and only if $(H_1-H_2)|_{Z_0}=0$, where $H_1$ and $H_2$ are holomorphic $(n,0)$ forms on a neighborhood of $Z_0$. Thus, Theorem \ref{thm:exten-fibra-single} gives a characterization of the holding of equality in optimal $L^2$ extension theorem when $p_j=n_1$ for any $1\le j\le n_1$.
\end{Remark}

 Let $Z_j=\{z_{j,1},\ldots,z_{j,m_j}\}\subset\Omega_j$ for any  $j\in\{1,\ldots,n_1\}$, where $m_j$ is a positive integer.
Let $w_{j,k}$ be a local coordinate on a neighborhood $V_{z_{j,k}}\Subset\Omega_{j}$ of $z_{j,k}\in\Omega_j$ satisfying $w_{j,k}(z_{j,k})=0$ for any $j\in\{1,\ldots,n_1\}$ and $k\in\{1,\ldots,m_j\}$, where $V_{z_{j,k}}\cap V_{z_{j,k'}}=\emptyset$ for any $j$ and $k\not=k'$. Denote that $I_1:=\{(\beta_1,\ldots,\beta_{n_1}):1\le \beta_j\le m_j$ for any $j\in\{1,\ldots,n_1\}\}$, $V_{\beta}:=\prod_{1\le j\le n_1}V_{z_{j,\beta_j}}$ and $w_{\beta}:=(w_{1,\beta_1},\ldots,w_{n_1,\beta_{n_1}})$ is a local coordinate on $V_{\beta}$ of $z_{\beta}:=(z_{1,\beta_1},\ldots,z_{n_1,\beta_{n_1}})\in\prod_{1\le j\le n_1}\Omega_j$ for any $\beta=(\beta_1,\ldots,\beta_{n_1})\in I_1$. Then $Z_0=\{(z_{\beta},y):\beta\in  I_1\,\&\,y\in Y\}\subset M_1$.

Let $\Psi\le0$ be a  plurisubharmonic function on $\prod_{1\le j\le n_1}\Omega_j$, and let $\varphi_j$ be a Lebesgue measurable function on $\Omega_j$ such that $\Psi+\sum_{1\le j\le n_1}\tilde\pi_{j}^*(\varphi_j)$ is plurisubharmonic on $\prod_{1\le j\le n_1}\Omega_j$, where $\tilde\pi_j$ is the natural projection from $\prod_{1\le j\le n_1}\Omega_j$ to $\Omega_j$. Let $\varphi_Y$ be a plurisubharmonic function on $Y$. Denote that
$$\psi:=\max_{1\le j\le n_1}\left\{2\sum_{1\le k\le m_j}p_{j,k}\pi_{1,j}^{*}(G_{\Omega_j}(\cdot,z_{j,k}))\right\}+\pi_1^*(\Psi)$$
and $\varphi:=\sum_{1\le j\le n_1}\pi_{1,j}^*(\varphi_j)+\pi_2^*(\varphi_Y)$ on $M$, where $p_{j,k}$ is a positive real number for $1\le j\le n_1$ and $1\le k\le m_j$.

 Denote that $E_{\beta}:=\left\{(\alpha_1,\ldots,\alpha_{n_1}):\sum_{1\le j\le n_1}\frac{\alpha_j+1}{p_{j,\beta_j}}=1\,\&\,\alpha_j\in\mathbb{Z}_{\ge0}\right\}$ and $\tilde E_{\beta}:=\left\{(\alpha_1,\ldots,\alpha_{n_1}):\sum_{1\le j\le n_1}\frac{\alpha_j+1}{p_{j,\beta_j}}\ge1\,\&\,\alpha_j\in\mathbb{Z}_{\ge0}\right\}$ for any $\beta\in I_1$.
Let $f$ be a holomorphic $(n,0)$ form on a neighborhood $U_0\subset M_1$ of $Z_0$ such that
$$f=\sum_{\alpha\in\tilde E_{\beta}}\pi_1^*(w_{\beta}^{\alpha}dw_{1,\beta_1}\wedge\ldots\wedge dw_{n_1,\beta_{n_1}})\wedge\pi_2^*(f_{\alpha,\beta})$$ on $U_0\cap(V_{\beta}\times Y)$, where $f_{\alpha,\beta}$ is a holomorphic $(n_2,0)$ form on $Y$ for any $\alpha\in E_{\beta}$ and $\beta\in I_1$. Let $\beta^*=(1,\ldots,1)\in I_1$, and let $\alpha_{\beta^*}=(\alpha_{\beta^*,1},\ldots,\alpha_{\beta^*,n_1})\in E_{\beta^*}$. Denote that $E':=\left\{\alpha\in\mathbb{Z}_{\ge0}^{n_1}:\sum_{1\le j\le n_1}\frac{\alpha_j+1}{p_{j,1}}>1\right\}$. Assume that $f=\pi_1^*\left(w_{\beta^*}^{\alpha_{\beta^*}}dw_{1,1}\wedge\ldots\wedge dw_{n_1,1}\right)\wedge\pi_2^*\left(f_{\alpha_{\beta^*},\beta^*}\right)+\sum_{\alpha\in E'}\pi_1^*(w^{\alpha}dw_{1,1}\wedge\ldots\wedge dw_{n_1,1})\wedge\pi_2^*(f_{\alpha,\beta})$ on $U_0\cap(V_{\beta^*}\times Y)$.
 Denote that
\begin{equation*}
c_{j,k}:=\exp\lim_{z\rightarrow z_{j,k}}\left(\frac{\sum_{1\le k_1\le m_j}p_{j,k_1}G_{\Omega_j}(z,z_{j,k_1})}{p_{j,k}}-\log|w_{j,k}(z)|\right)
\end{equation*}
 for any $j\in\{1,\ldots,n\}$ and $k\in\{1,\ldots,m_j\}$.

We obtain a characterization of the holding of equality in optimal jets $L^2$ extension problem for the case that $Z_j$ is finite.

\begin{Theorem}
\label{thm:exten-fibra-finite}
Let $c$ be a positive function on $(0,+\infty)$ such that $\int_{0}^{+\infty}c(t)e^{-t}dt<+\infty$ and $c(t)e^{-t}$ is decreasing on $(0,+\infty)$. Assume that
$$\sum_{\beta\in I_1}\sum_{\alpha\in E_{\beta}}\frac{(2\pi)^{n_1}e^{-\left(\Psi+\sum_{1\le j\le n_1}\tilde\pi_j^*(\varphi_j)\right)(z_{\beta})}\int_Y|f_{\alpha,\beta}|^2e^{-\varphi_Y}}{\prod_{1\le j\le n_1}(\alpha_j+1)c_{j,\beta_j}^{2\alpha_{j}+2}}\in(0,+\infty).$$
Then there exists a holomorphic $(n,0)$ form $F$ on $M_1$ satisfying that $(F-f,z)\in\left(\mathcal{O}(K_{M_1})\otimes\mathcal{I}\left(\max_{1\le j\le n_1}\left\{2\sum_{1\le k\le m_j}p_{j,k}\pi_{1,j}^{*}(G_{\Omega_j}(\cdot,z_{j,k}))\right\}\right)\right)_{z}$ for any $z\in Z_0$ and
\begin{displaymath}
	\begin{split}
	&\int_{M_1}|F|^2e^{-\varphi}c(-\psi)\\
	\le&\left(\int_0^{+\infty}c(s)e^{-s}ds\right)\sum_{\beta\in I_1}\sum_{\alpha\in E_{\beta}}\frac{(2\pi)^{n_1}e^{-\left(\Psi+\sum_{1\le j\le n_1}\tilde\pi_j^*(\varphi_j)\right)(z_{\beta})}\int_Y|f_{\alpha,\beta}|^2e^{-\varphi_Y}}{\prod_{1\le j\le n_1}(\alpha_j+1)c_{j,\beta_j}^{2\alpha_{j}+2}}.	
	\end{split}
\end{displaymath}
	
	Moreover, equality $\inf\bigg\{\int_{M_1}|\tilde{F}|^2e^{-\varphi}c(-\psi):\tilde{F}\in H^0(M_1,\mathcal{O}(K_{M_1}))\,\&\, (\tilde{F}-f,z)\in\left(\mathcal{O}(K_{M_1})\otimes\mathcal{I}\left(\max_{1\le j\le n_1}\left\{2\sum_{1\le k\le m_j}p_{j,k}\pi_{1,j}^{*}(G_{\Omega_j}(\cdot,z_{j,k}))\right\}\right)\right)_{z}$ for any $z\in Z_0\bigg\}=\left(\int_0^{+\infty}c(s)e^{-s}ds\right)\sum_{\beta\in I_1}\sum_{\alpha\in E_{\beta}}\frac{(2\pi)^{n_1}e^{-\left(\Psi+\sum_{1\le j\le n_1}\tilde\pi_j^*(\varphi_j)\right)(z_{\beta})}\int_Y|f_{\alpha,\beta}|^2e^{-\varphi_Y}}{\prod_{1\le j\le n_1}(\alpha_j+1)c_{j,\beta_j}^{2\alpha_{j}+2}}$ holds if and only if the following statements hold:

	$(1)$ $M_1=\left(\prod_{1\le j\le n_1}\Omega_j\right)\times Y$ and $\Psi\equiv0$;
	
	$(2)$ $\varphi_j=2\log|g_j|+2u_j$ for any $j\in\{1,\ldots,n_1\}$, where $u_j$ is a harmonic function on $\Omega_j$ and $g_j$ is a holomorphic function on $\Omega_j$ satisfying $g_j(z_{j,k})\not=0$ for any $k\in\{1,\ldots,m_j\}$;
	
	$(3)$ There exists a nonnegative integer $\gamma_{j,k}$ for any $j\in\{1,\ldots,n_1\}$ and $k\in\{1,\ldots,m_j\}$, which satisfies that $\prod_{1\le k\leq m_j}\chi_{j,z_{j,k}}^{\gamma_{j,k}+1}=\chi_{j,-u_j}$ and $\sum_{1\le j\le n_1}\frac{\gamma_{j,\beta_j}+1}{p_{j,\beta_j}}=1$ for any $\beta\in I_1$;
	
	$(4)$ $f_{\alpha,\beta}=c_{\beta}f_0$ holds for $\alpha=(\gamma_{1,\beta_1},\ldots,\gamma_{n_1,\beta_{n_1}})$ and $f_{\alpha,\beta}\equiv0$ holds for any $\alpha\in E_{\beta}\backslash\{(\gamma_{1,\beta_1},\ldots,\gamma_{n_1,\beta_{n_1}})\}$, where $\beta\in I_1$, $c_{\beta}$ is a constant and $f_0\not\equiv0$ is a holomorphic $(n_2,0)$ form on $Y$ satisfying $\int_Y|f_0|^2e^{-\varphi_2}<+\infty$;
		
	$(5)$ $c_{\beta}\prod_{1\le j\le n_1}\left(\lim_{z\rightarrow z_{j,\beta_j}}\frac{w_{j,\beta_j}^{\gamma_{j,\beta_j}}dw_{j,\beta_j}}{g_j(P_{j})_*\left(f_{u_j}\left(\prod_{1\le k\le m_j}f_{z_{j,k}}^{\gamma_{j,k}+1}\right)\left(\sum_{1\le k\le m_j}p_{j,k}\frac{df_{z_{j,k}}}{f_{z_{j,k}}}\right)\right)}\right)=c_0$ for any $\beta\in I_1$, where $c_0\in\mathbb{C}\backslash\{0\}$ is a constant independent of $\beta$, $f_{u_j}$ is a holomorphic function $\Delta$ such that $|f_{u_j}|=P_j^*(e^{u_j})$ and $f_{z_{j,k}}$ is a holomorphic function on $\Delta$ such that $|f_{z_{j,k}}|=P_j^*\left(e^{G_{\Omega_j}(\cdot,z_{j,k})}\right)$ for any $j\in\{1,\ldots,n_1\}$ and $k\in\{1,\ldots,m_j\}$.
\end{Theorem}

\begin{Remark}
	\label{r:1.8} If $(f_{\alpha,\beta},y)\in(\mathcal{O}(K_Y)\otimes\mathcal{I}(\varphi_Y))_y$ holds for any $y\in Y$,  $\alpha\in \tilde E_{\beta}\backslash E_{\beta}$ and $\beta\in I_1$, the above result also holds when we replace  the ideal sheaf $\mathcal{I}\left(\max_{1\le j\le n_1}\left\{2\sum_{1\le k\le m_j}p_{j,k}\pi_{1,j}^{*}(G_{\Omega_j}(\cdot,z_{j,k}))\right\}\right)$  by $\mathcal{I}(\varphi+\psi)$. We prove the remark in Section \ref{sec:proof-1.8}.
\end{Remark}

 Let $Z_j=\{z_{j,k}:1\le k<\tilde m_j\}$ be a discrete subset of $\Omega_j$ for any  $j\in\{1,\ldots,n_1\}$, where $\tilde m_j\in\mathbb{Z}_{\ge2}\cup\{+\infty\}$.
Let $w_{j,k}$ be a local coordinate on a neighborhood $V_{z_{j,k}}\Subset\Omega_{j}$ of $z_{j,k}\in\Omega_j$ satisfying $w_{j,k}(z_{j,k})=0$ for any $1\le j\le n_1$ and $1\le k<\tilde m_j$, where $V_{z_{j,k}}\cap V_{z_{j,k'}}=\emptyset$ for any $j$ and $k\not=k'$. Denote that $\tilde I_1:=\{(\beta_1,\ldots,\beta_{n_1}):1\le \beta_j<\tilde m_j$ for any $j\in\{1,\ldots,n_1\}\}$, $V_{\beta}:=\prod_{1\le j\le n_1}V_{z_{j,\beta_j}}$  and $w_{\beta}:=(w_{1,\beta_1},\ldots,w_{n_1,\beta_{n_1}})$ is a local coordinate on $V_{\beta}$ of $z_{\beta}:=(z_{1,\beta_1},\ldots,z_{n_1,\beta_{n_1}})\in\prod_{1\le j\le n_1}\Omega_j$ for any $\beta=(\beta_1,\ldots,\beta_{n_1})\in\tilde I_1$. Then $Z_0=\{(z_{\beta},y):\beta\in\tilde I_1\,\&\,y\in Y\}\subset M_1$.

Let $\Psi\le0$ be a  plurisubharmonic function on $\prod_{1\le j\le n_1}\Omega_j$, and let $\varphi_j$ be a Lebesgue measurable function on $\Omega_j$ such that $\Psi+\sum_{1\le j\le n_1}\tilde\pi_{j}^*(\varphi_j)$ is plurisubharmonic on $\prod_{1\le j\le n_1}\Omega_j$, where $\tilde\pi_j$ is the natural projection from $\prod_{1\le j\le n_1}\Omega_j$ to $\Omega_j$. Let $\varphi_Y$ be a plurisubharmonic function on $Y$.
Let $p_{j,k}$ be a positive number for any $1\le j\le n_1$ and $1\le k<\tilde m_j$, which satisfies that $\sum_{1\le k<\tilde m_j}p_{j,k}G_{\Omega_j}(\cdot,z_{j,k})\not\equiv-\infty$ for any $1\le j\le n_1$.
Denote that
$$\psi:=\max_{1\le j\le n_1}\left\{2\sum_{1\le k<\tilde m_j}p_{j,k}\pi_{1,j}^{*}(G_{\Omega_j}(\cdot,z_{j,k}))\right\}+\pi_1^*(\Psi)$$
 and $\varphi:=\sum_{1\le j\le n_1}\pi_{1,j}^*(\varphi_j)+\pi_2^*(\varphi_Y)$ on $M$.

 Denote that $E_{\beta}:=\left\{(\alpha_1,\ldots,\alpha_{n_1}):\sum_{1\le j\le n_1}\frac{\alpha_j+1}{p_{j,\beta_j}}=1\,\&\,\alpha_j\in\mathbb{Z}_{\ge0}\right\}$ and $\tilde E_{\beta}:=\left\{(\alpha_1,\ldots,\alpha_{n_1}):\sum_{1\le j\le n_1}\frac{\alpha_j+1}{p_{j,\beta_j}}\ge1\,\&\,\alpha_j\in\mathbb{Z}_{\ge0}\right\}$ for any $\beta\in\tilde I_1$.
Let $f$ be a holomorphic $(n,0)$ form on a neighborhood $U_0\subset M_1$ of $Z_0$ such that
$$f=\sum_{\alpha\in\tilde E_{\beta}}\pi_1^*(w_{\beta}^{\alpha}dw_{1,\beta_1}\wedge\ldots\wedge dw_{n_1,\beta_{n_1}})\wedge\pi_2^*(f_{\alpha,\beta})$$ on $U_0\cap(V_{\beta}\times Y)$, where $f_{\alpha,\beta}$ is a holomorphic $(n_2,0)$ form on $Y$ for any $\alpha\in E_{\beta}$ and $\beta\in\tilde I_1$. Let $\beta^*=(1,\ldots,1)\in\tilde I_1$, and let $\alpha_{\beta^*}=(\alpha_{\beta^*,1},\ldots,\alpha_{\beta^*,n_1})\in E_{\beta^*}$. Denote that $E':=\left\{\alpha\in\mathbb{Z}_{\ge0}^{n_1}:\sum_{1\le j\le n_1}\frac{\alpha_j+1}{p_{j,1}}>1\right\}$. Assume that $f=\pi_1^*\left(w_{\beta^*}^{\alpha_{\beta^*}}dw_{1,1}\wedge\ldots\wedge dw_{n_1,1}\right)\wedge\pi_2^*\left(f_{\alpha_{\beta^*},\beta^*}\right)+\sum_{\alpha\in E'}\pi_1^*(w^{\alpha}dw_{1,1}\wedge\ldots\wedge dw_{n_1,1})\wedge\pi_2^*(f_{\alpha,\beta})$ on $U_0\cap(V_{\beta^*}\times Y)$.
Denote that
\begin{equation*}
c_{j,k}:=\exp\lim_{z\rightarrow z_{j,k}}\left(\frac{\sum_{1\le k_1<\tilde m_j}p_{j,k_1}G_{\Omega_j}(z,z_{j,k_1})}{p_{j,k}}-\log|w_{j,k}(z)|\right)
\end{equation*}
 for any $j\in\{1,\ldots,n\}$ and $1\le k<\tilde m_j$ (following from Lemma \ref{l:green-sup} and Lemma \ref{l:green-sup2}, we get that the above limit exists).

We obtain that the equality in optimal jets $L^2$ extension problem could not hold when there exists $j_0\in\{1,\ldots,n_1\}$ such that $\tilde m_{j_0}=+\infty$.

\begin{Theorem}
\label{thm:exten-fibra-infinite}Let $c$ be a positive function on $(0,+\infty)$ such that $\int_{0}^{+\infty}c(t)e^{-t}dt<+\infty$ and $c(t)e^{-t}$ is decreasing on $(0,+\infty)$. Assume that
$$\sum_{\beta\in\tilde I_1}\sum_{\alpha\in E_{\beta}}\frac{(2\pi)^{n_1}e^{-\left(\Psi+\sum_{1\le j\le n_1}\tilde\pi_j^*(\varphi_j)\right)(z_{\beta})}\int_Y|f_{\alpha,\beta}|^2e^{-\varphi_Y}}{\prod_{1\le j\le n_1}(\alpha_j+1)c_{j,\beta_j}^{2\alpha_{j}+2}}\in(0,+\infty)$$
and there exists $j_0\in\{1,\ldots,n_1\}$ such that $\tilde m_{j_0}=+\infty$.

Then there exists a holomorphic $(n,0)$ form $F$ on $M_1$ satisfying that $(F-f,z)\in\left(\mathcal{O}(K_{M_1})\otimes\mathcal{I}\left(\max_{1\le j\le n_1}\left\{2\sum_{1\le k<\tilde m_j}p_{j,k}\pi_{1,j}^{*}(G_{\Omega_j}(\cdot,z_{j,k}))\right\}\right)\right)_{z}$ for any $z\in Z_0$ and
\begin{displaymath}
	\begin{split}
	&\int_{M_1}|F|^2e^{-\varphi}c(-\psi)\\
<&\left(\int_0^{+\infty}c(s)e^{-s}ds\right)\sum_{\beta\in\tilde I_1}\sum_{\alpha\in E_{\beta}}\frac{(2\pi)^{n_1}e^{-\left(\Psi+\sum_{1\le j\le n_1}\tilde\pi_j^*(\varphi_j)\right)(z_{\beta})}\int_Y|f_{\alpha,\beta}|^2e^{-\varphi_Y}}{\prod_{1\le j\le n_1}(\alpha_j+1)c_{j,\beta_j}^{2\alpha_{j}+2}}.	
	\end{split}
\end{displaymath}
\end{Theorem}

\begin{Remark}
	\label{r:1.9}If $(f_{\alpha,\beta},y)\in(\mathcal{O}(K_Y)\otimes\mathcal{I}(\varphi_Y))_y$ holds for any $y\in Y$, $\alpha\in \tilde E_{\beta}\backslash E_{\beta}$ and $\beta\in\tilde I_1$, the above result also holds when we replace  the ideal sheaf $\mathcal{I}\left(\max_{1\le j\le n_1}\left\{2\sum_{1\le k<\tilde m_j}p_{j,k}\pi_{1,j}^{*}(G_{\Omega_j}(\cdot,z_{j,k}))\right\}\right)$  by $\mathcal{I}(\varphi+\psi)$. We prove the remark in Section \ref{sec:proof-1.9}.
\end{Remark}

\subsubsection{\textbf{Suita conjecture and extended Suita conjecture}}\label{sec:1.2.1}

\
\

In this section, we present characterizations of the  equality parts of Suita conjecture and extended Suita conjecture for fibrations over products of open Riemann surfaces.

Let $\Omega$  be an open Riemann surface, which admits a nontrivial Green function $G_{\Omega}$. Let $w$ be a local coordinate on a neighborhood $V_{z_0}$ of $z_0\in\Omega$ satisfying $w(z_0)=0$. Let $\kappa_{\Omega}$ be the Bergman kernel for holomorphic $(1,0)$ form on $\Omega$. We define that
$$B_{\Omega}(z)dw\otimes\overline{dw}:=\kappa_{\Omega}|_{V_{z_0}}.$$Let $c_{\beta}(z)$ be the logarithmic capacity (see \cite{S-O69}) which is locally defined by
$$c_{\beta}(z_0):=\exp\lim_{z\rightarrow z_0}(G_{\Omega}(z,z_0)-\log|w(z)|)$$
on $\Omega$.
In \cite{suita72}, Suita stated a conjecture as below.
\begin{Conjecture}
	$c_{\beta}(z_0)^2\le\pi B_{\Omega}(z_0)$ holds for any $z_0\in \Omega$, and equality holds if and only if $\Omega$ is conformally equivalent to the unit disc less a (possible) closed set of inner capacity zero.
\end{Conjecture}
The inequality part of  Suita conjecture for bounded planar domain was proved by B\l ocki \cite{Blo13}, and original form of the inequality was proved by Guan-Zhou \cite{gz12}.
The equality part of Suita conjecture was proved by Guan-Zhou \cite{guan-zhou13ap}, which completed the proof of Suita conjecture.

Let $\Omega_j$  be an open Riemann surface, which admits a nontrivial Green function $G_{\Omega_j}$ for any  $1\le j\le n_1$. Let $Y$ be an $n_2-$dimensional weakly pseudoconvex K\"ahler manifold, and let $K_Y$ be the canonical (holomorphic) line bundle on $Y$. Let $M=\left(\prod_{1\le j\le n_1}\Omega_j\right)\times Y$ be an $n-$dimensional complex manifold, where $n=n_1+n_2$. Let $\pi_{1}$, $\pi_{1,j}$ and $\pi_2$ be the natural projections from $M$ to $\prod_{1\le j\le n_1}\Omega_j$, $\Omega_j$ and $Y$ respectively. Let $K_M$ be the canonical (holomorphic) line bundle on $M$.

 Denote the space of $L^2$ integrable holomorphic section of $K_M$ (resp. $K_Y$) by $A^2(M,K_M,dV_M^{-1},dV_M)$ (resp. $A^2(Y,K_Y,dV_Y^{-1},dV_Y)$).
Let $\{\sigma_l\}_{l=1}^{+\infty}$ (resp. $\{\tau_l\}_{l=1}^{+\infty}$) be a complete orthogonal system of $A^2(M,K_M,dV_M^{-1},dV_M)$ (resp. $A^2(Y,K_Y,dV_Y^{-1},dV_Y)$) satisfying $(\sqrt{-1})^{n^2}\int_{M}\frac{\sigma_i}{\sqrt{2^n}}\wedge\frac{\overline{\sigma}_j}{\sqrt{2^n}}=\delta_i^j$. Put $\kappa_M=\sum_{l=1}^{+\infty}\sigma_l\otimes\overline\sigma_l\in C^{\omega}(M,K_M\otimes\overline{K_M})$ and $\kappa_Y=\sum_{l=1}^{+\infty}\tau_l\otimes\overline\tau_l\in C^{\omega}(Y,K_Y\otimes\overline{K_Y})$.

Let $z_0=(z_1,\ldots,z_{n_1})\in\prod_{1\le j\le n_1}\Omega_j$, and let $y_0\in Y$.
Let $w_j$ be a local coordinate on a neighborhood $V_{z_j}$ of $z_j\in\Omega_j$ satisfying $w_j(z_j)=0$. Denote that $V_0:=\prod_{1\le j\le n_1}V_{z_j}$, and $w:=(w_1,\ldots,w_{n_1})$ is a local coordinate on $V_0$ of $z_0$. Let $\tilde w=(\tilde w_1,\ldots,\tilde w_{n_2})$ be a local coordinate on a neighborhood $U_0$ of $y_0$ satisfying $\tilde w(y_0)=0$.  We define
$$B_{M}((z,y))dw_1\wedge\ldots\wedge dw_{n_1}\wedge d\tilde w_1\wedge\ldots d\tilde w_{n_2} \otimes\overline{dw_1\wedge\ldots\wedge dw_{n_1}\wedge d\tilde w_1\wedge\ldots d\tilde w_{n_2}}:=\kappa_{M}$$
on $V_{0}\times U_0$ and
$$B_{Y}(y) d\tilde w_1\wedge\ldots d\tilde w_{n_2} \otimes\overline{d\tilde w_1\wedge\ldots d\tilde w_{n_2}}:=\kappa_{Y}$$
on $U_0$.
Let $c_{j}(z_j)$ be the logarithmic capacity which is locally defined by
$$c_{j}(z_j):=\exp\lim_{z\rightarrow z_j}(G_{\Omega_j}(z,z_j)-\log|w_j(z)|).$$

Assume that $B_Y(y_0)>0$. Theorem \ref{thm:exten-fibra-single} gives a characterization of the holding of equality in Suita conjecture for fibrations over products of open Riemann surfaces.
\begin{Theorem}
	\label{thm:suita}
	$\prod_{1\le j\le n_1}c_j(z_j)^{2}B_Y(y_0)\le \pi^{n_1} B_M((z_0,y_0))$ holds, and equality holds if and only if $\Omega_j$ is conformally equivalent to the unit disc less a (possible) closed set of inner capacity zero for any $j\in\{1,\ldots,n\}$.
\end{Theorem}

Let $M_1\subset M$ be an $n-$dimensional complex manifold satisfying that $\{z_0\}\times Y \subset M_1$. Similar to $M$, we can define the Bergman kernel $B_{M_1}$. Theorem \ref{thm:suita} implies the following result.

\begin{Remark}
	\label{r:suita}	$\prod_{1\le j\le n_1}c_j(z_j)^{2}B_Y(y_0)\le \pi^{n_1} B_{M_1}((z_0,y_0))$ holds, and equality holds if and only if $M_1=M$ and $\Omega_j$ is conformally equivalent to the unit disc less a (possible) closed set of inner capacity zero for any $j\in\{1,\ldots,n\}$.
\end{Remark}

Let $\Omega$  be an open Riemann surface, which admits a nontrivial Green function $G_{\Omega}$, and let $K_{\Omega}$ be the canonical (holomorphic) line bundle on $\Omega$. Let $w$ be a local coordinate on a neighborhood $V_{z_0}$ of $z_0\in\Omega$ satisfying $w(z_0)=0$. Let $\rho=e^{-2u}$ on $\Omega$, where $u$ is a harmonic function on $\Omega$. We define that
$$B_{\Omega,\rho}dw\otimes\overline{dw}:=\sum_{l=1}^{+\infty}(\sigma_l\otimes\overline{\sigma}_l)|_{V_{z_0}}\in C^{\omega}(V_{z_0},K_{\Omega}\otimes\overline{K_{\Omega}}),$$
where $\{\sigma_l\}_{l=1}^{+\infty}$ are holomorphic $(1,0)$ forms on $\Omega$ satisfying $\sqrt{-1}\int_{\Omega}\rho\frac{\sigma_i}{\sqrt{2}}\wedge\frac{\overline{\sigma}_j}{\sqrt{2}}=\delta_i^j$ and $\big\{F\in H^0(\Omega,K_{\Omega}):\int_{\Omega}\rho|F|^2<+\infty\,\&\,\int_{\Omega}\rho\sigma_l\wedge\overline F=0$ for any $l\in\mathbb{Z}_{>0}\big\}=\{0\}$.

In \cite{Yamada}, Yamada  stated a conjecture as below (so-called extended Suita conjecture).
\begin{Conjecture}
	$c_{\beta}(z_0)^2\le\pi \rho(z_0) B_{\Omega,\rho}(z_0)$ holds for any $z_0\in \Omega$, and equality holds if and only $\chi_{-u}=\chi_{z_0}$, where $\chi_{-u}$ and $\chi_{z_0}$ are the characters associated to the functions $-u$ and $G_{\Omega}(\cdot,z_0)$ respectively.
\end{Conjecture}
The inequality part of extended Suita conjecture  was proved by Guan-Zhou \cite{GZ15}.
The equality part of extended Suita conjecture was proved by Guan-Zhou \cite{guan-zhou13ap}.

Let $\rho=e^{-2\sum_{1\le j\le n_1}\pi_{1,j}^*(u_j)}$ on $M$, where $u_j$ is a harmonic function on $\Omega_j$ for any $j\in\{1,\ldots,n\}$. We define that
$$B_{M,\rho}dw_1\wedge\ldots\wedge dw_{n_1}\wedge d\tilde w_1\wedge\ldots d\tilde w_{n_2} \otimes\overline{dw_1\wedge\ldots\wedge dw_{n_1}\wedge d\tilde w_1\wedge\ldots d\tilde w_{n_2}}:=\sum_{l=1}^{+\infty}e_l\otimes\overline{e}_l$$
on $V_0\times Y$,
where $\{e_l\}_{l=1}^{+\infty}$ are holomorphic $(n,0)$ forms on $M$ satisfying $(\sqrt{-1})^{n^2}\int_{M}\rho\frac{e_i}{\sqrt{2^n}}\wedge\frac{\overline{e}_j}{\sqrt{2^n}}=\delta_i^j$ and $\big\{F\in H^0(M,K_{M}):\int_{M}\rho|F|^2<+\infty\,\&\,\int_{M}\rho e_l\wedge\overline F=0$ for any $l\in\mathbb{Z}_{>0}\big\}=\{0\}$.

Assume that $B_Y(y_0)>0$. Theorem \ref{thm:exten-fibra-single}  gives  a characterization of the holding of equality in the extended Suita conjecture for fibrations over products of open Riemann surfaces.
 \begin{Theorem}
	\label{thm:extend}
	$\prod_{1\le j\le n_1}c_j(z_j)^{2}B_Y(y_0)\le \pi^{n_1} \rho(z_0) B_{M,\rho}(z_0)$ holds, and equality holds if and only if $\chi_{j,-u_j}=\chi_{j,z_j}$ for any $j\in\{1,\ldots,n\}$, where $\chi_{j,-u_j}$ and $\chi_{j,z_j}$ are the characters associated to the functions $-u$ and $G_{\Omega}(\cdot,z_0)$ respectively.
\end{Theorem}

Let $M_1\subset M$ be an $n-$dimensional complex manifold satisfying that $\{z_0\}\times Y \subset M_1$. Similar to $M$, we can define the Bergman kernel $B_{M_1,\rho}$. Theorem \ref{thm:extend} implies the following result.

\begin{Remark}
	\label{r:extend}$\prod_{1\le j\le n_1}c_j(z_j)^{2}B_Y(y_0)\le \pi^{n_1} B_{M_1,\rho}((z_0,y_0))$ holds, and equality holds if and only if $M_1=M$ and $\chi_{j,-u_j}=\chi_{j,z_j}$ for any $j\in\{1,\ldots,n\}$.

\end{Remark}

\section{Preparation}

\subsection{Concavity property of minimal $L^2$ integrals}
\

In this section, we recall some results about the concavity property of minimal $L^2$ integrals (see \cite{GMY-concavity2,GY-concavity4}).

Let $M$ be a complex manifold. Let $X$ and $Z$ be closed subsets of $M$. We say that a triple $(M,X,Z)$ satisfies condition $(A)$, if the following statements hold:

$\uppercase\expandafter{\romannumeral1}.$ $X$ is a closed subset of $M$ and $X$ is locally negligible with respect to $L^2$ holomorphic functions; i.e., for any local coordinated neighborhood $U\subset M$ and for any $L^2$ holomorphic function $f$ on $U\backslash X$, there exists an $L^2$ holomorphic function $\tilde{f}$ on $U$ such that $\tilde{f}|_{U\backslash X}=f$ with the same $L^2$ norm;

$\uppercase\expandafter{\romannumeral2}.$ $Z$ is an analytic subset of $M$ and $M\backslash (X\cup Z)$ is a weakly pseudoconvex K\"ahler manifold.

Let $M$ be an $n-$dimensional complex manifold, and let $(M,X,Z)$  satisfy condition $(A)$. Let $K_M$ be the canonical line bundle on $M$. Let $\psi$ be a plurisubharmonic function on $M$ such that $\{\psi<-t\}\backslash (X\cup Z)$ is a weakly pseudoconvex K\"ahler manifold for any $t\in\mathbb{R}$, and let $\varphi$ be a Lebesgue measurable function on $M$ such that $\psi+\varphi$ is a plurisubharmonic function on $M$. Denote $T=-\sup\limits_M \psi$.
\begin{Definition}
We call a positive measurable function $c$  on $(T,+\infty)$ in class $P_{T,M}$ if the following two statements hold:
\par
$(1)$ $c(t)e^{-t}$ is decreasing with respect to $t$;
\par
$(2)$ there is a closed subset $E$ of $M$ such that $E\subset Z\cap \{\psi(z)=-\infty\}$ and for any compact subset $K\subset M\backslash E$, $e^{-\varphi}c(-\psi)$ has a positive lower bound on $K$.
\end{Definition}

 Let $Z_0$ be a subset of $\{\psi=-\infty\}$ such that $Z_0 \cap
Supp(\mathcal{O}/\mathcal{I}(\varphi+\psi))\neq \emptyset$. Let $U \supset Z_0$ be
an open subset of $M$, and let $f$ be a holomorphic $(n,0)$ form on $U$. Let $\mathcal{F}_{z_0} \supset \mathcal{I}(\varphi+\psi)_{z_0}$ be an ideal of $\mathcal{O}_{z_0}$ for any $z_{0}\in Z_0$.

Denote
\begin{equation*}
\begin{split}
\inf\bigg\{\int_{\{\psi<-t\}}|\tilde{f}|^{2}e^{-\varphi}c(-\psi):(\tilde{f}-f)\in H^{0}(Z_0,&
(\mathcal{O}(K_{M})\otimes\mathcal{F})|_{Z_0})\\&\&{\,}\tilde{f}\in H^{0}(\{\psi<-t\},\mathcal{O}(K_{M}))\bigg\}
\end{split}
\end{equation*}
by $G(t;c)$ ($G(t)$ for short),  where $t\in[T,+\infty)$, $c$ is a nonnegative function on $(T,+\infty)$,
$|f|^{2}:=\sqrt{-1}^{n^{2}}f\wedge\bar{f}$ for any $(n,0)$ form $f$ and $(\tilde{f}-f)\in H^{0}(Z_0,
(\mathcal{O}(K_{M})\otimes\mathcal{F})|_{Z_0})$ means $(\tilde{f}-f,z)\in\mathcal{O}(K_{M})_{z}\otimes\mathcal{F}_{z}$ for all $z\in Z_0$.

The following Theorem shows the concavity for $G(t)$.
\begin{Theorem}[\cite{GMY-concavity2}]
\label{thm:general_concave}
Let $c\in\mathcal{P}_{T,M}$ satisfying $\int_T^{+\infty}c(s)e^{-s}ds<+\infty$. If there exists $t\in[T,+\infty)$ satisfying that $G(t)<+\infty$, then $G(h^{-1}(r))$ is concave with respect to $r\in(0,\int_{T}^{+\infty}c(s)e^{-s}ds)$, $\lim_{t\rightarrow T+0}G(t)=G(T)$ and $\lim_{t\rightarrow +\infty}G(t)=0$, where $h(t)=\int_{t}^{+\infty}c(s)e^{-s}ds$.
\end{Theorem}

Denote that
\begin{displaymath}
	\begin{split}
		\mathcal{H}^2(c,t):=\bigg\{\tilde{f}:\int_{\{\psi<-t\}}|\tilde{f}|^2e^{-\varphi}c(-\psi)<+\infty,(\tilde{f}-f&)\in H^0(Z_0,(\mathcal{O}(K_{M})\otimes\mathcal{F})|_{Z_0})\\
		&\&\tilde{f}\in H^0(\{\psi<-t\},\mathcal{O}(K_M))\bigg\},
	\end{split}
\end{displaymath}
where $t\in[T,+\infty)$ and $c$ is a nonnegative measurable function on $(T,+\infty)$.

\begin{Corollary}[\cite{GMY-concavity2}]	\label{c:linear}
Let $c\in\mathcal{P}_{T,M}$ satisfying $\int_T^{+\infty}c(s)e^{-s}ds<+\infty$. If $G(t)\in(0,+\infty)$ for some $t\geq T$ and $G({h}^{-1}(r))$ is linear with respect to $r\in[0,\int_{T}^{+\infty}c(s)e^{-s}ds)$,
 then there is a unique holomorphic $(n,0)$ form $F$ on $M$ satisfying $(F-f)\in H^{0}(Z_0,(\mathcal{O}(K_{M})\otimes\mathcal F)|_{Z_0})$ and $G(t;c)=\int_{\{\psi<-t\}}|F|^2e^{-\varphi}c(-\psi)$ for any $t\geq T$. Furthermore,
\begin{equation}
	\label{eq:20210412b}
	\int_{\{-t_1\leq\psi<-t_2\}}|F|^2e^{-\varphi}a(-\psi)=\frac{G(T_1;c)}{\int_{T_1}^{+\infty}c(t)e^{-t}dt}\int_{t_2}^{t_1} a(t)e^{-t}dt
\end{equation}
for any nonnegative measurable function $a$ on $(T,+\infty)$, where $+\infty\geq t_1>t_2\geq T$ and $T_1>T$.

Especially, if $\mathcal H^2(\tilde{c},t_0)\subset\mathcal H^2(c,t_0)$ for some $t_0\geq T$, where $\tilde{c}$ is a nonnegative measurable function on $(T,+\infty)$, we have
\begin{equation}
	\label{eq:20210412a}
	G(t_0;\tilde{c})=\int_{\{\psi<-t_0\}}|F|^2e^{-\varphi}\tilde{c}(-\psi)=\frac{G(T_1;c)}{\int_{T_1}^{+\infty}c(s)e^{-s}ds}\int_{t_0}^{+\infty} \tilde{c}(s)e^{-s}ds.\end{equation}	
\end{Corollary}

The following lemma is a characterization of $G(t)= 0$, where $t\geq T$.

\begin{Lemma}[\cite{GMY-concavity2}]
The following two statements are equivalent:
\par
$(1)$ $(f) \in
H^0(Z_0,(\mathcal{O} (K_M) \otimes \mathcal{F})|_{Z_0} )$.
\par
$(2)$ $G(t) = 0$.
\label{l:G equal to 0}
\end{Lemma}

\begin{Lemma}[\cite{GMY-concavity2}]
\label{lem:A}Let $c\in\mathcal{P}_{T,M}$ satisfying $\int_T^{+\infty}c(s)e^{-s}ds<+\infty$.
Assume that $G(t)<+\infty$ for some $t\in[T,+\infty)$.
Then there exists a unique holomorphic $(n,0)$ form $F_{t}$ on
$\{\psi<-t\}$ satisfying $(F_{t}-f)\in H^{0}(Z_0,(\mathcal{O}(K_{M})\otimes\mathcal{F})|_{Z_0})$ and $\int_{\{\psi<-t\}}|F_{t}|^{2}e^{-\varphi}c(-\psi)=G(t)$.
Furthermore,
for any holomorphic $(n,0)$ form $\hat{F}$ on $\{\psi<-t\}$ satisfying $(\hat{F}-f)\in H^{0}(Z_0,(\mathcal{O}(K_{M})\otimes\mathcal{F})|_{Z_0})$ and
$\int_{\{\psi<-t\}}|\hat{F}|^{2}e^{-\varphi}c(-\psi)<+\infty$,
we have the following equality
\begin{equation}
\label{equ:20170913e}
\begin{split}
&\int_{\{\psi<-t\}}|F_{t}|^{2}e^{-\varphi}c(-\psi)+\int_{\{\psi<-t\}}|\hat{F}-F_{t}|^{2}e^{-\varphi}c(-\psi)
\\=&
\int_{\{\psi<-t\}}|\hat{F}|^{2}e^{-\varphi}c(-\psi).
\end{split}
\end{equation}
\end{Lemma}

The following result will be used in the proof of Theorem \ref{thm:exten-fibra-single}.
\begin{Lemma}
	\label{l:linear2}
	Let $c\in\mathcal{P}_{T,M}$ satisfying $\int_T^{+\infty}c(s)e^{-s}ds<+\infty$. Assume $G(t)\in(0,+\infty)$ for some $t\geq T$ and $G({h}^{-1}(r))$ is linear with respect to $r\in[0,\int_{T}^{+\infty}c(s)e^{-s}ds)$. Let $\tilde{c}$ be a nonnegative function on $(T,+\infty)$, and let $t_0\ge T$. If there is a holomorphic $(n,0)$ form $\tilde{F}\in\mathcal{H}^2(\tilde{c},t_0)$ such that
	$$G(t_0;\tilde{c})=\int_{\{\psi<-t_0\}}|\tilde F|^2e^{-\varphi}\tilde{c}(-\psi)$$
	and $\tilde{F}\in\mathcal{H}^2(c,t_0)$, then we have
	$$G(t_0;\tilde{c})=\int_{\{\psi<-t_0\}}|F|^2e^{-\varphi}\tilde{c}(-\psi)=\frac{G(T_1;c)}{\int_{T_1}^{+\infty}c(s)e^{-s}ds}\int_{t_0}^{+\infty}\tilde{c}(s)e^{-s}ds,$$
	where $T_1>T$.
	\end{Lemma}
\begin{proof}
Using Corollary \ref{c:linear}, we know there is a unique holomorphic $(n,0)$ form $F$ on $M$ satisfying $(F-f)\in H^{0}(Z_0,(\mathcal{O}(K_{M})\otimes\mathcal F)|_{Z_0})$ and $G(t)=\int_{\{\psi<-t\}}|F|^2e^{-\varphi}c(-\psi)=\frac{G(T_1)}{\int_{T_1}^{+\infty}c(s)e^{-s}ds}\int_{t}^{+\infty}c(s)e^{-s}ds$ for any $t\geq T$.  It follows from the dominated convergence theorem  that
\begin{equation}
	\label{eq:20210412g}
	\int_{\{z\in M:-\psi(z)\in N\}}|F|^2e^{-\varphi}=0
\end{equation}
holds for any $N\subset\subset(T,+\infty)$ satisfying $\mu(N)=0$, where $\mu$ is the Lebesgue measure on $\mathbb{R}$.
As $\tilde{F}\in\mathcal H^2(c,t_0)$, It follows from Lemma \ref{lem:A} that
\begin{displaymath}
\begin{split}
\int_{\{\psi<-t\}}|\tilde{F}|^2e^{-\varphi}c(-\psi)=&\int_{\{\psi<-t\}}|F|^2e^{-\varphi}c(-\psi)\\
	&+\int_{\{\psi<-t\}}|\tilde{F}-F|^2e^{-\varphi}c(-\psi)	
\end{split}	
\end{displaymath}
for any $t\geq t_0$, then
\begin{equation}
	\label{eq:20210413a}
	\begin{split}
	\int_{\{-t_3\leq\psi<-t_4\}}|\tilde{F}|^2e^{-\varphi}c(-\psi)=&\int_{\{-t_3\leq\psi<-t_4\}}|F|^2e^{-\varphi}c(-\psi)\\
	&+\int_{\{-t_3\leq\psi<-t_4\}}|\tilde{F}-F|^2e^{-\varphi}c(-\psi)
		\end{split}
\end{equation}
holds for any $t_3>t_4\geq t_0$.
It follows from the dominated convergence theorem, equality \eqref{eq:20210412g}, equality \eqref{eq:20210413a} and $c(t)>0$ for any $t>T$, that
\begin{equation}
	\label{eq:20210413d}
	\int_{\{z\in M:-\psi(z)=t\}}|\tilde{F}|^2e^{-\varphi}=\int_{\{z\in M:-\psi(z)=t\}}|\tilde{F}-F|^2e^{-\varphi}
\end{equation}
holds for any $t>t_0$.

Choosing any closed interval $[t_4',t_3']\subset (t_0,+\infty)\subset(T,+\infty)$. Note that  $c(t)$ is uniformly continuous  and have positive lower bound and upper bound on $[t_4',t_3']\backslash U_k$, where $\{U_k\}_{k\in\mathbb{Z}_{\ge1}}$ is a decreasing sequence of open subsets of $(T,+\infty)$, such that $c$ is continuous on $(T,+\infty)\backslash U_k$ and $\lim_{k\rightarrow+\infty}\mu(U_k)=0$ (As $c(t)e^{-t}$ is decreasing, $\{U_k\}_{k\in\mathbb{Z}_{\ge1}}$ exists). Take $N=\cap_{k=1}^{+\infty}U_k$. Note that
\begin{equation}
	\label{eq:20210413b}
	\begin{split}
		&\int_{\{-t_3'\leq\psi<-t_4'\}}|\tilde{F}|^2e^{-\varphi}\\
		=&\lim_{n\rightarrow+\infty}\sum_{i=0}^{n-1}\int_{\{z\in M:-\psi(z)\in I_{n,i}\backslash U_k\}}|\tilde{F}|^2e^{-\varphi}+\int_{\{z\in M:-\psi(z)\in (t_4',t_3']\cap U_k\}}|\tilde{F}|^2e^{-\varphi}\\
		\leq&\limsup_{n\rightarrow+\infty}\sum_{i=0}^{n-1}\frac{1}{\inf_{I_{n,i}\backslash U_k}c(t)}\int_{\{z\in M:-\psi(z)\in I_{n,i}\backslash U_k\}}|\tilde{F}|^2e^{-\varphi}c(-\psi)\\
		&+\int_{\{z\in M:-\psi(z)\in (t_4',t_3']\cap U_k\}}|\tilde{F}|^2e^{-\varphi},
	\end{split}
\end{equation}
where $I_{n,i}=(t_4'-(i+1)\alpha_{n},t_3'-i\alpha_{n}]$ and $\alpha_n=\frac{t_3'-t_4'}{n}$.
It follows from equality \eqref{eq:20210412g}, equality \eqref{eq:20210413a}, equality \eqref{eq:20210413d}  and the dominated convergence theorem that
\begin{equation}
	\label{eq:2106e}
	\begin{split}
		&\int_{\{z\in M:-\psi(z)\in I_{n,i}\backslash U_k\}}|\tilde{F}|^2e^{-\varphi}c(-\psi)\\
		=&\int_{\{z\in M:-\psi(z)\in I_{n,i}\backslash U_k)\}}|F|^2e^{-\varphi}c(-\psi)+\int_{\{z\in M:-\psi(z)\in I_{n,i}\backslash U_k)\}}|\tilde{F}-F|^2e^{-\varphi}c(-\psi).
	\end{split}
\end{equation}
As $c(t)$ is uniformly continuous  and have positive lower bound and upper bound on $[t_4',t_3']\backslash U_k$, following from equality \eqref{eq:2106e}, we have
\begin{equation}
	\label{eq:2106f}\begin{split}
		&\limsup_{n\rightarrow+\infty}\sum_{i=0}^{n-1}\frac{1}{\inf_{I_{n,i}\backslash U_k}c(t)}\int_{\{z\in M:-\psi(z)\in I_{n,i}\backslash U_k\}}|\tilde{F}|^2e^{-\varphi}c(-\psi)\\
		=&\limsup_{n\rightarrow+\infty}\sum_{i=0}^{n-1}\frac{1}{\inf_{I_{n,i}\backslash U_k}c(t)}\bigg(\int_{\{z\in M:-\psi(z)\in I_{n,i}\backslash U_k)\}}|F|^2e^{-\varphi}c(-\psi)\\
		&+\int_{\{z\in M:-\psi(z)\in I_{n,i}\backslash U_k)\}}|\tilde{F}-F|^2e^{-\varphi}c(-\psi)\bigg)\\
		\leq&\limsup_{n\rightarrow+\infty}\sum_{i=0}^{n-1}\frac{\sup_{I_{n,i}\backslash U_k}c(t)}{\inf_{I_{n,i}\backslash U_k}c(t)}\bigg(\int_{\{z\in M:-\psi(z)\in I_{n,i}\backslash U_k\}}|F|^2e^{-\varphi}\\
		&+\int_{\{z\in M:-\psi(z)\in I_{n,i}\backslash U_k\}}|\tilde{F}-F|^2e^{-\varphi}\bigg)\\
		=&\int_{\{z\in M:-\psi(z)\in (t_4',t_3']\backslash U_k\}}|F|^2e^{-\varphi}+\int_{\{z\in M:-\psi(z)\in (t_4',t_3']\backslash U_k\}}|\tilde{F}-F|^2e^{-\varphi}.
	\end{split}
\end{equation}
It follows from inequality \eqref{eq:20210413b} and \eqref{eq:2106f} that
\begin{equation}
	\label{eq:2106g}
	\begin{split}
		&\int_{\{-t_3'\leq\psi<-t_4'\}}|\tilde{F}|^2e^{-\varphi}\\
		\leq&\int_{\{z\in M:-\psi(z)\in (t_4',t_3']\backslash U_k\}}|F|^2e^{-\varphi}+\int_{\{z\in M:-\psi(z)\in (t_4',t_3']\backslash U_k\}}|\tilde{F}-F|^2e^{-\varphi}\\&+\int_{\{z\in M:-\psi(z)\in (t_4',t_3']\cap U_k\}}|\tilde{F}|^2e^{-\varphi}.
	\end{split}
\end{equation}
It follows from $\tilde{F}\in\mathcal{H}^2(c,t_0)$ that $\int_{\{-t_3'\leq\psi<-t_4'\}}|\tilde{F}|^2e^{-\varphi}<+\infty$. Letting $k\rightarrow+\infty$, it follows from equality \eqref{eq:20210412g}, inequality \eqref{eq:2106g} and the dominated convergence theorem that
\begin{equation}
\label{eq:20210414b}
	\begin{split}
		\int_{\{-t_3'\leq\psi<-t_4'\}}|\tilde{F}|^2e^{-\varphi}\leq&\int_{\{-t_3'\leq\psi<-t_4'\}}|F|^2e^{-\varphi}\\
		&+\int_{\{z\in M:-\psi(z)\in (t_4',t_3']\backslash N\}}|\tilde{F}-F|^2e^{-\varphi}\\&+\int_{\{z\in M:-\psi(z)\in (t_4',t_3']\cap N\}}|\tilde{F}|^2e^{-\varphi}.
	\end{split}
\end{equation}
 Following from a similar discussion, we can obtain that
\begin{displaymath}
	\begin{split}
		\int_{\{-t_3'\leq\psi<-t_4'\}}|\tilde{F}|^2e^{-\varphi}\geq&\int_{\{-t_3'\leq\psi<-t_4'\}}|F|^2e^{-\varphi}\\
		&+\int_{\{z\in M:-\psi(z)\in (t_4',t_3']\backslash N\}}|\tilde{F}-F|^2e^{-\varphi}\\&+\int_{\{z\in M:-\psi(z)\in (t_4',t_3']\cap N\}}|\tilde{F}|^2e^{-\varphi},
	\end{split}
\end{displaymath}
then combining inequality \eqref{eq:20210414b} we have
\begin{equation}
	\label{eq:20210413c}
	\begin{split}
		\int_{\{-t_3'\leq\psi<-t_4'\}}|\tilde{F}|^2e^{-\varphi}=&\int_{\{-t_3'\leq\psi<-t_4'\}}|F|^2e^{-\varphi}\\
		&+\int_{\{z\in M:-\psi(z)\in (t_4',t_3']\backslash N\}}|\tilde{F}-F|^2e^{-\varphi}\\&+\int_{\{z\in M:-\psi(z)\in (t_4',t_3']\cap N\}}|\tilde{F}|^2e^{-\varphi}.
	\end{split}\end{equation}
Using equality \eqref{eq:20210412g}, equality \eqref{eq:20210413d}, equality \eqref{eq:20210413c} and  the monotone convergence theorem, we have
\begin{equation*}
	\begin{split}
		\int_{\{z\in M:-\psi(z)\in U\}}|\tilde{F}|^2e^{-\varphi}=&\int_{\{z\in M:-\psi(z)\in U\}}|F|^2e^{-\varphi}\\
		&+\int_{\{z\in M:-\psi(z)\in  U\backslash N\}}|\tilde{F}-F|^2e^{-\varphi}\\&+\int_{\{z\in M:-\psi(z)\in U\cap N\}}|\tilde{F}|^2e^{-\varphi}
	\end{split}
\end{equation*}
holds for any open set $U\subset\subset(t_0,+\infty)$, and
\begin{equation*}
	\begin{split}
		\int_{\{z\in M:-\psi(z)\in V\}}|\tilde{F}|^2e^{-\varphi}=&\int_{\{z\in M:-\psi(z)\in V\}}|F|^2e^{-\varphi}\\
		&+\int_{\{z\in M:-\psi(z)\in  V\backslash N\}}|\tilde{F}-F|^2e^{-\varphi}\\&+\int_{\{z\in M:-\psi(z)\in V\cap N\}}|\tilde{F}|^2e^{-\varphi}
	\end{split}
\end{equation*}
holds for any compact set $V\subset(t_0,+\infty)$. For any  measurable set $E\subset\subset(t_0,+\infty)$, there exists a sequence of compact sets $\{V_l\}$, such that $V_l\subset V_{l+1}\subset E$ for any $l$ and $\lim_{l\rightarrow+\infty}\mu(V_l)=\mu(E)$, hence
\begin{equation}
\label{eq:2021525c}
	\begin{split}
		\int_{\{\psi<-t_0\}}|\tilde{F}|^2e^{-\varphi}\mathbb I_{E}(-\psi)\geq&\lim_{l\rightarrow+\infty}\int_{\{\psi<-t_0\}}|\tilde{F}|^2e^{-\varphi}\mathbb I_{V_{l}}(-\psi)
		\\\geq&\lim_{l\rightarrow+\infty}\int_{\{\psi<-t_0\}}|F|^2e^{-\varphi}\mathbb I_{V_l}(-\psi)
		\\=&\int_{\{\psi<-t_0\}}|F|^2e^{-\varphi}\mathbb I_{E}(-\psi).
	\end{split}
\end{equation}

It is clear that for any $t>t_0$, there exists a sequence of functions $\left\{\sum_{j=1}^{n_i}\mathbb I_{E_{ij}}\right\}_{i=1}^{+\infty}$ defined on $(t,+\infty)$, satisfying $E_{ij}\subset\subset(t,+\infty)$, $\sum_{j=1}^{n_{i+1}}\mathbb I_{E_{i+1j}}(s)\geq\sum_{j=1}^{n_i}\mathbb I_{E_{ij}}(s)$, and $\lim_{i\rightarrow+\infty}\sum_{j=1}^{n_i}\mathbb I_{E_{ij}}(s)=\tilde{c}(s)$  for any $s>t$.
Combining the monotone convergence theorem and inequality \eqref{eq:2021525c}, we have
\begin{equation*}
		\int_{\{\psi<-t_0\}}|\tilde{F}|^2e^{-\varphi}\tilde{c}(-\psi)\geq\int_{\{\psi<-t_0\}}|F|^2e^{-\varphi}\tilde{c}(-\psi).
\end{equation*}
By the definition of $G(t_0,\tilde{c})$, we have $G(t_0,\tilde{c})=\int_{\{\psi<-t_0\}}|F|^2e^{-\varphi}\tilde{c}(-\psi)$. Thus, Lemma \ref{l:linear2} holds.
\end{proof}

Let $\Omega_j$  be an open Riemann surface, which admits a nontrivial Green function $G_{\Omega_j}$ for any  $1\le j\le n$. Let $M=\prod_{1\le j\le n}\Omega_j$ be an $n-$dimensional complex manifold, and let $\pi_j$ be the natural projection from $M$ to $\Omega_j$. Let $K_M$ be the canonical (holomorphic) line bundle on $M$.
 Let $Z_j$ be a (closed) analytic subset of $\Omega_j$ for any $j\in\{1,\ldots,n\}$, and let $Z_0=\prod_{1\le j\le n}Z_j$. For any $j\in\{1,\ldots,n\}$, let $\varphi_j$ be a subharmonic function on $\Omega_j$ such that $\varphi_j(z)>-\infty$ for any $z\in Z_j$, and let $\varphi=\sum_{1\le j\le n}\pi_j^*(\varphi_j)$. Let $\psi$ be a plurisubharmonic function on $M$ such that $\psi(z)=-\infty$ for any $z\in Z_0$ and $\psi$ is continuous on $M\backslash Z_0$.
Let $c$ be a positive function on $(0,+\infty)$ such that $\int_{0}^{+\infty}c(t)e^{-t}dt<+\infty$ and $c(t)e^{-t}$ is decreasing on $(0,+\infty)$. Let $\mathcal{F}_{z}=\mathcal{I}(\psi)_z$ for any $z\in Z_0$.

In the following, we recall some  results about the concavity of $G(h^{-1}(r))$ degenerating to linearity.

Let $Z_0=\{z_0\}=\{(z_1,\ldots,z_n)\}\subset M$.
Let $\psi=\max_{1\le j\le n}\left\{2p_j\pi_j^{*}(G_{\Omega_j}(\cdot,z_j))\right\}$, where $p_j$ is positive real number.
Let $w_j$ be a local coordinate on a neighborhood $V_{z_j}$ of $z_j\in\Omega_j$ satisfying $w_j(z_j)=0$. Denote that $V_0:=\prod_{1\le j\le n}V_{z_j}$, and $w:=(w_1,\ldots,w_n)$ is a local coordinate on $V_0$ of $z_0\in M$.
Let $f$ be a holomorphic $(n,0)$ form on $V_0$. Denote that $E:=\left\{(\alpha_1,\ldots,\alpha_n):\sum_{1\le j\le n}\frac{\alpha_j+1}{p_j}=1\,\&\,\alpha_j\in\mathbb{Z}_{\ge0}\right\}$.

We recall a characterization of the concavity of $G(h^{-1}(r))$ degenerating to linearity for the case $Z_0$ is a single point set as follows.

\begin{Theorem}[\cite{GY-concavity4}]
	\label{thm:linear-2d}
	Assume that $G(0)\in(0,+\infty)$.  $G(h^{-1}(r))$ is linear with respect to $r\in(0,\int_{0}^{+\infty}c(t)e^{-t}dt]$  if and only if the  following statements hold:
	
	$(1)$ $f=\left(\sum_{\alpha\in E}d_{\alpha}w^{\alpha}+g_0\right)dw_1\wedge\ldots\wedge dw_n$ on $V_0$, where  $d_{\alpha}\in\mathbb{C}$ such that $\sum_{\alpha\in E}|d_{\alpha}|\not=0$ and $g_0$ is a holomorphic function on $V_0$ such that $(g_0,z_0)\in\mathcal{I}(\psi)_{z_0}$;
	
	$(2)$ $\varphi_j=2\log|g_j|+2u_j$, where $g_j$ is a holomorphic function on $\Omega_j$ such that $g_j(z_j)\not=0$ and $u_j$ is a harmonic function on $\Omega_j$ for any $1\le j\le n$;

    $(3)$ $\chi_{j,z_j}^{\alpha_j+1}=\chi_{j,-u_j}$ for any $j\in\{1,2,...,n\}$ and $\alpha\in E$ satisfying $d_{\alpha}\not=0$, $\chi_{j,z_j}$ and $\chi_{j,-u_j}$ are the characters associated to functions $G_{\Omega_j}(\cdot,z_j)$ and $-u_j$ respectively.
\end{Theorem}

Let $c_j(z)$ be the logarithmic capacity (see \cite{S-O69}) on $\Omega_j$, which is locally defined by
$$c_j(z_j):=\exp\lim_{z\rightarrow z_j}(G_{\Omega_j}(z,z_j)-\log|w_j(z)|).$$
\begin{Remark}[\cite{GY-concavity4}]
	\label{r:1.1}When the three statements in Theorem \ref{thm:linear-2d} hold,
$$\sum_{\alpha\in E}\tilde{d}_{\alpha}\wedge_{1\le j\le n}\pi_j^*\left(g_j(P_j)_*\left(f_{u_j}f_{z_j}^{\alpha_j}df_{z_j}\right)\right)$$
 is the unique holomorphic $(n,0)$ form $F$ on $M$ such that $(F-f,z_0)\in(\mathcal{O}(K_{M}))_{z_0}\otimes\mathcal{I}(\psi)_{z_0}$ and
	$$G(t)=\int_{\{\psi<-t\}}|F|^2e^{-\varphi}c(-\psi)=\left(\int_t^{+\infty}c(s)e^{-s}ds\right)\sum_{\alpha\in E}\frac{|d_{\alpha}|^2(2\pi)^ne^{-\varphi(z_{0})}}{\prod_{1\le j\le n}(\alpha_j+1)c_{j}(z_j)^{2\alpha_{j}+2}}$$
	 for any $t\ge0$, where $P_j:\Delta \rightarrow\Omega_j$ is the universal covering, $f_{u_j}$ is a holomorphic function on $\Delta$ such that $|f_{u_j}|=P_j^*(e^{u_j})$ for any $j\in\{1,\ldots,n\}$, $f_{z_j}$ is a holomorphic function on $\Delta$ such that $|f_{z_j}|=P_j^*\left(e^{G_{\Omega_j}(\cdot,z_j)}\right)$ for any $j\in\{1,\ldots,n\}$ and $\tilde{d}_{\alpha}$ is a constant such that $\tilde{d}_{\alpha}=\lim_{z\rightarrow z_0}\frac{d_{\alpha}w^{\alpha}dw_1\wedge\ldots\wedge dw_n}{\wedge_{1\le j\le n}\pi_j^*\left(g_j(P_j)_*\left(f_{u_j}f_{z_j}^{\alpha_j}df_{z_j}\right)\right)}$ for any $\alpha\in E$.
\end{Remark}

 Let $Z_j=\{z_{j,1},\ldots,z_{j,m_j}\}\subset\Omega_j$ for any  $j\in\{1,\ldots,n\}$, where $m_j$ is a positive integer.
Let $\psi=\max_{1\le j\le n}\left\{\pi_j^*\left(2\sum_{1\le k\le m_j}p_{j,k}G_{\Omega_j}(\cdot,z_{j,k})\right)\right\}$.

Let $w_{j,k}$ be a local coordinate on a neighborhood $V_{z_{j,k}}\Subset\Omega_{j}$ of $z_{j,k}\in\Omega_j$ satisfying $w_{j,k}(z_{j,k})=0$ for any $j\in\{1,\ldots,n\}$ and $k\in\{1,\ldots,m_j\}$, where $V_{z_{j,k}}\cap V_{z_{j,k'}}=\emptyset$ for any $j$ and $k\not=k'$. Denote that $I_1:=\{(\beta_1,\ldots,\beta_n):1\le \beta_j\le m_j$ for any $j\in\{1,\ldots,n\}\}$, $V_{\beta}:=\prod_{1\le j\le n}V_{z_{j,\beta_j}}$ for any $\beta=(\beta_1,\ldots,\beta_n)\in I_1$ and $w_{\beta}:=(w_{1,\beta_1},\ldots,w_{n,\beta_n})$ is a local coordinate on $V_{\beta}$ of $z_{\beta}:=(z_{1,\beta_1},\ldots,z_{n,\beta_n})\in M$.
Let $f$ be a holomorphic $(n,0)$ form on $\cup_{\beta\in I_1}V_{\beta}$ such that $f=w_{\beta^*}^{\alpha_{\beta_*}}dw_{1,1}\wedge\ldots\wedge dw_{n,1}$ on $V_{\beta^*}$, where $\beta^*=(1,\ldots,1)\in I_1$.

We recall a characterization of the concavity of $G(h^{-1}(r))$ degenerating to linearity for the case $Z_j$ is a  set of finite points as follows.

\begin{Theorem}[\cite{GY-concavity4}]
	\label{thm:prod-finite-point}Assume that $G(0)\in(0,+\infty)$.  $G(h^{-1}(r))$ is linear with respect to $r\in(0,\int_0^{+\infty} c(s)e^{-s}ds]$ if and only if the following statements hold:

	$(1)$ $\varphi_j=2\log|g_j|+2u_j$ for any $j\in\{1,\ldots,n\}$, where $u_j$ is a harmonic function on $\Omega_j$ and $g_j$ is a holomorphic function on $\Omega_j$ satisfying $g_j(z_{j,k})\not=0$ for any $k\in\{1,\ldots,m_j\}$;
	
	$(2)$ There exists a nonnegative integer $\gamma_{j,k}$ for any $j\in\{1,\ldots,n\}$ and $k\in\{1,\ldots,m_j\}$, which satisfies that $\prod_{1\le k\leq m_j}\chi_{j,z_{j,k}}^{\gamma_{j,k}+1}=\chi_{j,-u_j}$ and $\sum_{1\le j\le n}\frac{\gamma_{j,\beta_j}+1}{p_{j,\beta_j}}=1$ for any $\beta\in I_1$, where $\chi_{j,z_{j,k}}$ and $\chi_{j,-u_j}$ are the characters associated to $G_{\Omega_j}(\cdot,z_{j,k})$ and $-u_j$ respectively;
	
	$(3)$ $f=\left(c_{\beta}\prod_{1\le j\le n}w_{j,\beta_j}^{\gamma_{j,\beta_j}}+g_{\beta}\right)dw_{1,\beta_1}\wedge\ldots\wedge dw_{n,\beta_n}$ on $V_{\beta}$ for any $\beta\in I_1$, where $c_{\beta}$ is a constant and $g_{\beta}$ is a holomorphic function on $V_{\beta}$ such that $(g_{\beta},z_{\beta})\in\mathcal{I}(\psi)_{z_{\beta}}$;
	
	$(4)$ $\lim_{z\rightarrow z_{\beta}}\frac{c_{\beta}\prod_{1\le j\le n}w_{j,\beta_j}^{\gamma_{j,\beta_j}}dw_{1,\beta_1}\wedge\ldots\wedge dw_{n,\beta_n}}{{\wedge}_{1\le j\le n}\pi_{j}^*\left(g_j(P_{j})_*\left(f_{u_j}\left(\prod_{1\le k\le m_j}f_{z_{j,k}}^{\gamma_{j,k}+1}\right)\left(\sum_{1\le k\le m_j}p_{j,k}\frac{df_{z_{j,k}}}{f_{z_{j,k}}}\right)\right)\right)}=c_0$ for any $\beta\in I_1$, where $P_j:\Delta\rightarrow\Omega_j$ is the universal covering, $c_0\in\mathbb{C}\backslash\{0\}$ is a constant independent of $\beta$, $f_{u_j}$ is a holomorphic function $\Delta$ such that $|f_{u_j}|=P_j^*\left(e^{u_j}\right)$ and $f_{z_{j,k}}$ is a holomorphic function on $\Delta$ such that $|f_{z_{j,k}}|=P_j^*(e^{G_{\Omega_j}(\cdot,z_{j,k})})$ for any $j\in\{1,\ldots,n\}$ and $k\in\{1,\ldots,m_j\}$.
\end{Theorem}

Denote that
\begin{equation*}
c_{j,k}:=\exp\lim_{z\rightarrow z_{j,k}}\left(\frac{\sum_{1\le k_1\le m_j}p_{j,k_1}G_{\Omega_j}(z,z_{j,k_1})}{p_{j,k}}-\log|w_{j,k}(z)|\right)
\end{equation*}
 for any $j\in\{1,\ldots,n\}$ and $k\in\{1,\ldots,m_j\}$.
\begin{Remark}[\cite{GY-concavity4}]
	\label{r:1.2}When the four statements in Theorem \ref{thm:prod-finite-point} hold,
$$c_0{\wedge}_{1\le j\le n}\pi_{j}^*\left(g_j(P_{j})_*\left(f_{u_j}\left(\prod_{1\le k\le m_j}f_{z_{j,k}}^{\gamma_{j,k}+1}\right)\left(\sum_{1\le k\le m_j}p_{j,k}\frac{df_{z_{j,k}}}{f_{z_{j,k}}}\right)\right)\right)$$
 is the unique holomorphic $(n,0)$ form $F$ on $M$ such that $(F-f,z_\beta)\in(\mathcal{O}(K_{M}))_{z_\beta}\otimes\mathcal{I}(\psi)_{z_\beta}$ for any $\beta\in I_1$ and
	$$G(t)=\int_{\{\psi<-t\}}|F|^2e^{-\varphi}c(-\psi)=\left(\int_{t}^{+\infty}c(s)e^{-s}ds\right)\sum_{\beta\in I_1}\frac{|c_{\beta}|^2(2\pi)^ne^{-\varphi(z_{\beta})}}{\prod_{1\le j\le n}(\gamma_{j,\beta_j}+1)c_{j,\beta_j}^{2\gamma_{j,\beta_j}+2}}$$
	 for any $t\ge0$.
\end{Remark}

 Let ${Z}_j=\{z_{j,k}:1\le k<\tilde m_j\}$ be a discrete subset of $\Omega_j$ for any  $j\in\{1,\ldots,n\}$, where $\tilde{m}_j\in\mathbb{Z}_{\ge2}\cup\{+\infty\}$.
Let $p_{j,k}$ be a positive number for any $1\le j\le n$ and $1\le k<\tilde m_j$ such that $\sum_{1\le k<\tilde{m}_j}p_{j,k}G_{\Omega_j}(\cdot,z_{j,k})\not\equiv-\infty$ for any $j$.
Let $\psi=\max_{1\le j\le n}\left\{\pi_j^*\left(2\sum_{1\le k<\tilde{m}_j}p_{j,k}G_{\Omega_j}(\cdot,z_{j,k})\right)\right\}$. Assume that $\limsup_{t\rightarrow+\infty}c(t)<+\infty$.

Let $w_{j,k}$ be a local coordinate on a neighborhood $V_{z_{j,k}}\Subset\Omega_{j}$ of $z_{j,k}\in\Omega_j$ satisfying $w_{j,k}(z_{j,k})=0$ for any $j\in\{1,\ldots,n\}$ and $1\le k<\tilde{m}_j$, where $V_{z_{j,k}}\cap V_{z_{j,k'}}=\emptyset$ for any $j$ and $k\not=k'$. Denote that $\tilde I_1:=\{(\beta_1,\ldots,\beta_n):1\le \beta_j< \tilde m_j$ for any $j\in\{1,\ldots,n\}\}$, $V_{\beta}:=\prod_{1\le j\le n}V_{z_{j,\beta_j}}$ for any $\beta=(\beta_1,\ldots,\beta_n)\in\tilde I_1$ and $w_{\beta}:=(w_{1,\beta_1},\ldots,w_{n,\beta_n})$ is a local coordinate on $V_{\beta}$ of $z_{\beta}:=(z_{1,\beta_1},\ldots,z_{n,\beta_n})\in M$.
Let $f$ be a holomorphic $(n,0)$ form on $\cup_{\beta\in \tilde I_1}V_{\beta}$ such that $f=w_{\beta^*}^{\alpha_{\beta_*}}dw_{1,1}\ldots\wedge dw_{n,1}$ on $V_{\beta^*}$, where $\beta^*=(1,\ldots,1)\in \tilde I_1$.

We recall that $G(h^{-1}(r))$ is not linear when there exists $j_0\in\{1,\ldots,n\}$ such that $\tilde m_{j_0}=+\infty$ as follows.

\begin{Theorem}[\cite{GY-concavity4}]
	\label{thm:prod-infinite-point}If $G(0)\in(0,+\infty)$ and there exists $j_0\in\{1,\ldots,n\}$ such that $\tilde m_{j_0}=+\infty$, then $G(h^{-1}(r))$ is not linear with respect to $r\in(0,\int_0^{+\infty} c(s)e^{-s}ds]$.
	\end{Theorem}

\subsection{Some basic properties of the Green functions}
\

In this Section, we recall some basic properties of the Green functions. Let $\Omega$ be an open Riemann surface, which admits a nontrivial Green function $G_{\Omega}$, and let $z_0\in\Omega$.

\begin{Lemma}[see \cite{S-O69}, see also \cite{Tsuji}] 	\label{l:green-sup}Let $w$ be a local coordinate on a neighborhood of $z_0$ satisfying $w(z_0)=0$.  $G_{\Omega}(z,z_0)=\sup_{v\in\Delta_{\Omega}^*(z_0)}v(z)$, where $\Delta_{\Omega}^*(z_0)$ is the set of negative subharmonic function on $\Omega$ such that $v-\log|w|$ has a locally finite upper bound near $z_0$. Moreover, $G_{\Omega}(\cdot,z_0)$ is harmonic on $\Omega\backslash\{z_0\}$ and $G_{\Omega}(\cdot,z_0)-\log|w|$ is harmonic near $z_0$.
\end{Lemma}

\begin{Lemma}[see \cite{GY-concavity3}]
	\label{l:green-sup2}Let $K=\{z_j:j\in\mathbb{Z}_{\ge1}\,\&\,j<\gamma \}$ be a discrete subset of $\Omega$, where $\gamma\in\mathbb{Z}_{>1}\cup\{+\infty\}$. Let $\psi$ be a negative subharmonic function on $\Omega$ such that $\frac{1}{2}v(dd^c\psi,z_j)\ge p_j$ for any $j$, where $p_j>0$ is a constant. Then $2\sum_{1\le j< \gamma}p_jG_{\Omega}(\cdot,z_j)$ is a subharmonic function on $\Omega$ satisfying that $2\sum_{1\le j<\gamma }p_jG_{\Omega}(\cdot,z_j)\ge\psi$ and $2\sum_{1\le j<\gamma }p_jG_{\Omega}(\cdot,z_j)$ is harmonic on $\Omega\backslash K$.
\end{Lemma}

\begin{Lemma}[see \cite{GY-concavity}]\label{l:G-compact}
For any  open neighborhood $U$ of $z_0$, there exists $t>0$ such that $\{G_{\Omega}(z,z_0)<-t\}$ is a relatively compact subset of $U$.
\end{Lemma}

\begin{Lemma}[see \cite{GY-concavity3}]
	\label{l:green-approx} There exists a sequence of open Riemann surfaces $\{\Omega_l\}_{l\in\mathbb{Z}^+}$ such that $z_0\in\Omega_l\Subset\Omega_{l+1}\Subset\Omega$, $\cup_{l\in\mathbb{Z}^+}\Omega_l=\Omega$, $\Omega_l$ has a smooth boundary $\partial\Omega_l$ in $\Omega$  and $e^{G_{\Omega_l}(\cdot,z_0)}$ can be smoothly extended to a neighborhood of $\overline{\Omega_l}$ for any $l\in\mathbb{Z}^+$, where $G_{\Omega_l}$ is the Green function of $\Omega_l$. Moreover, $\{{G_{\Omega_l}}(\cdot,z_0)-G_{\Omega}(\cdot,z_0)\}$ is decreasingly convergent to $0$ on $\Omega$ with respect to $l$.
\end{Lemma}

Let $\Omega_j$ be an open Riemann surface for any $1\le j\le n$, which admits a a nontrivial Green function $G_{\Omega}$. Let $\{z_{j,k}:1\le k<\tilde m_j\}$ be a discrete subset of $\Omega_j$ for any $1\le j\le n$, where $\tilde m_j\in\mathbb{Z}_{\ge2}\cup\{+\infty\}$.
The following lemma will be used in the proof of the applications.

\begin{Lemma}[see \cite{GY-concavity4}]
	\label{l:psi=G}Let $\psi=\max_{1\le j\le n}\left\{\pi_j^*\left(2\sum_{1\le k<\tilde m_j}p_{j,k}G_{\Omega_j}(\cdot,z_{j,k})\right)\right\}$ be a plurisubharmonic function on $\prod_{1\le j\le n_1}\Omega_j$, where $\sum_{1\le k<\tilde m_j}p_{j,k}G_{\Omega_j}(\cdot,z_{j,k})\not\equiv-\infty$ for any $j\in\{1,..,n\}$.	 Let $\Psi\le0$ be a plurisubharmonic function on $\prod_{1\le j\le n_1}\Omega_j$, and denote that $\tilde\psi:=\psi+\Psi$. Let $l(t)$ be a positive Lebesgue measurable function on $(0,+\infty)$ satisfying that $l(t)$ is decreasing on $(0,+\infty)$ and $\int_0^{+\infty}l(t)dt<+\infty$. If $\Psi\not\equiv 0$ on $M$, there exists a Lebesgue measurable subset $V$ of $\prod_{1\le j\le n_1}\Omega_j$ such that  $l\left(-\tilde\psi(z)\right)<l(-\psi(z))$ for any  $z\in V$ and $\mu(V)>0$, where $\mu$ is the Lebesgue measure on $\prod_{1\le j\le n_1}\Omega_j$.
\end{Lemma}

\subsection{Some results related to $\max_{1\le j\le n}\{2p_j\log|w_j|\}$}
\

In this section, we recall some basic property related to $\max_{1\le j\le n}\{2p_j\log|w_j|\}$. In the following lemma,
we recall a closedness of the submodules of $\mathcal O_{\mathbb C^n,o}^q$.
\begin{Lemma}[see \cite{G-R}]
\label{l:closedness}
Let $N$ be a submodule of $\mathcal O_{\mathbb C^n,o}^q$, $1\leq q< +\infty$, let $f_j\in\mathcal O_{\mathbb C^n}(U)^q$ be a sequence of $q-$tuples holomorphic in an open neighborhood $U$ of the origin $o$. Assume that the $f_j$ converge uniformly in $U$ towards  a $q-$tuples $f\in\mathcal O_{\mathbb C^n}(U)^q$, assume furthermore that all germs $(f_{j},o)$ belong to $N$. Then $(f,o)\in N$.	
\end{Lemma}

Let $f=\sum_{\alpha\in\mathbb{Z}_{\ge0}^{n}}b_{\alpha}w^{\alpha}$ (Taylor expansion) be a holomorphic function  on $D=\{w\in\mathbb{C}^n:|w_j|<r_0$ for any $j\in\{1,\ldots,n\}\}$, where $r_0>0$. 	Let
$$\psi=\max_{1\le j\le n_1}\left\{2p_j\log|w_j|\right\}$$ be a plurisubharmonic function on $\mathbb{C}^n$, where $n_1\le n$ and $p_j>0$ is a constant for any $j\in\{1,\ldots,n_1\}$. We recall a characterization of  $\mathcal{I}(\psi)_o$, where $o$ is the origin in $\mathbb{C}^n$.
\begin{Lemma}[see \cite{guan-20}]\label{l:0}
$(f,o)\in\mathcal{I}(\psi)_{o}$ if and only if $\sum_{1\le j\le n_1}\frac{\alpha_j+1}{p_j}>1$ for any $\alpha\in\mathbb{Z}_{\ge0}^n$ satisfying $b_{\alpha}\not=0$.
\end{Lemma}
\begin{proof}
	For the convenience of the reader, we recall the proof.
	
Let $V=\left\{w\in\mathbb{C}^n:\max_{n_1+1\le j\le n}\{|w_j|\}<s\right\}$, where $s\in(0,r_0)$. There exists $r_1>0$ such that $\{\psi<\log r_1\}\cap V\Subset D$.	If $(f,o)\in\mathcal{I}(\psi)_{o}$,  we have
	\begin{equation}
		\label{eq:1125a}\int_{\{\psi<\log r_1\}\cap V}|f|^2e^{-\psi}d\lambda_n<+\infty,
	\end{equation}
	where $d\lambda_n$ is the Lebesgue measure on $\mathbb{C}^n$. Note that
	\begin{displaymath}
		\begin{split}
			&\int_{\{\psi<\log r_1\}\cap V}|f|^2e^{-\psi}d\lambda_n\\
			=&\lim_{\epsilon\rightarrow0+0}\int_{\left\{\epsilon<|w_1|<r_1^{\frac{1}{2p_1}}\right\}\cap\ldots\cap\left\{\epsilon<|w_{n_1}|<r_1^{\frac{1}{2p_{n_1}}}\right\}\cap V}|f|^2e^{-\psi}d\lambda_n\\
			=&\lim_{\epsilon\rightarrow0+0}\left(\sum_{\alpha\in\mathbb{Z}_{\ge0}^n}\int_{\left\{\epsilon<|w_1|<r_1^{\frac{1}{2p_1}}\right\}\cap\ldots\cap\left\{\epsilon<|w_{n_1}|<r_1^{\frac{1}{2p_{n_1}}}\right\}\cap V}|b_{\alpha}w^{\alpha}|^2e^{-\psi}d\lambda_n\right)\\
			=&\sum_{\alpha\in\mathbb{Z}_{\ge0}^n}|b_{\alpha}|^2\int_{\{\psi<\log r_1\}\cap V}|w^{\alpha}|^2e^{-\psi}d\lambda_n.
	   \end{split}
	\end{displaymath}
 Inequality \eqref{eq:1125a} implies that
	\begin{equation}
		\label{eq:1125b}\int_{\{\psi<\log r_1\}\cap V}|w^{\alpha}|^2e^{-\psi}d\lambda_n<+\infty
	\end{equation}
	for any $\alpha\in\mathbb{Z}_{\ge0}^n$ satisfying $b_{\alpha}\not=0$. Note that
	\begin{equation}	\label{eq:1125c}\begin{split}
	\int_{\{\psi<\log r_1\}\cap V}|w^{\alpha}|^2e^{-\psi}d\lambda_n=&\int_{\{\psi<\log r_1\}\cap V}|w^{\alpha}|^2\left(\int_{0}^{+\infty}\mathbb{I}_{\{l<e^{-\psi}\}}dl\right)d\lambda_n\\
		=&\int_{0}^{r_1}\left(\int_{\{\psi<\log r\}\cap V}|w^{\alpha}|^2d\lambda_n\right)r^{-2}dr\\
		&+\frac{1}{r_1}\int_{\{\psi<\log r_1\}\cap V}|w^{\alpha}|^2d\lambda_n
		\end{split}
	\end{equation}
	and
	\begin{equation}
		\label{eq:1125d}\begin{split}
			\int_{\{\psi<\log r\}\cap V}|w^{\alpha}|^2d\lambda_n=&\int_{\left\{|w_1|<r^{\frac{1}{2p_1}}\right\}\cap\ldots\cap\left\{|w_{n_1}|<r^{\frac{1}{2p_{n_1}}}\right\}\cap V}\bigg|\prod_{1\le j\le n}w_j^{\alpha_j}\bigg|^2d\lambda_n\\
			=&\pi^{n_1}\frac{r^{\sum_{1\le j\le n_1}\frac{\alpha_j+1}{p_j}}}{\prod_{1\le j\le n_1}(\alpha_j+1)}\int_{V}|w_{n_1+1}^{\alpha_{n_1+1}}\ldots w_n^{\alpha_n}|^2d\lambda_{n-n_1}.
		\end{split}
	\end{equation}
	It follows from inequality \eqref{eq:1125b}, equality \eqref{eq:1125c} and equality \eqref{eq:1125d} that
	$$\sum_{1\le j\le n_1}\frac{\alpha_j+1}{p_j}>1$$ for any $\alpha\in\mathbb{Z}_{\ge0}^n$ satisfying $b_{\alpha}\not=0$.
	
	If $\sum_{1\le j\le n_1}\frac{\alpha_j+1}{p_j}>1$ for any $\alpha\in\mathbb{Z}_{\ge0}^n$ satisfying $b_{\alpha}\not=0$, it follows from equality \eqref{eq:1125c} and equality \eqref{eq:1125d} that
	\begin{equation*}
	\int_{\{\psi<\log r_1\}}|w^{\alpha}|^2e^{-\psi}d\lambda_n<+\infty,
	\end{equation*}
	i.e. $(w^{\alpha},o)\in\mathcal{I}(\psi)_o$ for any $\alpha\in\mathbb{Z}_{\ge0}^n$ satisfying $b_{\alpha}\not=0$. Using Lemma \ref{l:closedness}, we have $(f,o)\in\mathcal{I}(\psi)_{o}$.	
	\end{proof}
	
	For any $y\in D'=\{y\in\mathbb{C}^{n-n_1}:|y_k|<r_0$ for $1\le k\le n-n_1\}$, denote that $f_y=f(\cdot,y)$ is a holomorphic function on $D''=\{x\in\mathbb{C}^{n_1}:|x_j|<r_0$ for any $j\in\{1,\ldots,n_1\}\}$. It follows from Lemma \ref{l:0} that the following lemma holds.
	
\begin{Lemma}
	\label{l:0'}
	$(f,(o_1,y))\in\mathcal{I}(\psi)_{(o_1,y)}$ for any $y\in D'$ if and only if $(f_y,o_1)\in \mathcal{I}(\psi)_{o_1}$ for any $y\in D'$, where $o_1$ is the origin in $\mathbb{C}^{n_1}$.
\end{Lemma}

The following lemma will be used in the proof of Lemma \ref{l:limit}.

\begin{Lemma}\label{l:m1}
Let $\psi=\max_{1\le j\le n}\{2p_j\log|w_j|\}$ be a plurisubharmonic function on $\mathbb{C}^n$, where $p_j>0$.
	Let $f=\sum_{\alpha\in \mathbb{Z}_{\ge0}^n}b_{\alpha}w^{\alpha}$ (Taylor expansion) be a holomorphic function on $\{\psi<-t_0\}$, where $t_0>0$. Let $c(t)$ be a nonnegative measurable function on $(t_0,+\infty)$. Denote that $q_{\alpha}:=\sum_{1\le j\le n}\frac{\alpha_j+1}{p_j}-1$ for any $\alpha\in\mathbb{Z}_{\ge0}^n$. Then
	$$\int_{\{\psi<-t\}}|f|^2c(-\psi)d\lambda_n=\sum_{\alpha\in\mathbb{Z}_{\ge0}^n}\left(\int_t^{+\infty}c(s)e^{-(q_{\alpha}+1)s}ds\right)\frac{(q_{\alpha}+1)|b_{\alpha}|^2\pi^{n}}{\prod_{1\le j\le n}(\alpha_j+1)}$$
	holds for any $t\ge t_0$.\end{Lemma}
	
\begin{proof}
	By direct calculations, we obtain that
	\begin{equation}\label{eq:1127b}\begin{split}
		&\int_{\{\psi<-t\}}|w^{\alpha}|^2c(-\psi)d\lambda_n\\
		=&(2\pi)^n\int_{\left\{\max_{1\le j\le n}\left\{s_j^{p_j}\right\}<e^{-\frac{t}{2}}\,\&\,s_j>0\right\}}\prod_{1\le j\le n}s_j^{2\alpha_j+1}\cdot c\left(-\log\max_{1\le j\le n}\left\{s_j^{2p_j}\right\}\right)ds_1ds_2\ldots ds_n\\
		=&(2\pi)^n\frac{1}{\prod_{1\le j\le n}p_j}\\
		&\times\int_{\left\{\max_{1\le j\le n}\{r_j\}<e^{-\frac{t}{2}}\,\&\,r_j>0\right\}}\prod_{1\le j\le n}r_j^{\frac{2\alpha_j+2}{p_j}-1}\cdot c\left(-\log\max_{1\le j\le n}\left\{r_j^2\right\}\right)dr_1dr_2\ldots dr_n.
		\end{split}
	\end{equation}
 By the Fubini's theorem, we have
\begin{equation}
	\label{eq:211125f}\begin{split}
		&\int_{\left\{\max_{1\le j\le n}\{r_j\}<e^{-\frac{t}{2}}\,\&\,r_j>0\right\}}\prod_{1\le j\le n}r_j^{\frac{2\alpha_j+2}{p_j}-1}\cdot c\left(-\log\max_{1\le j\le n}\left\{r_j^2\right\}\right)dr_1dr_2\ldots dr_n\\
		=&\sum_{j'=1}^n\int_{0}^{e^{-\frac{t}{2}}}\left(\int_{\{0\le r_j<r_{j'},j\not=j'\}}\prod_{j\not=j'}r_j^{\frac{2\alpha_j+2}{p_j}-1}\cdot\wedge_{j\not=j'}dr_j\right)r_{j'}^{\frac{2\alpha_{j'}+2}{p_{j'}}-1}c(-2\log r_{j'})dr_{j'}\\
		=&\sum_{j'=1}^n\left(\prod_{j\not=j'}\frac{p_j}{2\alpha_j+2}\right)\int_{0}^{e^{-\frac{t}{2}}}r_{j'}^{\sum_{1\le k\le n}\frac{2\alpha_k+2}{p_k}-1}c(-2\log r_{j'})dr_{j'}\\
		=&(q_{\alpha}+1)\left(\int_t^{+\infty}c(s)e^{-(q_{\alpha}+1)s}ds\right)\prod_{1\le j\le n}\frac{p_j}{2\alpha_j+2}.
			\end{split}
\end{equation}		
Following from $\int_{\{\psi<-t\}}|f|^2c(-\psi)d\lambda_n=\sum_{\alpha\in\mathbb{Z}_{\ge0}^n}|b_{\alpha}|^2\int_{\{\psi<-t\}}|w^{\alpha}|^2c(-\psi)d\lambda_n$, equality \eqref{eq:1127b} and equality \eqref{eq:211125f}, we obtain that
\begin{equation*}
\int_{\{\psi<-t\}}|f|^2d\lambda_n=\sum_{\alpha\in\mathbb{Z}_{\ge0}^n}\left(\int_t^{+\infty}c(s)e^{-(q_{\alpha}+1)s}ds\right)\frac{(q_{\alpha}+1)|b_{\alpha}|^2\pi^n}{\prod_{1\le j\le n}(\alpha_j+1)}.
\end{equation*}
\end{proof}

The following lemma will be used in the proof of Proposition \ref{p:exten-fibra}.
\begin{Lemma}[see \cite{GY-concavity4}]\label{l:m2}
Let $\psi=\max_{1\le j\le n}\{2p_j\log|w_j|\}$ be a plurisubharmonic function on $\mathbb{C}^n$, where $p_j>0$.
	Let $f=\sum_{\alpha\in\mathbb{Z}_{\ge0}^n}b_{\alpha}w^{\alpha}$ (Taylor expansion) be a holomorphic function on $\{\psi<-t_0\}$, where $t_0>0$. Denote that $q_{\alpha}:=\sum_{1\le j\le n}\frac{\alpha_j+1}{p_j}-1$ for any $\alpha\in\mathbb{Z}_{\ge0}^n$ and  $E_1:=\{\alpha\in\mathbb{Z}_{\ge0}^n:q_{\alpha}=0\}$. Then
\begin{displaymath}\begin{split}
	\int_{\{-t-1<\psi<-t\}}|f|^2e^{-\psi}d\lambda_n=&\sum_{\alpha\in E_1}\frac{|b_{\alpha}|^2\pi^{n}}{\prod_{1\le j\le n}(\alpha_j+1)}\\
	&+\sum_{\alpha\not\in E_1}\frac{|b_{\alpha}|^2\pi^{n}(q_{\alpha}+1)(e^{-q_{\alpha}t}-e^{-q_{\alpha}(t+1)})}{q_{\alpha}\prod_{1\le j\le n}(\alpha_j+1)}	
\end{split}\end{displaymath}
	 for any $t>t_0$.
\end{Lemma}

\subsection{Some results about fibrations}
\

In this section, we discuss the fibrations.

Let $\Delta^{n_1}=\{w\in\mathbb{C}^{n_1}:|w_j|<1$ for any $j\in\{1,\ldots,n_1\}\}$ be product of the unit disks. Let $Y$ be an $n_2-$dimensional complex manifold, and let $M=\Delta^{n_1} \times Y$. Denote $n=n_1+n_2$. Let $\pi_1$ and $\pi_2$ be the natural projections from $M$ to $\Delta^{n_1}$ and $Y$ respectively.  Let $\rho_1$ be a nonnegative Lebesgue measurable function on $\Delta^{n_1}$ satisfying that $\rho_1(w)=\rho_1(|w_1|,\ldots,|w_{n_1}|)$ for any $w\in\Delta^{n_1}$ and the Lebesgue measure of $\{w\in\Delta^{n_1}: \rho_1(w)>0\}$ is  positive. Let $\rho_2$ be a nonnegative Lebesgue measurable function on $Y$, and denote that $\rho=\pi_1^*(\rho_1)\times\pi_2^*(\rho_2)$ on $M$.
\begin{Lemma}
	\label{l:fibra-decom}	For any holomorphic $(n,0)$  form $F$ on $M$, there exists a unique sequence of  holomorphic  $(n_2,0)$   forms $\{F_{\alpha}\}_{\alpha\in\mathbb{Z}_{\ge0}^{n_1}}$ on $Y$ such that
	\begin{equation}\label{eq:1216c}
		F=\sum_{\alpha\in\mathbb{Z}_{\ge0}^{n_1}}\pi_1^*(w^{\alpha}dw_1\wedge\ldots\wedge dw_{n_1}) \wedge \pi_2^*(F_\alpha),
	\end{equation}
where the right term of the above equality is uniformly convergent on any compact subset of $M$. Moreover, if 	$\int_{M}|F|^2\rho<+\infty,$ we have
	\begin{equation}
	\label{eq:1216d}\int_{Y}|F_{\alpha}|^2\rho_2<+\infty
\end{equation}
for any $\alpha\in\mathbb{Z}_{\ge0}^{n_1}$.
\end{Lemma}

\begin{proof}
Firstly, we consider the local case. Assume that $Y=\Delta^{n_2}$, and the coordinate is $\tilde w=(\tilde w_1,\ldots,\tilde w_{n_2})$. Then there exists a holomorphic function $\tilde{F}(w,\tilde w)$ on $\Delta^n$ such that
\begin{equation*}
	F=\tilde{F}(w,\tilde w)dw_1\wedge\ldots\wedge dw_{n_1}\wedge d\tilde w_1\ldots\wedge d\tilde w_{n_2}.
\end{equation*}
Let
$$F_\alpha=\frac{1}{\alpha!}\left(\left(\frac{\partial}{\partial w}\right)^{\alpha}\tilde{F}\right)\bigg|_{w=0}d\tilde w_1\wedge\ldots\wedge d\tilde w_{n_2}$$
be a holomorphic $(n_2,0)$ form on $Y$.
 Considering the Taylor's expansion of $\tilde{F}$, we can assume that
\begin{equation*}
	\tilde{F}(w,\tilde w)=\sum_{\alpha\in\mathbb{Z}_{\ge0}^{n_1},\tilde\alpha\in \mathbb{Z}_{\ge0}^{n_2}}d_{\alpha,\tilde\alpha}w^{\alpha}{\tilde w}^{\tilde\alpha}=\sum_{\alpha\in\mathbb{Z}_{\ge0}^{n_1}}\frac{1}{\alpha!}\left(\left(\frac{\partial}{\partial w}\right)^{\alpha}\tilde{F}\right)\bigg|_{w=0}\cdot w^{\alpha},
\end{equation*}
where the summations are uniformly convergent on any compact subset of $M$, then we have
\begin{equation*}
	F=\sum_{\alpha\in\mathbb{Z}_{\ge0}^{n_1}}\pi_1^*(w^{\alpha}dw_1\wedge\ldots\wedge dw_{n_1}) \wedge \pi_2^*(F_\alpha).
\end{equation*}

Secondly, we need to prove that the gluing is independent of the choices of the local coordinates of $Y$. Assume that $y=(y_1,\ldots,y_{n_2})$ is another coordinate on $Y=\Delta^{n_2}$, and $F=\tilde{F}_0(w,y)dw_1\wedge\ldots\wedge dw_{n_1} \wedge dy_1\wedge\ldots\wedge dy_{n_2}$, thus we have $\tilde{F}(w,\tilde w(y))\frac{\partial(\tilde w_1,\ldots,\tilde w_{n_2})}{\partial(y_1,\ldots,y_{n_2})}=\tilde{F}_0(w,y)$. By direct calculations, we have
\begin{flalign*}
	\begin{split}
	F_\alpha&=\frac{1}{\alpha!}\left(\left(\frac{\partial}{\partial w}\right)^{\alpha}\tilde{F}\right)\bigg|_{w=0}d\tilde w_1\wedge\ldots\wedge d\tilde w_{n_2} \\
	&=\frac{1}{\alpha!}\left(\left(\frac{\partial}{\partial w}\right)^{\alpha}\tilde{F}\right)\bigg|_{w=0}\frac{\partial(\tilde w_1,\ldots,\tilde w_{n_2})}{\partial(y_1,\ldots,y_{n_2})}dy_1\wedge\ldots\wedge dy_{n_2} \\
	&=\frac{1}{\alpha!}\left(\left(\frac{\partial}{\partial w}\right)^{\alpha}\tilde{F}_0\right)\bigg|_{w=0} dy_1\wedge\ldots\wedge dy_{n_2},
	\end{split}
\end{flalign*}
which means that $F_{\alpha}$ is independent of the choices of the coordinates for any $\alpha\in\mathbb{Z}_{\ge0}^{n_1}$.
For general $Y$, we can find holomorphic $(n_2,0)$ forms $F_\alpha$ on $Y$ such that $F=\sum_{\alpha\in\mathbb{Z}_{\ge0}^{n_1}}\pi_1^*(w^{\alpha}dw_1\wedge\ldots\wedge dw_{n_1}) \wedge \pi_2^*(F_\alpha)$.

Then, for the uniqueness, it suffices to prove  the local case $Y=\Delta^{n_2}$. There exists a holomorphic function $\tilde{F}(w,\tilde w)$ on $\Delta^n$ such that
$F=\tilde{F}(w,\tilde w)dw_1\wedge\ldots\wedge dw_{n_1}\wedge d\tilde w_1\ldots\wedge d\tilde w_{n_2}$. If
\begin{equation*}
	F=\sum_{\alpha\in\mathbb{Z}_{\ge0}^{n_1}}\pi_1^*(w^{\alpha}dw_1\wedge\ldots\wedge dw_{n_1}) \wedge \pi_2^*(F_\alpha)
\end{equation*} for a holomorphic $(n_2,0)$ form $F_{\alpha}$ on $Y$, we have
\begin{equation*}
	F_\alpha=\frac{1}{\alpha!}\left(\left(\frac{\partial}{\partial w}\right)^{\alpha}\tilde{F}\right)\bigg|_{w=0}d\tilde w_1\wedge\ldots\wedge d\tilde w_{n_2}.
\end{equation*}
Thus, the uniqueness holds.

Finally, we prove inequality \eqref{eq:1216d}. Let $f=\sum_{\alpha\in \mathbb{Z}_{\ge0}^{n_1}}b_{\alpha}w^{\alpha}$ be a holomorphic function on $\Delta^{n_1}$. As $\rho(w)=\rho(|w_1|,\ldots,|w_{n_1}|)$ for any $w\in\Delta^{n_1}$, we have
	\begin{equation}
\label{eq:1216e}
		\begin{split}
			&\int_{\Delta^{n_1}}|f|^2\rho_1d\lambda_{n_1}\\
			=&\sum_{\alpha\in\mathbb{Z}_{\ge0}^{n_1}}(2\pi)^{n_1}|b_{\alpha}|^2\int_{\{0\le r_1\le 1\}\times\ldots\{0\le r_{n_1}\le 1\}}\left(\prod_{1\le j\le n_1}r_j^{2\alpha_j}\right)\rho_1(r_1,\ldots,r_{n_1})dr_{1}\ldots dr_{n_1}\\
			=&\sum_{\alpha\in\mathbb{Z}_{\ge0}^{n_1}}|b_{\alpha}|^2\int_{\Delta^{n_1}}|w^{\alpha}|^2\rho_1 d\lambda_{n_1}.
		\end{split}		
	\end{equation}
It follows from equality \eqref{eq:1216c}, equality \eqref{eq:1216e} and the Fubini's theorem that
	\begin{equation}
		\label{eq:1216f}\begin{split}
			\int_{M}|F|^2\rho&=\int_{\Delta^{n_1}\times Y}\bigg|\sum_{\alpha\in\mathbb{Z}_{\ge0}^{n_1}}\pi_1^*(w^{\alpha}dw_1\wedge\ldots\wedge dw_{n_1}) \wedge \pi_2^*(F_\alpha)\bigg|^2\pi_1^*(\rho_1)\pi_2^*(\rho_2)\\
			&=\sum_{\alpha\in\mathbb{Z}_{\ge0}^{n_1}}\left(\int_{\Delta^{n_1}}|w^{\alpha}dw_1\wedge\ldots\wedge dw_{n_1}|^2\rho_1\right)\left(\int_Y|F_{\alpha}|^2\rho_2\right).
		\end{split}
	\end{equation}
	As $\int_{M}|F|^2\rho<+\infty$ and the Lebesgue measure of $\{w\in\Delta^{n_1}: \rho_1(w)>0\}$ is a positive number, equality \eqref{eq:1216f} implies that $\int_{Y}|F_{\alpha}|^2\rho_2<+\infty$
for any $\alpha\in\mathbb{Z}_{\ge0}^{n_1}$.
\end{proof}

Let $M_1\subset M$ be an $n-$dimensional complex manifold satisfying that $\{o\}\times Y\subset M_1$, where $o$ is the origin in $\Delta^{n_1}$.

\begin{Lemma}
	\label{l:fibra-decom-2}For any holomorphic $(n,0)$  form $F$ on $M_1$, there exist a unique sequence of  holomorphic  $(n_2,0)$   forms $\{F_{\alpha}\}_{\alpha\in\mathbb{Z}_{\ge0}^{n_1}}$ on $Y$ and a neighborhood $M_2\subset M_1$ of $\{o\}\times Y$, such that
	\begin{equation*}
		F=\sum_{\alpha\in\mathbb{Z}_{\ge0}^{n_1}}\pi_1^*(w^{\alpha}dw_1\wedge\ldots\wedge dw_{n_1}) \wedge \pi_2^*(F_\alpha)
	\end{equation*}
on $M_2$, where the right term of the above equality is uniformly convergent on any compact subset of $M_2$. Moreover, if
	$\int_{M_1}|F|^2\rho<+\infty,$ we have
	\begin{equation*}
\int_{K}|F_{\alpha}|^2\rho_2<+\infty
\end{equation*}
for any compact subset $K$ of $Y$ and $\alpha\in\mathbb{Z}_{\ge0}^{n_1}$.
\end{Lemma}

\begin{proof}
	For any open subset $V$ of $Y$ satisfying $V\Subset Y$, there exists $s_V\in(0,1)$ such that $\Delta_{s_V}^{n_1}\times V\subset M_1$, where $\Delta_{s_V}=\{w\in\mathbb{C}:|w|<s_V\}$. It follows from Lemma \ref{l:fibra-decom} that there exists a sequence of holomorphic $(n_2,0)$ forms $\{F_{V,\alpha}\}_{\alpha\in\mathbb{Z}_{\ge0}^{n_1}}$ on $V$ such that
	$$F=\sum_{\alpha\in\mathbb{Z}_{\ge0}^{n_1}}\pi_1^*(w^{\alpha}dw_1\wedge\ldots\wedge dw_{n_1}) \wedge \pi_2^*(F_{V,\alpha})$$
	on $\Delta_{s_V}^{n_1}\times V$, where the right term of the above equality is uniformly convergent on any compact subset of $\Delta_{s_V}^{n_1}\times V$. If $\int_{M_1}|F|^2\rho<+\infty,$ Lemma \ref{l:fibra-decom} shows that
	$$\int_V|F_{V,\alpha}|^2\rho_2<+\infty.$$
	 Following from the uniqueness of decomposition in Lemma \ref{l:fibra-decom}, we get that there exists  a unique sequence of  holomorphic  $(n_2,0)$   forms $\{F_{\alpha}\}_{\alpha\in\mathbb{Z}_{\ge0}^{n_1}}$ on $Y$ and a neighborhood $M_2\subset M_1$ of $\{o\}\times Y$, such that
	\begin{equation}\label{eq:211222a}
		F=\sum_{\alpha\in\mathbb{Z}_{\ge0}^{n_1}}\pi_1^*(w^{\alpha}dw_1\wedge\ldots\wedge dw_{n_1}) \wedge \pi_2^*(F_\alpha)
	\end{equation}
on $M_2$, where the right term of the above equality is uniformly convergent on any compact subset of $M_2$. Moreover, if
	$\int_{M_1}|F|^2\rho<+\infty,$ we have
\begin{equation*}
\int_{K}|F_{\alpha}|^2\rho_2<+\infty
\end{equation*}
for any compact subset $K$ of $Y$ and $\alpha\in\mathbb{Z}_{\ge0}^{n_1}$.
\end{proof}

Let $M=X\times Y$ be $n-$dimensional complex manifold, and let $K_M$ be the canonical (holomorphic) line bundle on $M$, where $X$ is an $n_1-$dimensional weakly pseudoconvex K\"ahler manifold, $Y$ is an $n_2-$dimensional weakly pseudoconvex K\"ahler manifold, and $n=n_1+n_2$. Let $K_X$ and $K_Y$ be the canonical (holomorphic) line bundles on $X$ and $Y$ respectively. Let $\pi_X$ and $\pi_Y$ be the natural projections from $M$ to $X$ and $Y$ respectively. It is clear that $(M,\emptyset,\emptyset)$ satisfies condition $(A)$.

Let $\psi_1$ be a plurisubharmonic function on $X$, and let $\varphi_1$ be a Lebesgue measurable function on $X$ such that $\varphi_1+\psi_1$ is plurisubharmonic. Let $\varphi_2$ be a plurisubharmonic function on $Y$. Denote that $\varphi:=\pi_X^*(\varphi_1)+\pi_Y^*(\varphi_2)$ and $\psi:=\pi_X^*(\psi_1)$ on $M$. Let $T=-\sup_M\psi$, and let $c\in\mathcal{P}_{T,M}$ satisfying $\int_T^{+\infty}c(s)e^{-s}ds<+\infty$.

 Let $Z_0\subset X$ be a subset of $\{\psi_1=-\infty\}$ such that $Z_0 \cap
Supp\left(\mathcal{O}_X/\mathcal{I}(\varphi_1+\psi_1)\right)\neq \emptyset$,  and let $\tilde{Z}_0=Z_0\times Y\subset M$. Let $U \supset Z_0$ be
an open subset of $X$, and let $f_1$ be a holomorphic $(n_1,0)$ form on $U$. Let $f_2$ be a holomorphic $(n_2,0)$ form on $Y$, and let $f=\pi_X^*(f_1)\wedge\pi_Y^*(f_2)$ on $U\times Y$. Let $\mathcal{F}_{x} \supset \mathcal{I}(\varphi_1+\psi_1)_{x}$ be an ideal of $\mathcal{O}_{X,x}$ for any $x\in Z_0$. Let $\tilde{\mathcal{F}}_{z} \supset \mathcal{I}(\varphi+\psi)_{z}$ be an ideal of $\mathcal{O}_{M,z}$ for any $z\in \tilde Z_0$. For any $x\in Z_0$ and any holomorphic function $g$, assume that  $(g,(x,y))\in\tilde{\mathcal{F}}_{(x,y)}$ for any $y\in Y$ if and only if
$(g(\cdot,y),x)\in \mathcal{F}_x$ for any $y\in Y$.

Denote
\begin{equation*}
\begin{split}
\inf\bigg\{\int_{\{\psi_1<-t\}}|\tilde{f}|^{2}e^{-\varphi_1}c(-\psi_1):(\tilde{f}-f_1)\in H^{0}(Z_0,&
(\mathcal{O}(K_{X})\otimes\mathcal{F})|_{Z_0})\\&\&{\,}\tilde{f}\in H^{0}(\{\psi_1<-t\},\mathcal{O}(K_{X}))\bigg\}
\end{split}
\end{equation*}
by $G_X(t)$,  where $t\in[T,+\infty)$,
$|f|^{2}:=\sqrt{-1}^{n_1^{2}}f\wedge\bar{f}$ for any $(n_1,0)$ form $f$ and $(\tilde{f}-f)\in H^{0}(Z_0,
(\mathcal{O}(K_{M})\otimes\mathcal{F})|_{Z_0})$ means $(\tilde{f}-f,x)\in\mathcal{O}(K_{X})_{x}\otimes\mathcal{F}_{x}$ for all $x\in Z_0$. Denote
\begin{equation*}
\begin{split}
\inf\bigg\{\int_{\{\psi<-t\}}|\tilde{f}|^{2}e^{-\varphi}c(-\psi):(\tilde{f}-f)\in H^{0}(\tilde Z_0,&
(\mathcal{O}(K_{M})\otimes\tilde{\mathcal{F}})|_{\tilde Z_0})\\&\&{\,}\tilde{f}\in H^{0}(\{\psi<-t\},\mathcal{O}(K_{M}))\bigg\}
\end{split}
\end{equation*}
by $G_M(t)$,  where $t\in[T,+\infty)$.

Theorem \ref{thm:general_concave} shows that $G_X(h^{-1}(r))$ and $G_M(h^{-1}(r))$ are concave with respect to $r$, where $h(t)=\int_t^{+\infty}c(s)e^{-s}ds$. The following Proposition gives a property of the minimal $L^2$ integrals on fibration, which implies that $G_M(h^{-1}(r))$ is linear with respect to $r$ if and only if $G_X(h^{-1}(r))$ is linear with respect to $r$.
\begin{Proposition}
	\label{p:fibra}
	$G_M(t)=G_X(t)\int_Y|f_2|^2e^{-\varphi_2}$ holds for any $t\ge T$. Moreover, if $G_X(t)<+\infty$, there exists a holomorphic $(n_1,0)$ form $F_1$ on $\{\psi_1<-t\}$ such that $(F_{1,t}-f_1)\in H^{0}(Z_0,
(\mathcal{O}(K_{X})\otimes\mathcal{F})|_{Z_0})$,  $G_X(t)=\int_{\{\psi_1<-t\}}|F_{1,t}|^2e^{-\varphi_1}c(-\psi_1)$ and $G_M(t)=\int_{\{\psi<-t\}}|\pi_X^*(F_{1,t})\wedge\pi_2^*(f_2)|^2e^{-\varphi}c(-\psi)$.
\end{Proposition}
\begin{proof}
	Let $\tilde f_1$ be a holomorphic $(n_1,0)$ form on $\{\psi_1<-t\}$ satisfying $(\tilde f_1-f_1)\in H^{0}(Z_0,
(\mathcal{O}(K_{X})\otimes\mathcal{F})|_{Z_0})$, where $t\ge T$. As $f=\pi_X^*(f_1)\wedge\pi_Y^*(f_2)$ and $\tilde{Z}_0=Z_0\times Y$, it follows from the relationship between $\tilde{\mathcal{F}}$ and $\mathcal{F}$ that $(\pi_X^*(\tilde{f_1})\wedge\pi_Y^*(f_2)-f)\in H^{0}(\tilde Z_0,
(\mathcal{O}(K_{M})\otimes\tilde{\mathcal{F}})|_{\tilde Z_0})$. By the definitions of $G_X(t)$ and $G_M(t)$, we obtain that
\begin{equation}
	\label{eq:1215a}G_M(t)\le G_X(t)\int_Y|f_2|^2e^{-\varphi_2}
\end{equation}
for any $t\ge T$.

Let $t\ge T$. If $G_M(t)=+\infty$, inequality \eqref{eq:1215a} implies that $G_X(t)\int_Y|f_2|^2e^{-\varphi_2}=G_M(t)=+\infty$. Thus,  assume that $G_M(t)<+\infty$. Lemma \ref{lem:A} shows that there exists a holomorphic $(n,0)$ form $F_t$ on $\{\psi<-t\}$ such that $(F_t-f)\in H^{0}(\tilde Z_0,
(\mathcal{O}(K_{M})\otimes\tilde{\mathcal{F}})|_{\tilde Z_0})$ and $G_M(t)=\int_{\{\psi<-t\}}|F_t|^2e^{-\varphi}c(-\psi)$. For any $y_0\in Y$, let $w=(w_1,\ldots,w_{n_2})$ be a coordinate on a neighborhood $U$ of $y$ satisfying $w(y_0)=0$ and $w(U)=\Delta^{n_2}$. Lemma \ref{l:fibra-decom} implies that $F_t|_{U\times Y}=\sum_{\alpha\in\mathbb{Z}_{\ge0}^{n_2}}\pi_X^*(f_{\alpha})\wedge \pi_Y^*(w^{\alpha}dw_1\wedge\ldots\wedge dw_{n_2})$, where $f_{\alpha}$ is a holomorphic $(n_1,0)$ form on $\{\psi_1<-t\}$ for any $\alpha\in\mathbb{Z}_{\ge0}^{n_2}$. There exists a holomorphic function $\tilde f_2(w)$ on $U$ such that $f_2=\tilde f_2(w)dw_1\wedge\ldots\wedge dw_{n_2}$ on $U$. As   $(g,(x,y))\in\tilde{\mathcal{F}}_{(x,y)}$ for any $y\in Y$ if and only if
$(h(\cdot,y),x)\in \mathcal{F}_x$ for any $y\in Y$, where $x\in Z_0$ and $g$ is a holomorphic function, it follows from $(F_t-f)\in H^{0}(\tilde Z_0,
(\mathcal{O}(K_{M})\otimes\tilde{\mathcal{F}})|_{\tilde Z_0})$ and $f=\pi_X^*(f_1)\wedge\pi_Y^*(f_2)$ that $(\sum_{\alpha\in\mathbb{Z}_{\ge0}^{n_2}}w^{\alpha}f_{\alpha}-\tilde f_2(w)f_1)\in H^{0}(Z_0,
(\mathcal{O}(K_{X})\otimes{\mathcal{F}})|_{Z_0})$  for any $w\in \Delta^{n_2}$.  Let $U_1$ be an open subset of $U$, and let $V=w(U_1)\subset\Delta^{n_2}$.  Following the Fubini's theorem and the definition of $G_X(t)$, we have
\begin{equation*}
	\begin{split}
		&\int_{\{\psi_1<-t\}\times U_1}|F_t|^2e^{-\varphi}c(-\psi)\\
		=&\int_{V}\bigg(\int_{\{\psi_1<-t\}}|\sum_{\alpha\in\mathbb{Z}_{\ge0}^{n_2}}w^{\alpha}f_{\alpha}|^2e^{-\varphi_1}c(-\psi_1)\bigg)e^{-\varphi_2}|dw_1\wedge\ldots\wedge dw_{n_2}|^2\\
		\ge &G_X(t)\int_{V}|\tilde f_2(w)dw_1\wedge\ldots\wedge dw_{n_2}|^2e^{-\varphi_2}\\
		=&G_X(t)\int_{U_1}|f_2|^2e^{-\varphi_2},
	\end{split}
\end{equation*}
which implies $G_M(t)=\int_{\{\psi<-t\}}|F_t|^2e^{-\varphi}c(-\psi)\ge G_X(t)\int_{Y}|f_2|^2e^{-\varphi_2}$.
Thus, we have $G_M(t)=G_X(t)\int_Y|f_2|^2e^{-\varphi_2}$  for any $t\ge T$. If $G_X(t)<+\infty$, it follows from Lemma \ref{lem:A} that there exists a holomorphic $(n_1,0)$ for $F_{1,t}$ on $\{\psi_1<-t\}$ satisfying that $(F_{1,t}-f_1)\in H^{0}(Z_0,
(\mathcal{O}(K_{X})\otimes\mathcal{F})|_{Z_0})$ and $G_X(t)=\int_{\{\psi<-t\}}|F_{1,t}|^2e^{-\varphi_1}c(-\psi_1)$, hence $G_M(t)=G_X(t)\int_{Y}|f_2|^2e^{-\varphi_2}=\int_{\{\psi<-t\}}|\pi_X^*(F_{1,t})\wedge\pi_2^*(f_2)|^2e^{-\varphi}c(-\psi).$
\end{proof}

 We recall a result about multiplier ideal sheaves.
 \begin{Lemma}
 	\label{l:phi1+phi2}Let $\Phi_1$ and $\Phi_2$ be plurisubharmonic functions on $\Delta^n$ satisfying $\Phi_2(o)>-\infty$, where $n\in \mathbb{Z}_{>0}$ and $o$ is the origin in $\Delta^n$. Then $\mathcal{I}(\Phi_1)_o=\mathcal{I}(\Phi_1+\Phi_2)_o$.
 \end{Lemma}
 \begin{proof}For convenience of the reader, we give the proof.
 	It is clear that $\mathcal{I}(\Phi_1+\Phi_2)_o\subset\mathcal{I}(\Phi_1)_o$.  Let $f$ be a holomorphic function on a neighborhood of $o$ satisfying $(f,o)\in\mathcal{I}(\Phi_1)_o$. Following from the strong openness property of multiplier ideal sheaves  (\cite{GZSOC}) and $\Phi_2(o)>-\infty$, there exist $s>1$ and $r>0$ such that
 	\begin{equation}
 		\label{eq:1216a}
 		\int_{|w|<r}|f|^{2s}e^{-s\Phi_1}d\lambda_n<+\infty
 	\end{equation}
and
\begin{equation}
	\label{eq:1216b}\int_{|w|<r}e^{-\frac{s}{s-1}\Phi_2}d\lambda_n <+\infty,
\end{equation}
where $d\lambda_n$ is the Lebesgue measure on $\mathbb{C}^n$.
Combining inequality \eqref{eq:1216a}, inequality \eqref{eq:1216b} and the H\"older inequality, we have
\begin{displaymath}
	\begin{split}
		&\int_{|w|<r}|f|^2e^{-\Phi_1-\Phi_2}d\lambda_n\\
		\le& \left(\int_{|w|<r}|f|^{2s}e^{-s\Phi_1}d\lambda_n\right)^{\frac{1}{s}}\left(\int_{|w|<r}e^{-\frac{s}{s-1}\Phi_2}d\lambda_n \right)^{1-\frac{1}{s}}\\
		<&+\infty,
	\end{split}
\end{displaymath}
which implies that $(f,o)\in \mathcal{I}(\Phi_1+\Phi_2)_o$. Thus, we have $\mathcal{I}(\Phi_1)_o=\mathcal{I}(\Phi_1+\Phi_2)_o$.
 \end{proof}

In the following, we consider fibrations over products of open Riemann surfaces. Let $\Omega_j$  be an open Riemann surface, which admits a nontrivial Green function $G_{\Omega_j}$ for any  $1\le j\le n_1$. Let $Y$ be an $n_2-$dimensional weakly pseudoconvex K\"ahler manifold. Let $M=\left(\prod_{1\le j\le n_1}\Omega_j\right)\times Y$ be an $n-$dimensional complex manifold, where $n=n_1+n_2$. Let $\pi_{1}$, $\pi_{1,j}$ and $\pi_2$ be the natural projections from $M$ to $\prod_{1\le j\le n_1}\Omega_j$, $\Omega_j$ and $Y$ respectively. Let $K_M$ be the canonical (holomorphic) line bundle on $M$.

  Let ${Z}_j=\{z_{j,k}:1\le k<\tilde m_j\}$ be a discrete subset of $\Omega_j$ for any  $j\in\{1,\ldots,n_1\}$, where $\tilde{m}_j\in\mathbb{Z}_{\ge2}\cup\{+\infty\}$. Denote that $Z_0:=\left(\prod_{1\le j\le n_1}Z_j\right)\times Y\subset M$.
Let $p_{j,k}$ be a positive number such that $\sum_{1\le k<\tilde{m}_j}p_{j,k}G_{\Omega_j}(\cdot,z_{j,k})\not\equiv-\infty$ for any $j$, and
let
$$\psi=\max_{1\le j\le n_1}\left\{\pi_{1,j}^*\left(2\sum_{1\le k<\tilde{m}_j}p_{j,k}G_{\Omega_j}(\cdot,z_{j,k})\right)\right\}$$
on $M$.  For any $j\in\{1,\ldots,n_1\}$, let $\varphi_j$ be a subharmonic function on $\Omega_{j}$ such that $\varphi_j(z)>-\infty$ for any $z\in Z_j$. Let $\varphi_Y$ be a plurisubharmonic function on $Y$, and denote that $\varphi:=\sum_{1\le j\le n_1}\pi_{1,j}^*(\varphi_j)+\pi_2^*(\varphi_Y)$.
Let $c$ be a positive function on $(0,+\infty)$ such that $\int_{0}^{+\infty}c(t)e^{-t}dt<+\infty$ and $c(t)e^{-t}$ is decreasing on $(0,+\infty)$. Let $f$ be a holomorphic $(n,0)$ form on  $\{\psi<-t_0\}$ satisfying $\int_{\{\psi<-t_0\}}|f|^2e^{-\varphi}c(-\psi)<+\infty$, where $t_0>0$ is constant.

Denote
\begin{equation*}
\begin{split}
\inf\bigg\{\int_{\{\psi<-t\}}|\tilde{f}|^{2}e^{-\varphi}c(-\psi):(\tilde{f}-f,z)\in(\mathcal{O}(K_M)&\otimes\mathcal{I}(\varphi+\psi))_{z}\mbox{ for any $z\in Z_0$} \\&\&{\,}\tilde{f}\in H^{0}(\{\psi<-t\},\mathcal{O}(K_{M}))\bigg\}
\end{split}
\end{equation*}
by $G(t)$,  where $t\in[0,+\infty)$, and denote
\begin{equation*}
\begin{split}
\inf\bigg\{\int_{\{\psi<-t\}}|\tilde{f}|^{2}e^{-\varphi}c(-\psi):(\tilde{f}-f,z)\in(\mathcal{O}(K_M)&\otimes\mathcal{I}(\psi))_{z}\mbox{ for any $z\in Z_0$} \\&\&{\,}\tilde{f}\in H^{0}(\{\psi<-t\},\mathcal{O}(K_{M}))\bigg\}
\end{split}
\end{equation*}
by $\tilde G(t)$,  where $t\in[0,+\infty)$.

\begin{Lemma}\label{l:G1=G2}
	Let $t\ge 0$. If $\tilde G(t)< +\infty$, there exists a unique holomorphic $(n,0)$ form $F_t$ on $\{\psi<-t\}$ satisfying that $(F_t-f,z)\in(\mathcal{O}(K_M)\otimes\mathcal{I}(\varphi+\psi))_{z}$ for any $z\in Z_0$ and $G(t)=\tilde G(t)=\int_{\{\psi<-t\}}|F|^2e^{-\varphi}c(-\psi)$.
\end{Lemma}
\begin{proof}
	As $\mathcal{I}(\varphi+\psi)\subset\mathcal{I}(\psi)$, we have  $\tilde G(t)\le G(t)$. It follows from Lemma \ref{lem:A} that there exists a unique holomorphic $(n,0)$ form $F_t$ on $\{\psi<-t\}$ satisfying that $\tilde G(t)=\int_{\{\psi<-t\}}|F_t|^2e^{-\varphi}c(-\psi)$ and $(F_t-f,z)\in(\mathcal{O}(K_M)\otimes\mathcal{I}(\psi))_{z}$  for any $z\in Z_0$.
	
	Let $z_0=(z_{1,\beta_1},\ldots,z_{n_1,\beta_{n_1}})\in\prod_{1\le j\le n_1}\Omega_j$, where $1\le \beta_j<\tilde m_j$ for any $1\le j\le n_1$. It follows from Lemma \ref{l:green-sup} and Lemma \ref{l:green-sup2} that there exists a local coordinate $w_{j}$ on a neighborhood $V_{z_{j,\beta_j}}\Subset\Omega_{j}$ of $z_{j,\beta_j}\in\Omega_j$ satisfying $w_{j}(z_{j,\beta_j})=0$ and
	$$\log|w_{j}|=\frac{1}{p_{j,\beta_j}}\sum_{1\le k<\tilde m_j}p_{j,k}G_{\Omega_j}(\cdot,z_{j,k})$$ 	on $V_{z_{j,\beta_j}}$ for any $j\in\{1,\ldots,n_1\}$. Denote that $V_0:=\prod_{1\le j\le n_1}V_{z_{j,\beta_j}}$ and $w:=(w_1,\ldots,w_{n_1})$ on $V_0$. 	Thus, there exists $t_1>\max\{t,t_0\}$ such that
	$$\left\{z\in\Omega_j: 2\sum_{1\le k<\tilde m_j}p_{j,k}G_{\Omega_j}(z,z_{j,k})<-t_1\right\}\cap V_{z_{j,\beta_j}}\Subset V_{z_{j,\beta_j}}$$
 for any $1\le j\le n_1$. Let $\psi_1=\max_{1\le j\le n_1}\left\{\tilde\pi_{j}^*\left(2\sum_{1\le k<\tilde{m}_j}p_{j,k}G_{\Omega_j}(\cdot,z_{j,k})\right)\right\}$ on $\prod_{1\le j\le n_1}\Omega_j$, where $\tilde\pi_j$ is the natural projection from $\prod_{1\le j\le n_1}\Omega_j$ to $\Omega_j$. Note that
	$$\{\psi<-t_1\}=\{\psi_1<-t_1\}\times Y$$
	and
	$$\{\psi_1<-t_1\}\cap V_0=\prod_{1\le j\le n_1}\left\{|w_j|<e^{-\frac{t_1}{2p_{j,\beta_j}}}\right\}.$$
	As $\varphi_j$ is a subharmonic function on $\Omega_j$,  $\int_{\{\psi<-t_1\}}|f|^2e^{-\varphi}c(-\psi)\le \int_{\{\psi<-t_0\}}|f|^2e^{-\varphi}c(-\psi)<+\infty$ implies that $\int_{\{\psi<-t_1\}}|f|^2e^{-\pi_2^*(\varphi_Y)}c(-\psi)<+\infty$ and $\int_{\{\psi<-t_1\}}|F_t|^2e^{-\varphi}c(-\psi)<+\infty$ implies that $\int_{\{\psi<-t_1\}}|F_t|^2e^{-\pi_2^*(\varphi_Y)}c(-\psi)<+\infty$.
	 It follows from Lemma \ref{l:fibra-decom} that there exist a  sequence of  holomorphic  $(n_2,0)$   forms $\{f_{\alpha}\}_{\alpha\in\mathbb{Z}_{\ge0}^{n_1}}$ on $Y$ and a  sequence of  holomorphic  $(n_2,0)$   forms $\{F_{\alpha}\}_{\alpha\in\mathbb{Z}_{\ge0}^{n_1}}$ on $Y$ such that
	\begin{equation}\label{eq:1216g}
		f=\sum_{\alpha\in\mathbb{Z}_{\ge0}^{n_1}}\pi_1^*(w^{\alpha}dw_1\wedge\ldots\wedge dw_{n_1}) \wedge \pi_2^*(f_\alpha)
	\end{equation} on $(\{\psi_1<-t_1\}\cap V_0)\times Y$,
		\begin{equation}\label{eq:1216h}
		F_t=\sum_{\alpha\in\mathbb{Z}_{\ge0}^{n_1}}\pi_1^*(w^{\alpha}dw_1\wedge\ldots\wedge dw_{n_1}) \wedge \pi_2^*(F_\alpha)
	\end{equation}on $(\{\psi_1<-t_1\}\cap V_0)\times Y$,
	\begin{equation}
		\label{eq:1216i}\int_Y|f_{\alpha}|^2e^{-\varphi_Y}<+\infty
	\end{equation}
	for any $\alpha\in\mathbb{Z}_{\ge0}^{n_1}$ and
	\begin{equation}
		\label{eq:1216j}\int_Y|F_{\alpha}|^2e^{-\varphi_Y}<+\infty
	\end{equation}	for any $\alpha\in\mathbb{Z}_{\ge0}^{n_1}$, where the right terms of the  equalities \eqref{eq:1216g} and \eqref{eq:1216h} are uniformly convergent on any compact subset of $(\{\psi_1<-t_1\}\cap V_0)\times Y$.  As $(F_t-f,(z_0,y))\in(\mathcal{O}(K_M)\otimes\mathcal{I}(\psi))_{(z_0,y)}$ for any $y\in Y$, it follows from Lemma \ref{l:0} that
	$$f_{\alpha}=F_{\alpha}$$
	 for any $\alpha\in\mathbb{Z}_{\ge0}^{n_1}$ satisfying $\sum_{1\le j\le n_1}\frac{\alpha_j+1}{p_{j,\beta_j}}\le 1$. Denote that
	 $$R:=\left\{\alpha\in\mathbb{Z}_{\ge0}^{n_1}:\sum_{1\le j\le n_1}\frac{\alpha_j+1}{p_{j,\beta_j}}> 1\right\}.$$ Lemma \ref{l:0} shows that $(w^{\alpha},z_0)\in\mathcal{I}(\psi_1)_{z_0}$ for any $\alpha\in R$. It follows from inequality \eqref{eq:1216i} and inequality \eqref{eq:1216j} that $(\pi_1^*(w^{\alpha}dw_1\wedge\ldots\wedge dw_{n_1}) \wedge \pi_2^*(f_\alpha),(z_0,y))\in (\mathcal{O}(K_M)\otimes\mathcal{I}(\psi+\pi_2^*(\varphi_Y)))_{(z_0,y)}$ and $(\pi_1^*(w^{\alpha}dw_1\wedge\ldots\wedge dw_{n_1}) \wedge \pi_2^*(F_\alpha),(z_0,y))\in (\mathcal{O}(K_M)\otimes\mathcal{I}(\psi+\pi_2^*(\varphi_Y)))_{(z_0,y)}$ for any $y\in Y$ and any $\alpha\in R$. As $\varphi_j(z_{j,\beta_j})>-\infty$, using Lemma \ref{l:phi1+phi2}, we obtain that
	$$(\pi_1^*(w^{\alpha}dw_1\wedge\ldots\wedge dw_{n_1}) \wedge \pi_2^*(f_\alpha),(z_0,y))\in (\mathcal{O}(K_M)\otimes\mathcal{I}(\varphi+\psi))_{(z_0,y)}$$ and
	$$(\pi_1^*(w^{\alpha}dw_1\wedge\ldots\wedge dw_{n_1}) \wedge \pi_2^*(F_\alpha),(z_0,y))\in (\mathcal{O}(K_M)\otimes\mathcal{I}(\varphi+\psi))_{(z_0,y)}$$
	 for any $y\in Y$ and any $\alpha\in R$. It follows from equality \eqref{eq:1216g}, equality \eqref{eq:1216h} and Lemma \ref{l:closedness} that
	 \begin{displaymath}
	 	\begin{split}
	 		(f-F_t,(z_0,y))&=\left(\sum_{\alpha\in R}\pi_1^*(w^{\alpha}dw_1\wedge\ldots\wedge dw_{n_1}) \wedge \pi_2^*(f_\alpha-F_{\alpha}),(z_0,y)\right)\\
	 		&\in(\mathcal{O}(K_M)\otimes\mathcal{I}(\varphi+\psi))_{(z_0,y)}	 	\end{split}
	 \end{displaymath}
holds for any $y\in Y$. Hence we have $(F_t-f,z)\in(\mathcal{O}(K_M)\otimes\mathcal{I}(\varphi+\psi))_{z}$ for any $z\in Z_0$, which implies that $G(t)\le \int_{\{\psi<-t\}}|F_t|^2e^{-\varphi}c(-\psi)=\tilde G(t)$. Thus, we obtain that $G(t)=\tilde G(t)=\int_{\{\psi<-t\}}|F_t|^2e^{-\varphi}c(-\psi)$.
 \end{proof}

The following two lemmas will be used in the proof of Lemma \ref{l:limit}.
\begin{Lemma}[see \cite{GY-concavity4}]
	\label{l:c(t)e^{-at}}
	Let $c(t)$ be a positive measurable function on $(0,+\infty)$, and let $a\in\mathbb{R}$. Assume that $\int_{t}^{+\infty}c(s)e^{-s}ds\in(0,+\infty)$ when $t$ near $+\infty$. Then we have
	
	$(1)$ $\lim_{t\rightarrow+\infty}\frac{\int_{t}^{+\infty}c(s)e^{-as}ds}{\int_t^{+\infty}c(s)e^{-s}ds}=1$ if and only if $a=1$;
	
	$(2)$ $\lim_{t\rightarrow+\infty}\frac{\int_{t}^{+\infty}c(s)e^{-as}ds}{\int_t^{+\infty}c(s)e^{-s}ds}=0$ if and only if $a>1$;
	
		$(3)$ $\lim_{t\rightarrow+\infty}\frac{\int_{t}^{+\infty}c(s)e^{-as}ds}{\int_t^{+\infty}c(s)e^{-s}ds}=+\infty$ if and only if $a<1$.
\end{Lemma}
\begin{proof} For the convenience of the reader, we recall the proof.

	 If $a=1$, it clear that $\lim_{t\rightarrow+\infty}\frac{\int_{t}^{+\infty}c(s)e^{-as}ds}{\int_t^{+\infty}c(s)e^{-s}ds}=1$.
	
	 If $a>1$, then $c(s)e^{-as}\le e^{(1-a)s_0} c(s)e^{-s}$ for $s\ge s_0>0$, which implies that $\limsup_{t\rightarrow+\infty}\frac{\int_{t}^{+\infty}c(s)e^{-as}ds}{\int_t^{+\infty}c(s)e^{-s}ds}\le e^{(1-a)s_0}$. Let $s_0\rightarrow+\infty$, we have $\lim_{t\rightarrow+\infty}\frac{\int_{t}^{+\infty}c(s)e^{-as}ds}{\int_t^{+\infty}c(s)e^{-s}ds}=0$
	
If $a<1$, then $c(s)e^{-as}\ge e^{(1-a)s_0} c(s)e^{-s}$ for $a>s_0>0$, which implies that $\liminf_{t\rightarrow+\infty}\frac{\int_{t}^{+\infty}c(s)e^{-as}ds}{\int_t^{+\infty}c(s)e^{-s}ds}\ge e^{(1-a)s_0}$. Let $s_0\rightarrow+\infty$, we have $\lim_{t\rightarrow+\infty}\frac{\int_{t}^{+\infty}c(s)e^{-as}ds}{\int_t^{+\infty}c(s)e^{-s}ds}=+\infty$.
\end{proof}

The following Lemma belongs to Fornaess
and Narasimhan on approximation property of plurisubharmonic
functions of Stein manifolds.

\begin{Lemma}[\cite{FN1980}]
\label{l:FN1} Let $X$ be a Stein manifold and $\varphi \in PSH(X)$. Then there exists a sequence
$\{\varphi_{n}\}_{n=1,\cdots}$ of smooth strongly plurisubharmonic functions such that
$\varphi_{n} \downarrow \varphi$.
\end{Lemma}

 It follows from Lemma \ref{l:green-sup} and Lemma \ref{l:green-sup2} that there eixsts a local coordinate $w_{j,k}$  on a neighborhood $V_{z_{j,k}}\Subset\Omega_{j}$ of $z_{j,k}\in\Omega_j$ satisfying $w_{j,k}(z_{j,k})=0$ and
$$\log|w_{j,k}|=\frac{1}{p_{j,k}}\sum_{1\le k<\tilde m_j}p_{j,k}G_{\Omega_j}(\cdot,z_{j,k})$$
 for any $j\in\{1,\ldots,n_1\}$ and $1\le k<\tilde{m}_j$, where $V_{z_{j,k}}\cap V_{z_{j,k'}}=\emptyset$ for any $j$ and $k\not=k'$. Denote that $\tilde I_1:=\{(\beta_1,\ldots,\beta_{n_1}):1\le \beta_j< \tilde m_j$ for any $j\in\{1,\ldots,n_1\}\}$, $V_{\beta}:=\prod_{1\le j\le n_1}V_{z_{j,\beta_j}}$ for any $\beta=(\beta_1,\ldots,\beta_n)\in\tilde I_1$ and $w_{\beta}:=(w_{1,\beta_1},\ldots,w_{n,\beta_n})$ is a local coordinate on $V_{\beta}$ of $z_{\beta}:=(z_{1,\beta_1},\ldots,z_{n,\beta_n})\in M$. Let
 $$\psi_1=\max_{1\le j\le n_1}\left\{\tilde\pi_{j}^*\left(2\sum_{1\le k<\tilde{m}_j}p_{j,k}G_{\Omega_j}(\cdot,z_{j,k})\right)\right\}$$
  on $\prod_{1\le j\le n_1}\Omega_j$, where $\tilde\pi_j$ is the natural projection from $\prod_{1\le j\le n_1}\Omega_j$ to $\Omega_j$. Note that $\psi=\pi_1^*(\psi_1)$.

Let $F$ be a holomorphic $(n,0)$ form on $\{\psi<-t_0\}\subset M$ for some $t_0>0$ satisfying $\int_{\{\psi<-t_0\}}|F|^2e^{-\varphi}c(-\psi)<+\infty$. Without loss of generality, we can assume $\cup_{\beta\in \tilde I_1}V_{\beta}\times Y\subset\{\psi<-t_0\}$. For any $\beta\in \tilde{I}_1$, it follows from Lemma \ref{l:fibra-decom} that there exists a sequence of  holomorphic  $(n_2,0)$  forms $\{F_{\alpha,\beta}\}_{\alpha\in\mathbb{Z}_{\ge0}^{n_1}}$ on $Y$ such that
$$F=\sum_{\alpha\in\mathbb{Z}_{\ge0}^{n_1}}\pi_1^*(w_{\beta}^{\alpha}dw_{1,\beta_1}\wedge\ldots\wedge dw_{n_1,\beta_{n_1}})\wedge\pi_2^*(F_{\alpha,\beta})$$
on $V_{\beta}\times Y$ and
$$\int_{Y}|F_{\alpha,\beta}|^2e^{-\varphi_Y}<+\infty$$
for any $\alpha\in\mathbb{Z}_{\ge0}^{n_1}$. Denote that $E_{\beta}:=\left\{\alpha\in\mathbb{Z}_{\ge0}^{n_1}:\sum_{1\le j\le n_1}\frac{\alpha_j+1}{p_{j,\beta_j}}=1\right\}$, $E_{1,\beta}:=\left\{\alpha\in\mathbb{Z}_{\ge0}^{n_1}:\sum_{1\le j\le n_1}\frac{\alpha_j+1}{p_{j,\beta_j}}<1\right\}$ and $E_{2,\beta}:=\left\{\alpha\in\mathbb{Z}_{\ge0}^{n_1}:\sum_{1\le j\le n_1}\frac{\alpha_j+1}{p_{j,\beta_j}}>1\right\}$.
\begin{Lemma}
	\label{l:limit}If $\liminf_{t\rightarrow+\infty}\frac{\int_{\{\psi<-t\}}|F|^2e^{-\varphi}c(-\psi)}{\int_t^{+\infty}c(s)e^{-s}ds}<+\infty$, we have $F_{\alpha,\beta}\equiv0$ for any $\alpha\in E_{1,\beta}$ and $\beta\in\tilde I_1$, and
	$$\liminf_{t\rightarrow+\infty}\frac{\int_{\{\psi<-t\}}|F|^2e^{-\varphi}c(-\psi)}{\int_t^{+\infty}c(s)e^{-s}ds}\ge\sum_{\beta\in\tilde I_1}\sum_{\alpha\in E_{\beta}}\frac{(2\pi)^{n_1}e^{-\sum_{1\le j\le n_1}\varphi_j(z_{j,\beta_j})}}{\prod_{1\le j\le n_1}(\alpha_j+1)}\int_{Y}|F_{\alpha,\beta}|^2e^{-\varphi_Y}.$$
\end{Lemma}
\begin{proof}
	 As $\sum_{1\le j\le n_1}\tilde\pi_j^*(\varphi_j)$ is a plurisubharmonic function on $\prod_{1\le j\le n_1}\Omega_j$, it follows from Lemma \ref{l:FN1} that  there exists a sequence of smooth plurisubharmonic functions $\{\Phi_l\}_{l\in\mathbb{Z}_{\ge0}}$ on $\prod_{1\le j\le n_1}\Omega_j$, which is decreasingly convergent to $\sum_{1\le j\le n_1}\tilde\pi_j^*(\varphi_j)$.
	
Let $\beta\in \tilde{I}_1$ and $l\in\mathbb{Z}_{\ge0}$. For any $\epsilon>0$, there exists $t_{\beta}>t_{0}$ such that $\{\psi_1<-t_{\beta}\}\cap V_{\beta}\Subset V_{\beta}$ and
$$\sup_{z\in\{\psi_1<-t_{\beta}\}\cap V_{\beta}}|\Phi_l(z)-\Phi(z_{\beta})|<\epsilon.$$
For any $t\ge t_{\beta}$, note that $\{\psi_1<-t\}=\prod_{1\le j\le n_1}\left\{|w_{j,\beta_j}|<e^{-\frac{t}{2p_{j,\beta_j}}}\right\}$ and $F=\sum_{\alpha\in\mathbb{Z}_{\ge0}^{n_1}}\pi_1^*(w_{\beta}^{\alpha}dw_{1,\beta_1}\wedge\ldots\wedge dw_{n_1,\beta_{n_1}})\wedge\pi_2^*(F_{\alpha,\beta})$ on $\{\psi_1<-t\}\times Y$, then we have
\begin{equation}\begin{split}
	\label{eq:1218a}&\int_{\{\psi<-t\}\cap(V_{\beta}\times Y)}|F|^2e^{-\varphi}c(-\psi)\\
	\ge&\int_{\{\psi<-t\}\cap(V_{\beta}\times Y)}|F|^2e^{-\pi_1^*(\Phi_l)-\pi_2^*(\varphi_Y)}c(-\psi)	\\
	\ge& e^{-\Phi_l(z_{\beta})-\epsilon}\int_{\left(\{\psi_1<-t\}\cap V_{\beta}\right)\times Y}|F|^2e^{-\pi_2^*(\varphi_Y)}c(-\pi_1^*(\psi_1))\\
	=&e^{-\Phi_l(z_{\beta})-\epsilon}\sum_{\alpha\in\mathbb{Z}_{\ge0}^{n_1}}\int_{\{\psi_1<-t\}}|w_{\beta}^{\alpha}dw_{1,\beta_1}\wedge\ldots\wedge dw_{n_1,\beta_{n_1}}|^2c(-\psi)\times\int_Y|F_{\alpha,\beta}|^2e^{-\varphi_Y}.
\end{split}\end{equation}
Denote that $q_{\alpha}:=\sum_{1\le j\le n_1}\frac{\alpha_j+1}{p_{j,\beta_j}}-1$. It follows from Lemma \ref{l:m1} and inequality \eqref{eq:1218a} that
\begin{equation*}
	\begin{split}
		&\int_{\{\psi<-t\}\cap(V_{\beta}\times Y)}|F|^2e^{-\varphi}c(-\psi)\\
		\ge& e^{-\Phi_l(z_{\beta})-\epsilon}\sum_{\alpha\in\mathbb{Z}_{\ge0}^{n_1}}\left(\int_t^{+\infty}c(s)e^{-(q_{\alpha}+1)s}ds\right)\frac{(q_{\alpha}+1)(2\pi)^{n_1}}{\prod_{1\le j\le n_1}(\alpha_j+1)}\int_Y|F_{\alpha,\beta}|^2e^{-\varphi_Y}.
	\end{split}
\end{equation*}
It follows from $\liminf_{t\rightarrow+\infty}\frac{\int_{\{\psi<-t\}}|F|^2e^{-\varphi}c(-\psi)}{\int_t^{+\infty}c(s)e^{-s}ds}<+\infty$ and Lemma \ref{l:c(t)e^{-at}} that
$$F_{\alpha,\beta}\equiv0$$
 for any $\alpha$ satisfying $q_{\alpha}<0$ and
\begin{equation*}
\liminf_{t\rightarrow+\infty}\frac{\int_{\{\psi<-t\}\cap(V_{\beta}\times Y)}|F|^2e^{-\varphi}c(-\psi)}{\int_t^{+\infty}c(s)e^{-s}ds}\ge e^{-\Phi_l(z_{\beta})-\epsilon}\sum_{\alpha\in E_{\beta}}\frac{(2\pi)^{n_1}\int_Y|F_{\alpha,\beta}|^2e^{-\varphi_Y}}{\prod_{1\le j\le n_1}(\alpha_j+1)}.
\end{equation*}
Letting $\epsilon\rightarrow 0$ and $l\rightarrow+\infty$, we have
\begin{equation}
	\label{eq:1218b}
	\begin{split}
		&\liminf_{t\rightarrow+\infty}\frac{\int_{\{\psi<-t\}\cap(V_{\beta}\times Y)}|F|^2e^{-\varphi}c(-\psi)}{\int_t^{+\infty}c(s)e^{-s}ds}\\
		\ge&\sum_{\alpha\in E_{\beta}}\frac{(2\pi)^{n_1}e^{-\sum_{1\le j\le n_1}\varphi_j(z_{j,\beta_j})}}{\prod_{1\le j\le n_1}(\alpha_j+1)}\int_Y|F_{\alpha,\beta}|^2e^{-\varphi_Y}.
			\end{split}
\end{equation}

Note that $V_{\beta}\cap V_{\tilde\beta}=\emptyset$ for any $\beta\not=\tilde\beta$ and $\{\psi_1<-t_{\beta}\}\cap V_{\beta} \Subset V_{\beta}$ for any $\beta\in\tilde I_1$. It follows from inequality \eqref{eq:1218b} that
$$\liminf_{t\rightarrow+\infty}\frac{\int_{\{\psi<-t\}}|F|^2e^{-\varphi}c(-\psi)}{\int_t^{+\infty}c(s)e^{-s}ds}\ge\sum_{\beta\in\tilde I_1}\sum_{\alpha\in E_{\beta}}\frac{(2\pi)^{n_1}e^{-\sum_{1\le j\le n_1}\varphi_j(z_{j,\beta_j})}}{\prod_{1\le j\le n_1}(\alpha_j+1)}\int_{Y}|F_{\alpha,\beta}|^2e^{-\varphi_Y}.$$
Thus, Lemma \ref{l:limit} holds.
\end{proof}

Let $M_1$ be an open complex submanifold of $M$ satisfying that $Z_0=\{z_{\beta}:\beta\in\tilde I_1\}\times Y\subset M_1$, and let $K_{M_1}$  be the  canonical (holomorphic) line bundle on  $M_1$. Let $F_1$ be a holomorphic $(n,0)$ form on $\{\psi<-t_0\}\cap M_1$ for $t_0>0$ satisfying that $\int_{\{\psi<-t_0\}\cap M_1}|F_1|^2e^{-\varphi}c(-\psi)<+\infty$. For any $\beta\in\tilde I_1$, it follows from Lemma \ref{l:fibra-decom-2} that there exist a sequence of holomorphic $(n_2,0)$ forms $\{F_{\alpha,\beta}\}_{\alpha\in\mathbb{Z}_{\ge0}^{n_1}}$ on $Y$ and an open subset $U_{\beta}$ of $\{\psi<-t_0\}\cap M_1\cap (V_{\beta}\times Y)$ such that
$$F_1=\sum_{\alpha\in\mathbb{Z}_{\ge0}^{n_1}}\pi_1^*(w_{\beta}^{\alpha}dw_{1,\beta_1}\wedge\ldots\wedge dw_{n_1,\beta_{n_1}})\wedge\pi_2^*(F_{\alpha,\beta})$$
on $U_{\beta}$ and
$$\int_{K}|F_{\alpha,\beta}|^2e^{-\varphi_Y}<+\infty$$
for any $\alpha\in\mathbb{Z}_{\ge0}^{n_1}$ and compact subset $K$ of $Y$.
\begin{Lemma}
	\label{l:limit2}If $\liminf_{t\rightarrow+\infty}\frac{\int_{\{\psi<-t\}\cap M_1}|F|^2e^{-\varphi}c(-\psi)}{\int_t^{+\infty}c(s)e^{-s}ds}<+\infty$, we have $F_{\alpha,\beta}\equiv0$ for any $\alpha\in E_{1,\beta}$ and $\beta\in\tilde I_1$ and
	\begin{displaymath}
		\begin{split}
			&\liminf_{t\rightarrow+\infty}\frac{\int_{\{\psi<-t\}\cap M_1}|F|^2e^{-\varphi}c(-\psi)}{\int_t^{+\infty}c(s)e^{-s}ds}\\
			\ge&\sum_{\beta\in\tilde I_1}\sum_{\alpha\in E_{\beta}}\frac{(2\pi)^{n_1}e^{-\sum_{1\le j\le n_1}\varphi_j(z_{j,\beta_j})}}{\prod_{1\le j\le n_1}(\alpha_j+1)}\int_{Y}|F_{\alpha,\beta}|^2e^{-\varphi_Y}.
		\end{split}
	\end{displaymath}
	\end{Lemma}
\begin{proof} Note that $\psi_1=\max_{1\le j\le n_1}\left\{\tilde\pi_j^*\left(2\sum_{1\le k<\tilde m_j}G_{\Omega_j}(\cdot,z_{j,k})\right)\right\}$ on $\prod_{1\le j\le n_1}\Omega_j$.
	For any $\beta\in\tilde I_1$ and any open subset $V$ of $Y$ satisfying $V\Subset Y$, it follows from Lemma \ref{l:green-sup} and Lemma \ref{l:green-sup2} that there exists $t_{\beta,V}>t_0$ such that  $\{\psi_1<-t_{\beta,V}\}\times V\Subset U_{\beta}$. $\liminf_{t\rightarrow+\infty}\frac{\int_{\{\psi<-t\}\cap M_1}|F|^2e^{-\varphi}c(-\psi)}{\int_t^{+\infty}c(s)e^{-s}ds}<+\infty$ implies that
	\begin{equation}
		\label{eq:1224a}\liminf_{t\rightarrow+\infty}\frac{\int_{\{\psi_1<-t\}\times V}|F|^2e^{-\varphi}c(-\psi)}{\int_t^{+\infty}c(s)e^{-s}ds}<+\infty.
	\end{equation}
	 It follows from equality \eqref{eq:1224a} and Lemma \ref{l:limit} that $F_{\alpha,\beta}\equiv0$ on $V$ for any $\alpha\in E_{1,\beta}$ and
	 \begin{displaymath}
	 	\begin{split}
	 		&\liminf_{t\rightarrow+\infty}\frac{\int_{\{\psi<-t\}\cap U_{\beta}}|F|^2e^{-\varphi}c(-\psi)}{\int_t^{+\infty}c(s)e^{-s}ds}\\
	 		\ge&\liminf_{t\rightarrow+\infty}\frac{\int_{\{\psi_1<-t\}\times V}|F|^2e^{-\varphi}c(-\psi)}{\int_t^{+\infty}c(s)e^{-s}ds}\\
	 		\ge&\sum_{\alpha\in E_{\beta}}\frac{(2\pi)^{n_1}e^{-\sum_{1\le j\le n_1}\varphi_j(z_{j,\beta_j})}}{\prod_{1\le j\le n_1}(\alpha_j+1)}\int_{V}|F_{\alpha,\beta}|^2e^{-\varphi_Y}.
	 	\end{split}
	 \end{displaymath}
Following from the arbitrariness of $V$, we have
$$F_{\alpha,\beta}\equiv0$$
 on $Y$ for any $\alpha\in E_{1,\beta}$ and
	 \begin{equation}\label{eq:1224b}
	 	\begin{split}
	 		&\liminf_{t\rightarrow+\infty}\frac{\int_{\{\psi<-t\}\cap U_{\beta}}|F|^2e^{-\varphi}c(-\psi)}{\int_t^{+\infty}c(s)e^{-s}ds}\\
	 		\ge&\sum_{\alpha\in E_{\beta}}\frac{(2\pi)^{n_1}e^{-\sum_{1\le j\le n_1}\varphi_j(z_{j,\beta_j})}}{\prod_{1\le j\le n_1}(\alpha_j+1)}\int_{Y}|F_{\alpha,\beta}|^2e^{-\varphi_Y}.
	 	\end{split}
	 \end{equation}
 $V_{\beta}\cap V_{\beta'}=\emptyset$ for any $\beta\not=\beta'$ implies that $U_{\beta}\cap U_{\beta'}=\emptyset$ for any $\beta\not=\beta'$. It follows from inequality \eqref{eq:1224b} that
 \begin{displaymath}
		\begin{split}
			&\liminf_{t\rightarrow+\infty}\frac{\int_{\{\psi<-t\}\cap M_1}|F|^2e^{-\varphi}c(-\psi)}{\int_t^{+\infty}c(s)e^{-s}ds}\\
			\ge&\sum_{\beta\in\tilde I_1}\sum_{\alpha\in E_{\beta}}\frac{(2\pi)^{n_1}e^{-\sum_{1\le j\le n_1}\varphi_j(z_{j,\beta_j})}}{\prod_{1\le j\le n_1}(\alpha_j+1)}\int_{Y}|F_{\alpha,\beta}|^2e^{-\varphi_Y}.
		\end{split}
	\end{displaymath}
	Thus, Lemma \ref{l:limit2} holds.
\end{proof}

In the following, we consider the case that $Z_j$ is a single point set.
Let $M'=\prod_{1\le j\le n_1}\Omega_j$ be an $n_1-$dimensional complex manifold, and let $K_{M'}$ be the canonical (holomorphic) line bundle on $M'$.
Let $z_j\in\Omega_j$ and $z_0=(z_1,\ldots,z_{n_1})\in M'$. Let $\varphi_j$ be subharmonic functions on $\Omega_j$ such that $\varphi_j(z_j)>-\infty$. Denote that
$$\psi_1:=\max_{1\le j\le n_1}\left\{2p_j\tilde\pi_j^{*}(G_{\Omega_j}(\cdot,z_j))\right\}$$
 and $\tilde \varphi:=\sum_{1\le j\le n_1}\tilde\pi_j^*(\varphi_j)$ on $M'$, where $p_j$ is a positive real number for $1\le j\le n_1$ and $\tilde\pi_j$ is the natural projection from $M'$ to $\Omega_j$.

Let $w_j$ be a local coordinate on a neighborhood $V_{z_j}$ of $z_j\in\Omega_j$ satisfying $w_j(z_j)=0$. Denote that $V_0:=\prod_{1\le j\le n_1}V_{z_j}$, and $w:=(w_1,\ldots,w_{n_1})$ is a local coordinate on $V_0$ of $z_0\in M'$.
Take $E=\left\{(\alpha_1,\ldots,\alpha_{n_1}):\sum_{1\le j\le n_1}\frac{\alpha_j+1}{p_{j}}=1\,\&\,\alpha_j\in\mathbb{Z}_{\ge0}\right\}$.

Let $c_j(z)$ be the logarithmic capacity (see \cite{S-O69}) on $\Omega_j$, which is locally defined by
$$c_j(z_j):=\exp\lim_{z\rightarrow z_j}(G_{\Omega_j}(z,z_j)-\log|w_j(z)|).$$
\begin{Lemma}[see \cite{GMY-concavity2}]
 	\label{p:exten-pro-finite} Let $c(t)$ be a positive function on $(0,+\infty)$ satisfying that $c(t)e^{-t}$ is decreasing and $\int_0^{+\infty}c(s)e^{-s}ds<+\infty$. For any $\alpha\in E$, there exists a holomorphic $(n_1,0)$ form $F$ on $M'$, which satisfies that $(F-w^{\alpha}dw_1\wedge\ldots\wedge dw_{n_1},z_0)\in(\mathcal{O}(K_{\Omega_j})\otimes\mathcal{I}(\psi_1))_{z_0}$  and
 	\begin{displaymath}
 		\int_{M'}|F|^2e^{-\tilde\varphi}c(-\psi_1)\le\left(\int_0^{+\infty}c(s)e^{-s}ds\right)\frac{(2\pi)^{n_1}e^{-\tilde\varphi(z_{\beta})}}{\Pi_{1\le j\le n_1}(\alpha_j+1)c_{j}(z_j)^{2\alpha_{j}+2}}.
 	\end{displaymath}
 \end{Lemma}

As $\varphi_j$ is subharmonic on $\Omega_j$, it follows from Lemma \ref{p:exten-pro-finite} and Lemma \ref{lem:A} that there exists a holomorphic  $(1,0)$ form $f_{j,\alpha_j}$ on $\Omega_j$ such that $(f_{j,\alpha_j}-w_j^{\alpha_j}dw_j,z_j)\in(\mathcal{O}(K_{\Omega_j})\otimes \mathcal{I}(2(\alpha_j+1)G_{\Omega_j}(\cdot,z_j)))_{z_j}$ and $\int_{\Omega_j}|f_{j,\alpha_j}|^2e^{-\varphi_j}=\inf\big\{\int_{\Omega_j}|\tilde{f}|^2e^{-\varphi_j}:\tilde{f}\in H^0(\Omega_j,\mathcal{O}(K_{\Omega_j}))$ $\&\,(\tilde{f}-w_j^{\alpha_j}dw_j,z_j)\in(\mathcal{O}(K_{\Omega_j})\otimes \mathcal{I}(2(\alpha_j+1)G_{\Omega_j}(\cdot,z_j)))_{z_j}\big\}<+\infty$ for any $\alpha\in E$ and $j\in\{1,\ldots,n_1\}$.

\begin{Lemma}[see \cite{GY-concavity4}]
	\label{l:orth1}$F=\sum_{\alpha\in E}d_{\alpha}\prod_{1\le j\le n_1}\tilde\pi_{j}^*(f_{j,\alpha_j})$ is a holomorphic $(n_1,0)$ form on $M'$ satisfying that $({F}-\sum_{\alpha\in E}d_{\alpha}w^{\alpha}dw_1\wedge \ldots\wedge dw_{n_1} ,z_0)\in\mathcal{O}(K_{M'})\otimes\mathcal{I}(\psi_1))_{z_0}$,
	$$\int_{M'}|F|^2e^{-\tilde\varphi}=\sum_{\alpha\in E}|d_{\alpha}|^2\int_{M'}\bigg|\prod_{1\le j\le n_1}\pi_{j}^*(f_{j,\alpha_j})\bigg|^2e^{-\tilde\varphi}$$
	 and $\int_{M'}|F|^2e^{-\tilde\varphi}=\inf\big\{\int_{M'}|\tilde{F}|^2e^{-\tilde\varphi}:\tilde{F}$ is a holomorphic $(n_1,0)$ form on $M'$ such that $(\tilde{F}-\sum_{\alpha\in E}d_{\alpha}w^{\alpha}dw_1\wedge \ldots\wedge dw_{n_1} ,z_0)\in\mathcal{O}(K_{M'})\otimes\mathcal{I}(\psi_1))_{z_0}\big\}$, where $d_{\alpha}$ is a constant for any $\alpha\in E$.
\end{Lemma}

Let $\varphi_Y$ be a plurisubharmonic function on $Y$. Let $f_{\alpha}$ be a holomorphic $(n_2,0)$ form on $Y$ satisfying $\int_Y|f_{\alpha}|^2e^{-\varphi_Y}<+\infty$ for any $\alpha\in E$. Let $f=\sum_{\alpha\in E}\pi_1^*(w^{\alpha}dw_1\wedge \ldots\wedge dw_{n_1})\wedge \pi_2^*(f_{\alpha})$ be a holomorphic $(n,0)$ form  on $V_0\times Y\subset M=M'\times Y$. Denote that $\varphi:=\pi_1^*(\tilde\varphi)+\pi_2^*(\varphi_Y)$ and $\psi:=\pi_1^*(\psi_1)$ on $M$.
\begin{Lemma}
	\label{l:orth2}$F=\sum_{\alpha\in E}\pi_{1,1}^*(f_{1,\alpha_1})\wedge\ldots\wedge\pi_{1,n_1}^*(f_{n_1,\alpha_{n_1}})\wedge \pi_2^*(f_{\alpha})$ is a holomorphic $(n,0)$ form on $M$, and satisfies that $(F-f,(z_0,y))\in(\mathcal{O}(K_M)\otimes\mathcal{I}(\psi))_{(z_0,y)}$ for any $y\in Y$,
	$$\int_{M}|F|^2e^{-\varphi}=\sum_{\alpha\in E}\left(\int_Y|f_{\alpha}|^2e^{-\varphi_Y}\right)\prod_{1\le j\le n_1}\int_{\Omega_j}|f_{j,\alpha_j}|^2e^{-\varphi_j}$$
	and $\int_{M}|F|^2e^{-\varphi}=\inf\big\{\int_{M}|\tilde{F}|^2e^{-\varphi}:\tilde{F}$ is a holomorphic $(n,0)$ form on $M$ such that $(\tilde{F}-f,(z_0,y))\in\mathcal{O}(K_{M})\otimes\mathcal{I}(\psi))_{(z_0,y)}$ for any $y\in Y\big\}$.
\end{Lemma}
\begin{proof}
	It follows from Lemma \ref{l:0'} that $(f,(z_0,y))\in\mathcal{I}(\psi)_{(z_0,y)}$ for any $y\in Y$ if and only if $(f(\cdot,y),z_0)\in\mathcal{I}(\psi_1)_{z_0}$. For any $\alpha\in E$, using Proposition \ref{p:fibra} and Lemma \ref{l:orth1}, we obtain that $F_{\alpha}=\pi_{1,1}^*(f_{1,\alpha_1})\wedge\ldots\wedge\pi_{1,n_1}^*(f_{n_1,\alpha_{n_1}})\wedge \pi_2^*(f_{\alpha})$ satisfies that $\int_M|F_{\alpha}|^2e^{-\varphi}=\left(\int_Y|f_{\alpha}|^2e^{-\varphi_Y}\right)\prod_{1\le j\le n_1}\int_{\Omega_j}|f_{j,\alpha_j}|^2e^{-\varphi_j}=\inf\big\{\int_{M}|\tilde{F}|^2e^{-\varphi}:\tilde{F}$ is a holomorphic $(n,0)$ form on $M$ such that $(\tilde{F}-\pi_1^*(w^{\alpha}dw_1\wedge \ldots\wedge dw_{n_1})\wedge \pi_2^*(f_{\alpha}),(z_0,y))\in\mathcal{O}(K_{M})\otimes\mathcal{I}(\psi))_{(z_0,y)}$ for any $y\in Y\big\}$, i.e.
	\begin{equation}
		\label{eq:211218a}\int_MF_{\alpha}\wedge\overline{\tilde F}e^{-\varphi}=0
	\end{equation}
	for any holomorphic $(n,0)$ form $\tilde F$ satisfying $\int_{M}|\tilde F|^2e^{-\varphi}<+\infty$ and $(\tilde F,(z_0,y))\in(\mathcal{O}(K_{M})\otimes\mathcal{I}(\psi))_{(z_0,y)}$ for any $y\in Y$. It follows from the Fubini's theorem and Lemma \ref{l:orth1} that
	\begin{equation}
		\label{eq:211218b}
		\int_{M}F_{\alpha}\wedge\overline{F_{\tilde\alpha}}e^{-\varphi}=0
	\end{equation}
	for any $\alpha\not=\tilde\alpha$. Note that $F=\sum_{\alpha\in E}F_{\alpha}$ and $(F-f,(z_0,y))\in(\mathcal{O}(K_M)\otimes\mathcal{I}(\psi))_{(z_0,y)}$ for any $y\in Y$. It follows from equality \eqref{eq:211218a} and equality \eqref{eq:211218b} that
	$$\int_{M}|F|^2e^{-\varphi}=\sum_{\alpha\in E}\left(\int_Y|f_{\alpha}|^2e^{-\varphi_Y}\right)\prod_{1\le j\le n_1}\int_{\Omega_j}|f_{j,\alpha_j}|^2e^{-\varphi_j}$$
	and $\int_{M}|F|^2e^{-\varphi}=\inf\big\{\int_{M}|\tilde{F}|^2e^{-\varphi}:\tilde{F}$ is a holomorphic $(n,0)$ form on $M$ such that $(\tilde{F}-f,(z_0,y))\in\mathcal{O}(K_{M})\otimes\mathcal{I}(\psi))_{(z_0,y)}$ for any $y\in Y\big\}$.
	\end{proof}

Let $X$ be an $n_1-$dimensional complex manifold, and let $Y$ be an $n_2-$dimensional complex manifold. Let $M=X\times Y$ be an $n-$dimensional complex manifold, where $n=n_1+n_2$. Let $\pi_1$ and $\pi_2$ be the natural projections from $M$ to $X$ and $Y$ respectively. We recall the following lemma.
\begin{Lemma}[see \cite{BGY-concavity5}]
	\label{l:decom-product}Let $F\not\equiv0$ be a holomorphic $(n,0)$ form on $M$. Let $f_1$ be a holomorphic $(n_1,0)$ form on an open subset $U$ of $X$, and let $f_2$ be a holomorphic $(n_2,0)$ form on an open subset $V$ of $Y$. If
	$$F=\pi_1^*(f_1)\wedge\pi_2^*(f_2)$$
	on $U\times V$, there exist a holomorphic $(n_1,0)$ form $F_1$ on $X$ and a holomorphic $(n_2,0)$ form $F_2$ on $Y$ such that $F_1=f_1$ on $U$, $F_2=f_2$ on $V$, and
	$$F=\pi_1^*(F_1)\wedge\pi_2^*(F_2)$$
	on $M$.
\end{Lemma}

\subsection{Optimal jets $L^2$ extension}
\

In this section, we give an optimal jets $L^2$ extension result, i.e. Proposition \ref{p:exten-fibra}.
We recall two lemmas, which will be used in the proof of Proposition \ref{p:exten-fibra}.

\begin{Lemma}[\cite{GMY-concavity2}] \label{lem:L2} Let $c$ be a positive function on $(0,+\infty)$, such that $\int_{0}^{+\infty}c(t)e^{-t}dt<+\infty$ and $c(t)e^{-t}$ is decreasing on $(0,+\infty)$.
Let $B\in(0,+\infty)$ and $t_{0}\geq 0$ be arbitrarily given.
Let $M$ be an $n-$dimensional weakly pseudoconvex K\"ahler  manifold.
Let $\psi<0$ be a plurisubharmonic function on $M$.
Let $\varphi$ be a plurisubharmonic function on $M$.
Let $F$ be a holomorphic $(n,0)$ form on $\{\psi<-t_{0}\}$,
such that
\begin{equation}
\label{equ:20171124a}
\int_{K\cap\{\psi<-t_{0}\}}|F|^{2}<+\infty
\end{equation}
for any compact subset $K$ of $M$, and
\begin{equation}
\label{equ:20171122a}
\int_{M}\frac{1}{B}\mathbb{I}_{\{-t_{0}-B<\psi<-t_{0}\}}|F|^{2}e^{-\varphi}\leq C<+\infty.
\end{equation}
Then there exists a
holomorphic $(n,0)$ form $\tilde{F}$ on $M$, such that
\begin{equation}
\label{equ:3.4}
\begin{split}
\int_{M}&|\tilde{F}-(1-b_{t_0,B}(\psi))F|^{2}e^{-\varphi+v_{t_0,B}(\psi)}c(-v_{t_0,B}(\psi))\leq C\int_{0}^{t_{0}+B}c(t)e^{-t}dt
\end{split}
\end{equation}
where $b_{t_0,B}(t)=\int_{-\infty}^{t}\frac{1}{B}\mathbb{I}_{\{-t_{0}-B< s<-t_{0}\}}ds$ and
$v_{t_0,B}(t)=\int_{-t_0}^{t}b_{t_0,B}(s)ds-t_0$.
\end{Lemma}

It is clear that $\mathbb{I}_{(-t_{0},+\infty)}\leq b_{t_0,B}(t)\leq\mathbb{I}_{(-t_{0}-B,+\infty)}$ and $\max\{t,-t_{0}-B\}\leq v_{t_0,B}(t) \leq\max\{t,-t_{0}\}$.

\begin{Lemma}[see \cite{GY-concavity}]
	\label{l:converge}
	Let $M$ be a complex manifold. Let $S$ be an analytic subset of $M$.  	
	Let $\{g_j\}_{j\in\mathbb{Z}_{\ge1}}$ be a sequence of nonnegative Lebesgue measurable functions on $M$, which satisfies that $g_j$ are almost everywhere convergent to $g$ on  $M$ when $j\rightarrow+\infty$,  where $g$ is a nonnegative Lebesgue measurable function on $M$. Assume that for any compact subset $K$ of $M\backslash S$, there exist $s_K\in(0,+\infty)$ and $C_K\in(0,+\infty)$ such that
	$$\int_{K}{g_j}^{-s_K}dV_M\leq C_K$$
	 for any $j$, where $dV_M$ is a continuous volume form on $M$.
	
Let $\{F_j\}_{j\in\mathbb{Z}_{\ge1}}$ be a sequence of holomorphic $(n,0)$ form on $M$. Assume that $\liminf_{j\rightarrow+\infty}\int_{M}|F_j|^2g_j\leq C$, where $C$ is a positive constant. Then there exists a subsequence $\{F_{j_l}\}_{l\in\mathbb{Z}_{\ge1}}$, which satisfies that $\{F_{j_l}\}$ is uniformly convergent to a holomorphic $(n,0)$ form $F$ on $M$ on any compact subset of $M$ when $l\rightarrow+\infty$, such that
$$\int_{M}|F|^2g\leq C.$$
\end{Lemma}

Let $\Omega_j$  be an open Riemann surface, which admits a nontrivial Green function $G_{\Omega_j}$ for any  $1\le j\le n_1$. Let $Y$ be an $n_2-$dimensional weakly pseudoconvex K\"ahler manifold, and let $K_Y$ be the canonical (holomorphic) line bundle on $Y$. Let
$M=\left(\prod_{1\le j\le n_1}\Omega_j\right)\times Y$
 be an $n-$dimensional complex manifold, where $n=n_1+n_2$, and let $K_M$ be the canonical (holomorphic) line bundle on $M$. Let $\pi_{1}$, $\pi_{1,j}$ and $\pi_2$ be the natural projections from $M$ to $\prod_{1\le j\le n_1}\Omega_j$, $\Omega_j$ and $Y$ respectively.
 Let $Z_j=\{z_{j,k}:1\le k<\tilde m_j\}$ be a discrete subset of $\Omega_j$ for any  $j\in\{1,\ldots,n_1\}$, where $\tilde m_j\in\mathbb{Z}_{\ge2}\cup\{+\infty\}$. Denote that $Z_0:=\left(\prod_{1\le j\le n_1}Z_j\right)\times Y$.

Let $\varphi_X$ be a plurisubharmonic function on $\prod_{1\le j\le n_1}\Omega_j$ satisfying that $\varphi_X(z)>-\infty$ for any $z\in\prod_{1\le j\le n_1}Z_j$, and let $\varphi_Y$ be a plurisubharmonic function on $Y$.
Let $p_{j,k}$ be a positive number for any $1\le j\le n_1$ and $1\le k<\tilde m_j$, which satisfies that $\sum_{1\le k<\tilde m_j}p_{j,k}G_{\Omega_j}(\cdot,z_{j,k})\not\equiv-\infty$ for any $1\le j\le n_1$.
Denote that
$$\psi:=\max_{1\le j\le n_1}\left\{2\sum_{1\le k<\tilde m_j}p_{j,k}\pi_{1,j}^{*}(G_{\Omega_j}(\cdot,z_{j,k}))\right\}$$
and
$$\varphi:=\pi_1^*(\varphi_X)+\pi_2^*(\varphi_Y)$$
 on $M$.

 Let $w_{j,k}$ be a local coordinate on a neighborhood $V_{z_{j,k}}\Subset\Omega_{j}$ of $z_{j,k}\in\Omega_j$ satisfying $w_{j,k}(z_{j,k})=0$ for any $1\le j\le n_1$ and $1\le k<\tilde m_j$, where $V_{z_{j,k}}\cap V_{z_{j,k'}}=\emptyset$ for any $j$ and $k\not=k'$. Denote that $\tilde I_1:=\{(\beta_1,\ldots,\beta_{n_1}):1\le \beta_j<\tilde m_j$ for any $j\in\{1,\ldots,n_1\}\}$, $V_{\beta}:=\prod_{1\le j\le n_1}V_{z_{j,\beta_j}}$  and $w_{\beta}:=(w_{1,\beta_1},\ldots,w_{n_1,\beta_{n_1}})$ is a local coordinate on $V_{\beta}$ of $z_{\beta}:=(z_{1,\beta_1},\ldots,z_{n_1,\beta_{n_1}})\in\prod_{1\le j\le n_1}\Omega_j$ for any $\beta=(\beta_1,\ldots,\beta_{n_1})\in\tilde I_1$.
 Denote that $E_{\beta}:=\left\{(\alpha_1,\ldots,\alpha_{n_1}):\sum_{1\le j\le n_1}\frac{\alpha_j+1}{p_{j,\beta_j}}=1\,\&\,\alpha_j\in\mathbb{Z}_{\ge0}\right\}$ and $\tilde E_{\beta}:=\left\{(\alpha_1,\ldots,\alpha_{n_1}):\sum_{1\le j\le n_1}\frac{\alpha_j+1}{p_{j,\beta_j}}\ge1\,\&\,\alpha_j\in\mathbb{Z}_{\ge0}\right\}$ for any $\beta\in\tilde I_1$.
Let $f$ be a holomorphic $(n,0)$ form on a neighborhood $U_0$ of $Z_0$ such that
$$f=\sum_{\alpha\in\tilde E_{\beta}}\pi_1^*(w_{\beta}^{\alpha}dw_{1,\beta_1}\wedge\ldots\wedge dw_{n_1,\beta_{n_1}})\wedge\pi_2^*(f_{\alpha,\beta})$$ on $U_0\cap(V_{\beta}\times Y)$, where $f_{\alpha,\beta}$ is a holomorphic $(n_2,0)$ form on $Y$ for any $\alpha\in\tilde E_{\beta}$ and $\beta\in\tilde I_1$. Denote that
\begin{equation*}
c_{j,k}:=\exp\lim_{z\rightarrow z_{j,k}}\left(\frac{\sum_{1\le k_1<\tilde m_j}p_{j,k_1}G_{\Omega_j}(z,z_{j,k_1})}{p_{j,k}}-\log|w_{j,k}(z)|\right)
\end{equation*}
 for any $j\in\{1,\ldots,n\}$ and $1\le k<\tilde m_j$ (following from Lemma \ref{l:green-sup} and Lemma \ref{l:green-sup2}, we get that the above limit exists).

 \begin{Proposition}
 	\label{p:exten-fibra}Let $c$ be a positive function on $(0,+\infty)$ such that $\int_{0}^{+\infty}c(t)e^{-t}dt<+\infty$ and $c(t)e^{-t}$ is decreasing on $(0,+\infty)$. Assume that
$$\sum_{\beta\in\tilde I_1}\sum_{\alpha\in E_{\beta}}\frac{(2\pi)^{n_1}e^{-\varphi_X(z_{\beta})}\int_Y|f_{\alpha,\beta}|^2e^{-\varphi_{Y}}}{\prod_{1\le j\le n_1}(\alpha_j+1)c_{j,\beta_j}^{2\alpha_{j}+2}}<+\infty.$$
Then there exists a holomorphic $(n,0)$ form $F$ on $M$ satisfying that $(F-f,z)\in(\mathcal{O}(K_M)\otimes\mathcal{I}(\psi))_{z}$ for any $z\in Z_0$ and
\begin{displaymath}
	\begin{split}
	\int_{M}|F|^2e^{-\varphi}c(-\psi)
\le\left(\int_0^{+\infty}c(s)e^{-s}ds\right)\sum_{\beta\in\tilde I_1}\sum_{\alpha\in E_{\beta}}\frac{(2\pi)^{n_1}e^{-\varphi_X(z_{\beta})}\int_Y|f_{\alpha,\beta}|^2e^{-\varphi_{Y}}}{\prod_{1\le j\le n_1}(\alpha_j+1)c_{j,\beta_j}^{2\alpha_{j}+2}}.	
	\end{split}
\end{displaymath}
\end{Proposition}

\begin{proof}
The following Remark shows that it suffices to prove Proposition \ref{p:exten-fibra} for the case $\tilde m_j<+\infty$ for any $j\in\{1,\ldots,n_1\}$.
\begin{Remark} Assume that Proposition \ref{p:exten-fibra} holds for the case $\tilde m_j<+\infty$ for any $j\in\{1,\ldots,n_1\}$.
For any $j\in\{1,\ldots,n_1\}$, it follows from Lemma \ref{l:green-approx} that there exists a sequence of Riemann surfaces $\{\Omega_{j,l}\}_{l\in\mathbb{Z}_{\ge1}}$, which satisfies that $\Omega_{j,l}\Subset\Omega_{j,l+1}\Subset\Omega_{j}$ for any $l$, $\cup_{l\in\mathbb{Z}_{\ge1}}\Omega_{j,l}=\Omega_j$ and $\{G_{\Omega_{j,l}}(\cdot,z)-G_{\Omega_j}(\cdot,z)\}_{l\in\mathbb{Z}_{\ge1}}$ is decreasingly convergent to $0$ with respect to $l$ for any $z\in\Omega_j$. As ${Z}_j$ is a discrete subset of $\Omega_j$, $ Z_{j,l}:=\Omega_{j,l}\cap{Z}_{j}$ is a set of finite points. Denote that $M_l:=\left(\prod_{1\le j\le n_1}\Omega_{j,l}\right)\times Y$ and $\psi_l:=\max_{1\le j\le n_1}\left\{\pi_{1,j}^*\left(\sum_{z_{j,k}\in {Z}_{j,l}}2p_{j,k}G_{\Omega_{j,l}}(\cdot,z_{j,k})\right)\right\}$  on $M_l$.
Denote that $$c_{j,k,l}=\exp\lim_{z\rightarrow z_{j,k}}\left(\frac{\sum_{z_{j,k_1}\in Z_{j,l}}p_{j,k_1}G_{\Omega_{j,l}}(z,z_{j,k_1})}{p_{j,k}}-\log|w_{j,k}(z)|\right)$$
for any $1\le j\le n_1$, $l\in\mathbb{Z}_{\ge1}$ and $1\le k<\tilde m_j$ satisfying $z_{j,k}\in Z_{j,l}$.
Hence $c_{j,k,l}$ is decreasingly convergent to $c_{j,k}$ with respect to $l$, $\psi_l$ is decreasingly convergent to $\psi$ with respect to $l$ and $\cup_{l\in\mathbb{Z}_{\ge1}}M_l=M$.

Then there exists a holomorphic $(n,0)$ form $F_l$ on $M_l$ such that $(F_l-f,(z_{\beta},y))\in(\mathcal{O}(K_{M_l})\otimes\mathcal{I}(\psi_l))_{(z_{\beta},y)}=(\mathcal{O}(K_{M})\otimes\mathcal{I}(\psi))_{(z_{\beta},y)}$ for any $\beta\in\{\tilde\beta\in\tilde{I}_1:z_{\tilde\beta}\in \prod_{1\le j\le n_1}\Omega_{j,l}\}$ and $y\in Y$, and $F_l$ satisfies
\begin{displaymath}
	\begin{split}
		&\int_{M_l}|F_l|^2e^{-\varphi}c(-\psi_l)\\
		\leq&\left(\int_0^{+\infty}c(s)e^{-s}ds\right)\sum_{\beta\in\{\tilde\beta\in\tilde{I}_1:z_{\tilde\beta}\in \prod_{1\le j\le n_1}\Omega_{j,l}\}}\sum_{\alpha\in E_{\beta}}\frac{(2\pi)^{n_1}e^{-\varphi_X(z_{\beta})}\int_Y|f_{\alpha,\beta}|^2e^{-\varphi_{Y}}}{\prod_{1\le j\le n_1}(\alpha_j+1)c_{j,\beta_j,l}^{2\alpha_{j}+2}}\\
		\le& \left(\int_{0}^{+\infty}c(s)e^{-s}ds\right)\sum_{\beta\in\tilde I_1}\sum_{\alpha\in E_{\beta}}\frac{(2\pi)^{n_1}e^{-\varphi_X(z_{\beta})}\int_Y|f_{\alpha,\beta}|^2e^{-\varphi_{Y}}}{\prod_{1\le j\le n_1}(\alpha_j+1)c_{j,\beta_j}^{2\alpha_{j}+2}}\\
		<&+\infty.
	\end{split}
\end{displaymath}
Since $\psi\le\psi_l$ and $c(t)e^{-t}$ is decreasing on $(0,+\infty)$, we have
\begin{equation}
	\label{eq:211223c}\begin{split}
		&\int_{M_l}|F_l|^2e^{-\varphi-\psi_l+\psi}c(-\psi)\\
		\le&\int_{M_l}|F_l|^2e^{-\varphi}c(-\psi_l)\\
		\leq&\left(\int_{0}^{+\infty}c(s)e^{-s}ds\right)\sum_{\beta\in\tilde I_1}\sum_{\alpha\in E_{\beta}}\frac{(2\pi)^{n_1}e^{-\varphi_X(z_{\beta})}\int_Y|f_{\alpha,\beta}|^2e^{-\varphi_{Y}}}{\prod_{1\le j\le n_1}(\alpha_j+1)c_{j,\beta_j}^{2\alpha_{j}+2}}.
	\end{split}
\end{equation}
Note that $\psi$ is continuous on $M\backslash Z_0$, $\psi_l$ is continuous on $M_l\backslash Z_0$ and $Z_0$ is a closed complex submanifold of $M$.
 For any compact subset $K$ of $M\backslash Z_0$, there exist $l_K>0$ such that $K\Subset M_{l_K}$ and $C_K>0$  such that $\frac{e^{\varphi+\psi_l-\psi}}{c(-\psi)}\le C_K$ for any $l\ge l_K$.
 It follows from Lemma \ref{l:converge} and the diagonal method that there exists a subsequence of $\{F_l\}$, denoted still by $\{F_l\}$, which is uniformly convergent to a holomorphic $(n,0)$ form $F$ on $M$ on any compact subset of $M$. It follows from the Fatou's Lemma and inequality \eqref{eq:211223c} that
\begin{displaymath}
	\begin{split}
		\int_{M}|F|^2e^{-\varphi}c(-\psi)&=\int_{M}\lim_{l\rightarrow+\infty}|F_l|^2e^{-\varphi-\psi_l+\psi}c(-\psi)\\
		&\leq\liminf_{l\rightarrow+\infty}\int_{M_l}|F_l|^2e^{-\varphi-\psi_l+\psi}c(-\psi)\\
		&\le\left(\int_{0}^{+\infty}c(s)e^{-s}ds\right)\sum_{\beta\in\tilde I_1}\sum_{\alpha\in E_{\beta}}\frac{(2\pi)^{n_1}e^{-\varphi_X(z_{\beta})}\int_Y|f_{\alpha,\beta}|^2e^{-\varphi_{Y}}}{\prod_{1\le j\le n_1}(\alpha_j+1)c_{j,\beta_j}^{2\alpha_{j}+2}}.	
	\end{split}
\end{displaymath}
 Since $\{F_l\}$ is uniformly convergent to   $F$ on any compact subset of $M$ and $(F_l-f,(z_{\beta},y))\in(\mathcal{O}(K_{M})\otimes\mathcal{I}(\psi))_{(z_{\beta},y)}$ for any $\beta\in\left\{\tilde\beta\in\tilde{I}_1:z_{\tilde\beta}\in \prod_{1\le j\le n_1}\Omega_{j,l}\right\}$ and $y\in Y$, it follows from Lemma \ref{l:closedness} that  $(F-f,(z_{\beta},y))\in(\mathcal{O}(K_{M})\otimes\mathcal{I}(\psi))_{(z_{\beta},y)}$ for any $\beta\in\tilde I_1$ and $y\in Y$.
\end{Remark}

In the following, we assume that $\tilde m_j<+\infty$ for any $1\le j\le n_1$.
Denote that $m_j=\tilde{m}_j-1$. As $\prod_{1\le j\le n_1}\Omega_j$ is a Stein manifold, it follows from Lemma \ref{l:FN1} that  there exist smooth plurisubharmonic functions $\Phi_l$ on $\prod_{1\le j\le n_1}\Omega_j$, which are decreasingly convergent to $\varphi_X$ with respect to $l$. Denote that
$$\varphi_l:=\pi_1^*(\Phi_l)+\pi_2^*(\varphi_Y).$$ As $Y$ is a weakly pseudoconvex K\"ahler manifold, it is well-known that there exist open weakly pseudoconvex K\"ahler manifolds $D_1\Subset\ldots\Subset D_{l'}\Subset D_{l'+1}\Subset\ldots$ such that $\cup_{l'\in\mathbb{Z}_{\ge1}}D_{l'}=Y$. Denote that $M_{l'}:=\left(\prod_{1\le j\le n_1}\Omega_j\right)\times D_{l'}$.

 It follows from Lemma \ref{l:green-sup} and Lemma \ref{l:green-sup2} that there exists a local coordinate $\tilde{w}_{j,k}$ on a neighborhood $\tilde{V}_{z_{j,k}}\Subset V_{z_{j,k}}$ of $z_{j,k}$ satisfying $\tilde{w}_{j,k}(z_{j,k})=0$ and
 $$|\tilde{w}_{j,k}|=\exp\left({\frac{\sum_{1\le k_1\le m_j}p_{j,k_1}G_{\Omega_j}(\cdot,z_{j,k_1})}{p_{j,k}}}\right)$$
  on $\tilde{V}_{z_{j,k}}.$
 Denote that $\tilde{V}_{\beta}:=\prod_{1\le j\le n_1}\tilde V_{j,\beta_j}$ for any $\beta\in\tilde I_1$. Let $\tilde{f}$ be a holomorphic $(n,0)$ form on $\cup_{\beta\in\tilde I_1}\tilde{V}_{\beta}\times Y$ satisfying
$$\tilde{f}=\sum_{\alpha\in E_{\beta}}c_{\alpha,\beta}\pi_1^*(\tilde{w}_{\beta}^{\alpha}d\tilde{w}_{1,\beta_1}\wedge d\tilde{w}_{2,\beta_2}\wedge\ldots\wedge d\tilde{w}_{n_1,\beta_{n_1}})\wedge\pi_2^*(f_{\alpha,\beta})$$
on $\tilde V_{\beta}\times Y$,  where $c_{\alpha,\beta}=\prod_{1\le j\le n_1}\left(\lim_{z\rightarrow z_{j,\beta_j}}\frac{w_{j,\beta_j}(z)}{\tilde{w}_{j,\beta_j}(z)}\right)^{\alpha_j+1}$.
It follows from Lemma \ref{l:0} that
$$(f-\tilde{f},z)\in(\mathcal{O}(K_{M})\otimes\mathcal{I}(\psi))_{z}$$
 for any $z\in Z_0$. Denote that
$$\psi_1:=\max_{1\le j\le n_1}\left\{2\sum_{1\le k\le m_j}p_{j,k}\tilde\pi_j^*(G_{\Omega_j}(\cdot,z_{j,k}))\right\}$$
on $\prod_{1\le j\le n_1}\Omega_j$, where $\tilde\pi_j$ is the natural projection from $\prod_{1\le j\le n_1}\Omega_j$ to $\Omega_j$. Note that $\psi=\pi_1^*(\psi_1)$.
It follows from Lemma \ref{l:G-compact} and Lemma \ref{l:green-sup2} that there exists $t_0>0$ such that $\{\psi_1<-t_0\}\Subset \cup_{\beta\in\tilde I_1}\tilde V_{\beta}$, which implies that $\int_{\{\psi_1<-t\}\times D_{l'}}|\tilde f|^2<+\infty$.

Using Lemma \ref{lem:L2}, there exists a holomorphic $(n,0)$ form $F_{l,l',t}$ on $M_{l'}$ such that
\begin{equation}
	\label{eq:1125m}
	\begin{split}
&\int_{M_{l'}}|F_{l,l',t}-(1-b_{t,1}(\psi))\tilde{f}|^{2}e^{-\varphi_l-\psi+v_{t,1}(\psi)}c(-v_{t,1}(\psi))\\
\leq& \left(\int_{0}^{t+1}c(s)e^{-s}ds\right) \int_{M_{l'}}\mathbb{I}_{\{-t-1<\psi<-t\}}|\tilde f|^2e^{-\varphi_l-\psi},
\end{split}
\end{equation}
where $t\ge t_0$. Note that $b_{t,1}(s)=0$ for large enough $s$, then $(F_{l,l',t}-\tilde f,z)\in(\mathcal{O}(K_{M})\otimes\mathcal{I}(\psi))_{z}$ for any $z\in Z_0\cap M_{l'}$.

For any $\epsilon>0$, there exists $t_1>t_0$, such that
 $\sup_{z\in\{\psi_1<-t_1\}\cap \tilde{V}_{\beta}}|\Phi_l(z)-\Phi_l(z_{\beta})|<\epsilon$ for any $\beta\in\tilde I_1$. Note that $\varphi_l=\pi_1^*(\Phi_l)+\pi_2^*(\varphi_Y)$ and $|c_{\alpha,\beta}|=\frac{1}{\prod_{1\le j\le n_1}c_{j,\beta_j}^{\alpha_j+1}}$ for any $\beta\in \tilde I_1$ and $\alpha\in E_{\beta}$. As $\{\psi_1<-t_1\}\Subset \cup_{\beta\in I_1}\tilde{V}_{\beta}$, it follows from Lemma \ref{l:m2}, the Fubini's theorem and
 $$\int_Y|f_{\alpha,\beta}|^2e^{-\varphi_Y}<+\infty$$
  that
\begin{equation}
	\label{eq:1203a}
	\begin{split}
		\int_{M_{l'}}\mathbb{I}_{\{-t-1<\psi<-t\}}|\tilde{f}|^2e^{-\varphi_l-\psi}=&	\int_{\{-t-1<\psi_1<-t\}\times D_{l'}}|\tilde{f}|^2e^{-\pi_1^*(\Phi_l+\psi)-\pi_2^*(\varphi_Y)}\\
		\le & \sum_{\beta\in\tilde I_1}\sum_{\alpha\in E_{\beta}}\frac{(2\pi)^{n_1}e^{-\Phi_l(z_{\beta})+\epsilon}}{\prod_{1\le j\le n_1}(\alpha_j+1)c_{j,\beta_j}^{2\alpha_{j}+2}}\int_{D_{l'}}|f_{\alpha,\beta}|^2e^{-\varphi_Y}
	\end{split}
\end{equation}
for $t>t_1$.
Letting $t\rightarrow+\infty$ and $\epsilon\rightarrow0$, inequality \eqref{eq:1203a} implies that
\begin{equation}
	\label{eq:1126c}\limsup_{t\rightarrow+\infty}\int_{M_{l'}}\mathbb{I}_{\{-t-1<\psi<-t\}}|\tilde{f}|^2e^{-\varphi_l-\psi}\le \sum_{\beta\in\tilde I_1}\sum_{\alpha\in E_{\beta}}\frac{(2\pi)^{n_1}e^{-\Phi_l(z_{\beta})}\int_{D_{l'}}|f_{\alpha,\beta}|^2e^{-\varphi_Y}}{\prod_{1\le j\le n_1}(\alpha_j+1)c_{j,\beta_j}^{2\alpha_{j}+2}}.
\end{equation}
As $v_{t,1}(\psi)\ge\psi$ and $c(t)e^{-t}$ is decreasing, Combining inequality \eqref{eq:1125m} and \eqref{eq:1126c}, then we have
\begin{equation}
	\label{eq:1126d}
	\begin{split}
&\limsup_{t\rightarrow+\infty}\int_{M_{l'}}|F_{l,l',t}-(1-b_{t,1}(\psi))\tilde{f}|^{2}e^{-\varphi_l}c(-\psi)\\
\le&\limsup_{t\rightarrow+\infty}\int_{M_{l'}}|F_{l,l',t}-(1-b_{t,1}(\psi))\tilde{f}|^{2}e^{-\varphi_l-\psi+v_{t,1}(\psi)}c(-v_{t,1}(\psi))\\
\leq& \limsup_{t\rightarrow+\infty}\left(\int_{0}^{t+1}c(s)e^{-s}ds\right)\int_{M_{l'}}\mathbb{I}_{\{-t-1<\psi<-t\}}|\tilde{f}|^2e^{-\varphi_l-\psi}\\
\leq&\left(\int_0^{+\infty}c(s)e^{-s}ds\right)\sum_{\beta\in\tilde I_1}\sum_{\alpha\in E_{\beta}}\frac{(2\pi)^{n_1}e^{-\Phi_l(z_{\beta})}\int_{D_{l'}}|f_{\alpha,\beta}|^2e^{-\varphi_Y}}{\prod_{1\le j\le n_1}(\alpha_j+1)c_{j,\beta_j}^{2\alpha_{j}+2}}\\
	<&+\infty.
\end{split}
\end{equation}
Note that $\psi$ is continuous on $M\backslash Z_0$. For any open set $K\Subset M_{l'}\backslash Z_0$, as $b_{t,1}(s)=1$ for any $s\ge -t$ and $c(s)e^{-s}$ is decreasing with respect to $s$, we get that there exists a constant $C_K>0$ such that
$$\int_{K}|(1-b_{t,1}(\psi))\tilde{f}|^2e^{-\varphi_l}c(-\psi)\le C_K\int_{\{\psi<-t_1\}\cap K}|\tilde{f}|^2<+\infty$$ for any $t>t_1$,
which implies that
$$\limsup_{t\rightarrow+\infty}\int_{K}|F_{l,l',t}|^2e^{-\varphi_l}c(-\psi)<+\infty.$$
Using Lemma \ref{l:converge} and the diagonal method,
we obtain that
there exists a subsequence of $\{F_{l,l',t}\}_{t\rightarrow+\infty}$ denoted by $\{F_{l,l',t_m}\}_{m\rightarrow+\infty}$
uniformly convergent on any compact subset of $M_{l'}\backslash Z_0$. As $Z_0$ is a closed complex submanifold of $M$, we obtain that $\{F_{l,l',t_m}\}_{m\rightarrow+\infty}$ is uniformly convergent to a holomorphic $(n,0)$ form $F_{l,l'}$ on $M_{l'}$ on any compact subset of $M_{l'}$. Then it follows from inequality \eqref{eq:1126d} and the Fatou's Lemma that
\begin{displaymath}
\begin{split}
&\int_{M_{l'}}|F_{l,l'}|^{2}e^{-\varphi_l}c(-\psi)
\\=&\int_{M_{l'}}\liminf_{m\rightarrow+\infty}|F_{l,l',t_m}-(1-b_{t_m,1}(\psi))\tilde{f}|^{2}e^{-\varphi_l}c(-\psi)\\
\leq&\liminf_{m\rightarrow+\infty}\int_{M_{l'}}|F_{l,l',t_m}-(1-b_{t_m,1}(\psi))\tilde{f}|^{2}e^{-\varphi_l}c(-\psi)\\
\leq&\left(\int_0^{+\infty}c(s)e^{-s}ds\right)\sum_{\beta\in\tilde I_1}\sum_{\alpha\in E_{\beta}}\frac{(2\pi)^{n_1}e^{-\Phi_l(z_{\beta})}\int_{D_{l'}}|f_{\alpha,\beta}|^2e^{-\varphi_Y}}{\prod_{1\le j\le n_1}(\alpha_j+1)c_{j,\beta_j}^{2\alpha_{j}+2}}\\
	<&+\infty.
\end{split}	
\end{displaymath}

Note that $\lim_{l\rightarrow+\infty}\Phi_l(z_{\beta})=\varphi_X(z_{\beta})>-\infty$ for any $\beta\in I_1$, then we have
\begin{equation}
	\label{eq:1126e}\begin{split}
		&\limsup_{l\rightarrow+\infty}\int_{M_{l'}}|F_{l,l'}|^{2}e^{-\varphi_l}c(-\psi)\\
		\leq&\left(\int_0^{+\infty}c(s)e^{-s}ds\right)\sum_{\beta\in\tilde I_1}\sum_{\alpha\in E_{\beta}}\frac{(2\pi)^{n_1}e^{-\varphi_X(z_{\beta})}\int_{D_{l'}}|f_{\alpha,\beta}|^2e^{-\varphi_Y}}{\prod_{1\le j\le n_1}(\alpha_j+1)c_{j,\beta_j}^{2\alpha_{j}+2}}\\
	<&+\infty.
	\end{split}
\end{equation}
Note that $\psi$ is continuous on $M\backslash Z_0$ and $Z_0$ is a closed complex submanifold of $M$.
Using Lemma \ref{l:converge},
we obtain that
there exists a subsequence of $\{F_{l,l'}\}_{l\rightarrow+\infty}$ (also denoted by $\{F_{l,l'}\}_{l\rightarrow+\infty}$)
uniformly convergent to a holomorphic $(n,0)$ form $F_{l'}$ on $M_{l'}$ on any compact subset of $M_{l'}$, which satisfies that
\begin{equation*}
\int_{M_{l'}}|F_{l'}|^2e^{-\varphi}c(-\psi)\le \left(\int_0^{+\infty}c(s)e^{-s}ds\right)\sum_{\beta\in\tilde I_1}\sum_{\alpha\in E_{\beta}}\frac{(2\pi)^{n_1}e^{-\varphi_X(z_{\beta})}\int_{D_{l'}}|f_{\alpha,\beta}|^2e^{-\varphi_Y}}{\prod_{1\le j\le n_1}(\alpha_j+1)c_{j,\beta_j}^{2\alpha_{j}+2}}.
\end{equation*}

As $\cup_{l'\in\mathbb{Z}_{\ge1}}D_{l'}=Y$, we have
\begin{equation}
	\label{eq:211223d}\begin{split}
	&\limsup_{l'\rightarrow+\infty}\int_{M_{l'}}|F_{l'}|^2e^{-\varphi}c(-\psi)\\
	\leq
		&\lim_{l'\rightarrow+\infty}\sum_{\beta\in\tilde I_1}\sum_{\alpha\in E_{\beta}}\frac{(2\pi)^{n_1}e^{-\varphi_X(z_{\beta})}\int_{D_{l'}}|f_{\alpha,\beta}|^2e^{-\varphi_Y}}{\prod_{1\le j\le n_1}(\alpha_j+1)c_{j,\beta_j}^{2\alpha_{j}+2}}\\		
	=&\sum_{\beta\in\tilde I_1}\sum_{\alpha\in E_{\beta}}\frac{(2\pi)^{n_1}e^{-\varphi_X(z_{\beta})}\int_{Y}|f_{\alpha,\beta}|^2e^{-\varphi_Y}}{\prod_{1\le j\le n_1}(\alpha_j+1)c_{j,\beta_j}^{2\alpha_{j}+2}}\\
	<&+\infty.
	\end{split}
\end{equation}
Note that $\psi$ is continuous on $M\backslash Z_0$, $Z_0$ is a closed complex submanifold of $M$ and $\cup_{l'\in\mathbb{Z}_{\ge1}}M_{l'}=M$.
Using Lemma \ref{l:converge} and  the diagonal method,
we get that
there exists a subsequence of $\{F_{l'}\}$ (also denoted by $\{F_{l'}\}$)
uniformly convergent to a holomorphic $(n,0)$ form $F$ on $M$ on any compact subset of $M$.  Then it follows from inequality \eqref{eq:211223d} and the Fatou's Lemma that
\begin{equation*}
\begin{split}
	\int_M|F|^2e^{-\varphi}c(-\psi)&=\int_{M}\liminf_{l'\rightarrow+\infty}\mathbb{I}_{M_{l'}}|F_{l'}|^2e^{-\varphi}c(-\psi)\\
	&\le\liminf_{l'\rightarrow+\infty}\int_{M_{l'}}|F_{l'}|^2e^{-\varphi}c(-\psi)\\
	&\leq\sum_{\beta\in\tilde I_1}\sum_{\alpha\in E_{\beta}}\frac{(2\pi)^{n_1}e^{-\varphi_X(z_{\beta})}\int_{Y}|f_{\alpha,\beta}|^2e^{-\varphi_Y}}{\prod_{1\le j\le n_1}(\alpha_j+1)c_{j,\beta_j}^{2\alpha_{j}+2}}.\end{split}
\end{equation*}
Following from Lemma \ref{l:closedness}, we have $(F-f,z)\in(\mathcal{O}(K_{M})\otimes\mathcal{I}(\psi))_{z}$ for any $z\in Z_0$.

Thus, Proposition \ref{p:exten-fibra} holds.
\end{proof}

\section{proofs of Theorem \ref{thm:linear-fibra-single} and remark \ref{r:fibra-single}}

In this section, we prove Theorem \ref{thm:linear-fibra-single} and Remark \ref{r:fibra-single}.

\subsection{Proofs of the sufficiency part of Theorem \ref{thm:linear-fibra-single} and Remark \ref{r:fibra-single}}\label{sec:s-1}

\

In this section, we prove the sufficiency part of Theorem \ref{thm:linear-fibra-single} and Remark \ref{r:fibra-single}.

Denote that $M':=\prod_{1\le j\le n_1}\Omega_j$, and let $\tilde\pi_j$ be the natural projection from $M'$ to $\Omega_j$.  Denote that
$\psi_1:=\max_{1\le j\le n_1}\left\{\tilde\pi_{j}^*(2p_{j}G_{\Omega_j}(\cdot,z_{j}))\right\}$ and $\tilde\varphi:=\sum_{1\le j\le n_1}\tilde\pi_j^*(\varphi_j)$ on $M'$.
It follows from statements $(2)$ and $(3)$ in Theorem \ref{thm:linear-fibra-single} that
$$\tilde{f}_{\alpha}=\wedge_{1\le j\le n_1}\tilde\pi_j^*\left(g_j(P_j)_*\left(f_{u_j}f_{z_j}^{\alpha_j}df_{z_j}\right)\right)$$
 is a (single-value) holomorphic $(n_1,0)$ form on $M'$ for any $\alpha\in E$ satisfying $f_{\alpha}\not\equiv0$, where $P_j:\Delta\rightarrow\Omega_j$ is the universal covering, $f_{u_j}$ is a holomorphic $(1,0)$ form on $\Delta$ satisfying $|f_{u_j}|=(P_j)^*(e^{u_j})$ and $f_{z_j}$ is a holomorphic $(1,0)$ form on $\Delta$ satisfying $|f_{z_j}|=(P_j)^*\left(e^{G_{\Omega_j}(\cdot,z_j)}\right)$. Denote that $\tilde E:=\{\alpha\in E:f_{\alpha}\not\equiv0\}$. Let
 $$F=\sum_{\alpha\in\tilde E}c_{\alpha}\pi_1^*(\tilde f_{\alpha})\wedge\pi_2^*(f_{\alpha})$$
 be a holomorphic $(n,0)$ form on $M$, where $c_{\alpha}=\lim_{z\rightarrow z_0}\frac{w^{\alpha}dw_1\wedge\ldots\wedge dw_{n_1}}{\tilde{f}_{\alpha}}$. As $\int_Y|f_{\alpha}|^2e^{-\varphi_Y}<+\infty$ and $\tilde{\varphi}(z_0)>-\infty$, it follows Lemma \ref{l:0} and Lemma \ref{l:phi1+phi2} that
 $$(F-f,z)\in(\mathcal{O}(K_M)\otimes\mathcal{I}(\varphi+\psi))_z$$
  for any $z\in Z_0$.

It follows from Remark \ref{r:1.1} that $\sum_{\alpha\in \tilde E}c_{\alpha}d_{\alpha}\tilde{f_{\alpha}}$
 is the unique holomorphic $(n_1,0)$ form  on $M'$ such that $\left(\sum_{\alpha\in \tilde E}c_{\alpha}d_{\alpha}\tilde{f_{\alpha}}-\sum_{\alpha\in\tilde E}d_{\alpha}w^{\alpha}dw_1\wedge\ldots\wedge dw_{n_1},z_0\right)\in(\mathcal{O}(K_{M'})\otimes\mathcal{I}(\psi_1))_{z_0}$, $\int_{\{\psi_1<-t\}}|\sum_{\alpha\in \tilde E}c_{\alpha}d_{\alpha}\tilde{f_{\alpha}}|^2e^{-\tilde\varphi}c(-\psi_1)=\inf\big\{\int_{\{\psi_1<-t\}}|\tilde F|^2e^{-\tilde\varphi}c(-\psi_1):\tilde F$ is a holomorphic $(n_1,0)$ form on $\{\psi_1<-t\}$ satisfying that $(\tilde F-\sum_{\alpha\in\tilde E}d_{\alpha}w^{\alpha}dw_1\wedge\ldots\wedge dw_{n_1},z_0)\in(\mathcal{O}(K_{M'}))_{z_0}\otimes\mathcal{I}(\psi_1)_{z_0}\big\}$ and
 \begin{equation}
 	\label{eq:211219b}
 	\begin{split}
 	&\int_{\{\psi_1<-t\}}\bigg|\sum_{\alpha\in\tilde E}c_{\alpha}d_{\alpha}\tilde{f}_{\alpha}\bigg|^2e^{-\tilde\varphi}c(-\psi_1)\\
 	=&\left(\int_t^{+\infty}c(s)e^{-s}ds\right)\sum_{\alpha\in\tilde E}\frac{|d_{\alpha}|^2(2\pi)^{n_1}e^{-\tilde\varphi(z_{0})}}{\prod_{1\le j\le n_1}(\alpha_j+1)c_{j}(z_j)^{2\alpha_{j}+2}}\end{split}
 \end{equation}
	 for any $t\ge0$, where  $c_j(z_j)=\exp\lim_{z\rightarrow z_j}(G_{\Omega_j}(z,z_j)-\log|w_j(z)|)$. Following from equality \eqref{eq:211219b} and the Fubini's theorem, we obtain that
	 \begin{equation}
	 	\label{eq:211219c}\begin{split}
	 		&\int_{\{\psi<-t\}}|F|^2e^{-\varphi}c(-\psi)\\
	 		=&\left(\int_t^{+\infty}c(s)e^{-s}ds\right)\sum_{\alpha\in E}\frac{(2\pi)^{n_1}e^{-\sum_{1\le j\le n_1}\varphi(z_j)}}{\prod_{1\le j\le n_1}(\alpha_j+1)c_{j}(z_j)^{2\alpha_{j}+2}}\int_Y|f_{\alpha}|^2e^{-\varphi_Y}\\
	 		<&+\infty
	 	\end{split}
	 \end{equation}
	 for any $t\ge 0$.
 Thus, $G(t)\le \int_{\{\psi<-t\}}|F|^2e^{-\varphi}c(-\psi)<+\infty$	 for any $t\ge 0$.

It follows from Lemma \ref{lem:A} that there exists a holomorphic $(n,0)$ form $F_t$ on $\{\psi<-t\}$ satisfying that $(F_t-f,z)\in(\mathcal{O}(K_M)\otimes\mathcal{I}(\varphi+\psi))_z$
  for any $z\in Z_0$ and $G(t)=\int_{\{\psi<-t\}}|F_t|^2e^{-\varphi}c(-\psi)$.
 For any $y_0\in Y$, let $u=(u_1,\ldots,u_{n_2})$ be a coordinate on a neighborhood $U$ of $y$ satisfying $u(y_0)=0$ and $u(U)=\Delta^{n_2}$. Lemma \ref{l:fibra-decom} implies that $F_t|_U=\sum_{\gamma \in\mathbb{Z}_{\ge0}^{n_2}}\pi_1^*(f_{t,\gamma})\wedge \pi_2^*(u^{\gamma}du_1\wedge\ldots du_{n_2})$, where $f_{t,\gamma}$ is a holomorphic $(n_1,0)$ form on $\{\psi_1<-t\}$ for any $\gamma\in\mathbb{Z}_{\ge0}^{n_2}$. There exists a holomorphic function $ f_{u,\alpha}$ on $U$ such that $f_{\alpha}= f_{u,\alpha}du_1\wedge\ldots\wedge du_{n_2}$ on $U$ for any $\alpha\in\tilde E$. Note that $f=\sum_{\alpha\in\tilde E}\pi_{1}^*\left(w^{\alpha}dw_1\wedge\ldots\wedge dw_{n_1}\right)\wedge \pi_2^*(f_{\alpha})+g_0$ on $V_0\times Y$, where  $g_0$ is a holomorphic $(n,0)$ form on $V_0\times Y$ satisfying $(g_0,z)\in(\mathcal{O}(K_M)\otimes\mathcal{I}(\varphi+\psi))_{z}$ for any $z\in Z_0$. It follows from Lemma \ref{l:0'} and  $(F_t-f,z)\in
(\mathcal{O}(K_{M})\otimes\mathcal{I}(\varphi+\psi))_z$ for any $z\in Z_0$ that $\left(\sum_{\gamma \in\mathbb{Z}_{\ge0}^{n_2}}u^{\gamma}f_{t,\gamma}- \sum_{\alpha\in\tilde E}f_{u,\alpha}(u)w^{\alpha}dw_1\wedge\ldots\wedge dw_{n_1}\right)\in
(\mathcal{O}(K_{M'})\otimes\mathcal{I}(\psi_1))_{z_0}$  for any $u\in \Delta^{n_2}$.  Let $U_1$ be an open subset of $U$, and let $V=u(U_1)\subset\Delta^{n_2}$. Note that $\left(\sum_{\alpha\in\tilde E}c_{\alpha}f_{u,\alpha}(u)\tilde{f}_{\alpha}-\sum_{\alpha\in\tilde E}f_{u,\alpha}(u)w^{\alpha}dw_1\wedge\ldots\wedge dw_{n_1}\right)\in
(\mathcal{O}(K_{M'})\otimes\mathcal{I}(\psi_1))_{z_0}$  for any $u\in \Delta^{n_2}$. Following the Fubini's theorem and the minimal property of  $\int_{\{\psi_1<-t\}}|\sum_{\alpha\in\tilde E} c_{\alpha}f_{u,\alpha}\tilde{f}_{\alpha}|^2e^{-\tilde\varphi}c(-\psi_1)$, we have
\begin{equation*}
	\begin{split}
		&\int_{\{\psi_1<-t\}\times U_1}|F_t|^2e^{-\varphi}c(-\psi)\\
		=&\int_{V}\bigg(\int_{\{\psi_1<-t\}}\bigg|\sum_{\gamma \in\mathbb{Z}_{\ge0}^{n_2}}u^{\gamma}f_{t,\gamma}\bigg|^2e^{-\tilde\varphi}c(-\psi_1)\bigg)e^{-\varphi_Y}|du_1\wedge\ldots\wedge du_{n_2}|^2\\
		\ge &\int_{V}\bigg(\int_{\{\psi_1<-t\}}\bigg|\sum_{\alpha\in\tilde E}c_{\alpha}f_{u,\alpha}(u)\tilde f_{\alpha}\bigg|^2e^{-\tilde\varphi}c(-\psi_1)\bigg)e^{-\varphi_Y}|du_1\wedge\ldots\wedge du_{n_2}|^2\\
		=&\int_{\{\psi_1<-t\}\times U_1}\bigg|\sum_{\alpha\in\tilde E}c_{\alpha}\pi_1^*(\tilde f_{\alpha})\wedge\pi_2^*(f_{\alpha})\bigg|^2e^{-\varphi}c(-\psi),
	\end{split}
\end{equation*}
which implies $G(t)=\int_{\{\psi<-t\}}|F_t|^2e^{-\varphi}c(-\psi)\ge\int_{\{\psi<-t\}}|F|^2e^{-\varphi}c(-\psi)$. It follows from $G(t)\le\int_{\{\psi<-t\}}|F|^2e^{-\varphi}c(-\psi)$ and inequality \eqref{eq:211219c} that
\begin{displaymath}
	\begin{split}
		G(t)&=\int_{\{\psi<-t\}}|F|^2e^{-\varphi}c(-\psi)\\
		&=\left(\int_t^{+\infty}c(s)e^{-s}ds\right)\sum_{\alpha\in E}\frac{(2\pi)^{n_1}e^{-\sum_{1\le j\le n_1}\varphi(z_j)}}{\prod_{1\le j\le n_1}(\alpha_j+1)c_{j}(z_j)^{2\alpha_{j}+2}}\int_Y|f_{\alpha}|^2e^{-\varphi_Y},
	\end{split}
\end{displaymath}
hence $G(h^{-1}(r))$ is linear with respect to $r\in(0,\int_0^{+\infty}c(s)e^{-s}ds]$. The uniqueness of $F$ follows from Corollary \ref{c:linear}.

Thus,  the sufficiency part of Theorem \ref{thm:linear-fibra-single} and Remark \ref{r:fibra-single} hold.

\subsection{Proof of the necessity part of Theorem \ref{thm:linear-fibra-single}}

\

In this section,  we prove the necessity part of Theorem \ref{thm:linear-fibra-single} in three steps.

\

\emph{Step 1. $f=\sum_{\alpha\in E}\pi_1^*(w^{\alpha}d w_1\wedge\ldots\wedge dw_{n_1}) \wedge \pi_2^*(f_\alpha)+g_0.$}

\

 Corollary \ref{c:linear} show that  there is a unique holomorphic $(n,0)$ form $F$ on $M$ satisfying $(F-f,z)\in(\mathcal{O}(K_{M})\otimes\mathcal{I}(\varphi+\psi))_{z}$ for any $z\in Z_0$ and $G(t)=\int_{\{\psi<-t\}}|F|^2e^{-\varphi}c(-\psi)$ for any $t\geq 0$.
It follows from Lemma \ref{l:green-sup}  that there exists a local coordinate $\tilde w_{j}$ on a neighborhood $\tilde V_{z_{j}}\Subset V_{z_j}$ of $z_{j}\in\Omega_j$ satisfying $\tilde w_{j}(z_{j})=0$ and
	$$\log|\tilde w_{j}|=G_{\Omega_j}(\cdot,z_{j})$$ on $\tilde V_{z_j}$ for any $j\in\{1,\ldots,n_1\}$. Denote that $\tilde V_0:=\prod_{1\le j\le n_1}\tilde V_{z_{j}}$ and $\tilde w:=(\tilde w_1,\ldots,\tilde w_{n_1})$ on $\tilde V_0$. 	Using Lemma \ref{l:G-compact}, we get that  there exists $t_0>0$ such that
	$$\{2p_j G_{\Omega_j}(\cdot,z_j)<-t_0\}\Subset \tilde V_{z_j}$$
for any $1\le j\le n_1$.
As $\varphi_j$ is a subharmonic function on $\Omega_j$, $\int_{\{\psi<-t_0\}}|F|^2e^{-\varphi}c(-\psi)<+\infty$ implies that
$$\int_{\{\psi<-t_0\}}|F|^2e^{-\pi_{2}^*(\varphi_Y)}c(-\psi)<+\infty.$$
Note that
$$\{\psi<-t_0\}=\left(\prod_{1\le j\le n_1}\left\{|\tilde w_j|<e^{-\frac{t_0}{2p_j}}\right\}\right)\times Y.$$
It follows from Lemma \ref{l:fibra-decom} that  there exists a unique sequence of  holomorphic  $(n_2,0)$   forms $\{F_{\alpha}\}_{\alpha\in\mathbb{Z}_{\ge0}^{n_1}}$ on $Y$ such that
	\begin{equation}\label{eq:1217a}
		F=\sum_{\alpha\in\mathbb{Z}_{\ge0}^{n_1}}\pi_1^*(\tilde w^{\alpha}d\tilde w_1\wedge\ldots\wedge d\tilde w_{n_1}) \wedge \pi_2^*(F_\alpha)
	\end{equation}
	on $\{\psi<-t_0\}$
	and
	\begin{equation}
		\label{eq:1217b}\int_Y|F_{\alpha}|^2e^{-\varphi_Y}<+\infty,
	\end{equation}
where the right term of the above equality is uniformly convergent on any compact subset of $M$. As $\frac{\int_{\{\psi<-t\}}|F|^2e^{-\varphi}c(-\psi)}{\int_{t}^{+\infty}c(s)e^{-s}ds}$ is a positive number independent of $t$, Lemma \ref{l:limit} implies that $F_{\alpha}\equiv0$ for any $\alpha\in\mathbb{Z}_{\ge0}$ satisfying $\sum_{1\le j\le n_1}\frac{\alpha_j+1}{p_j}<1$. Denote that $E_2:=\left\{\alpha\in\mathbb{Z}_{\ge0}^{n_1}:\sum_{1\le j\le n_1}\frac{\alpha_j+1}{p_j}>1\right\}$. Note that $\varphi(z_j)>-\infty$ for any $1\le j\le n_1$. It follows from Lemma \ref{l:0} and Lemma \ref{l:phi1+phi2} that $(\pi_1^*(\tilde w^{\alpha}d\tilde w_1\wedge\ldots\wedge d\tilde w_{n_1}) \wedge \pi_2^*(F_\alpha),z)\in(\mathcal{O}(K_{M})\otimes\mathcal{I}(\varphi+\psi))_{z}$ for any $z\in Z_0$ and
$\alpha\in E_2$, thus
$$\left(\sum_{\alpha\in E_2}\pi_1^*(\tilde w^{\alpha}d\tilde w_1\wedge\ldots\wedge d\tilde w_{n_1}) \wedge \pi_2^*(F_\alpha),z\right)\in(\mathcal{O}(K_{M})\otimes\mathcal{I}(\varphi+\psi))_{z}$$
for any $z\in Z_0$ (by using Lemma \ref{l:closedness}). As $(F-f,z)\in(\mathcal{O}(K_{M})\otimes\mathcal{I}(\varphi+\psi))_{z}$ for any $z\in Z_0$, we have
$$\left(f-\sum_{\alpha\in E}\pi_1^*(\tilde w^{\alpha}d\tilde w_1\wedge\ldots\wedge d\tilde w_{n_1}) \wedge \pi_2^*(F_\alpha),z\right)\in(\mathcal{O}(K_{M})\otimes\mathcal{I}(\varphi+\psi))_{z}$$
for any $z\in Z_0$.  Denote that
$$\psi_1:=\max_{1\le j\le n_1}\left\{\tilde\pi_{j}^*(2p_{j}G_{\Omega_j}(\cdot,z_{j}))\right\}$$ on $\prod_{1\le j\le n_1}\Omega_j$, where $\tilde\pi_j$ is the natural projection from $\prod_{1\le j\le n_1}\Omega_j$ to $\Omega_j$. Taking $c_{\alpha}=\prod_{1\le j\le n_1}\left(\lim_{z\rightarrow z_j}\frac{\tilde w_j}{w_j}\right)^{\alpha_j+1}$, it follows from Lemma \ref{l:0} and Lemma \ref{l:phi1+phi2} that  $(\tilde w^{\alpha}d\tilde w_1\wedge\ldots\wedge d\tilde w_{n_1}-c_{\alpha}w^{\alpha}d w_1\wedge\ldots\wedge d w_{n_1},z_0)\in\mathcal{O}(K_{\prod_{1\le j\le n_1}\Omega_j})_{z_0}\otimes\mathcal{I}\left(\sum_{1\le j\le n_1}\tilde{\pi}_j(\varphi_j)+\psi_1\right)_{z_0}$ for any $\alpha\in E$, which implies that $\big(\sum_{\alpha\in E}\pi_1^*(\tilde w^{\alpha}d\tilde w_1\wedge\ldots\wedge d\tilde w_{n_1}) \wedge \pi_2^*(F_\alpha)-\sum_{\alpha\in E}\pi_1^*(c_{\alpha}w^{\alpha}d w_1\wedge\ldots\wedge dw_{n_1}) \wedge \pi_2^*(F_\alpha),z\big)\in(\mathcal{O}(K_{M})\otimes\mathcal{I}(\varphi+\psi))_{z}$ for any $z\in Z_0$. Taking $f_{\alpha}=c_{\alpha}F_{\alpha}$, there exists a holomorphic $(n,0)$ form $g_0$ on $V_0\times Y$ such that
$$f=\sum_{\alpha\in E}\pi_1^*(w^{\alpha}d w_1\wedge\ldots\wedge dw_{n_1}) \wedge \pi_2^*(f_\alpha)+g_0$$
and $(g_0,z)\in(\mathcal{O}(K_{M})\otimes\mathcal{I}(\varphi+\psi))_{z}$ for any $z\in Z_0$. As $G(0)>0$, we know that there exists $\alpha\in E$ such that $f_{\alpha}\not\equiv0$.

\

\emph{Step 2. $G(-\log r;\tilde{c}\equiv1 )$ is linear with respect to $r$.}

\

It follows from Corollary \ref{c:linear} that $G(t;\tilde{c}\equiv1 )\le\int_{\{\psi<-t\}}|F|^2e^{-\varphi}=\frac{G(0;c)}{\int_{0}^{+\infty}c(s)e^{-s}ds}e^{-t}<+\infty$.
Denote
\begin{equation*}
\begin{split}
\inf\bigg\{\int_{\{\psi<-t\}}|\tilde{f}|^{2}e^{-\varphi}:(\tilde{f}-f,z)\in(\mathcal{O}(K_M)&\otimes\mathcal{I}(\psi))_{z}\mbox{ for any $z\in Z_0$} \\&\&{\,}\tilde{f}\in H^{0}(\{\psi<-t\},\mathcal{O}(K_{M}))\bigg\}
\end{split}
\end{equation*}
by $\tilde G(t)$, where $t\ge0$.
It follows from Lemma \ref{l:G1=G2} that $G(t;\tilde{c}\equiv1 )=\tilde G(t)$ for any $t\ge 0$. Denote that $M':=\prod_{1\le j\le n_1}\Omega_j$, and let $K_{M'}$ be the canonical (holomorphic) line bundle on $M'$.
Using Lemma \ref{l:G1=G2}, Lemma \ref{l:orth1} and Lemma \ref{l:orth2}, we obtain that there exists a unique holomorphic $(n,0)$ form $F_t=\sum_{\alpha\in E}\pi_1^*(h_{t,\alpha})\wedge\pi_2^*(f_{\alpha})$ on $\{\psi<-t\}$
satisfying
\begin{equation}
	\label{eq:211218d}G(t;\tilde{c}\equiv1)=\tilde G(t)=\int_{\{\psi<-t\}}|F_t|^2e^{-\varphi}=\sum_{\alpha\in E}\int_{\{\psi<-t\}}|\pi_1^*(h_{t,\alpha})\wedge\pi_2^*(f_{\alpha})|^2e^{-\varphi},
\end{equation}
where $h_{t,\alpha}$ is a holomorphic $(n_1,0)$ form on $\{\psi_1<-t\}$ satisfying $$(h_{t,\alpha}-w^{\alpha}d w_1\wedge\ldots\wedge dw_{n_1},z_0)\in (\mathcal{O}(K_{M'})\otimes\mathcal{I}(\psi_1))_{z_0}$$
 and $\int_{\{\psi_1<-t\}}|h_{t,\alpha}|^2e^{-\sum_{1\le j\le n_1}\tilde{\pi}_j^*(\varphi_j)}=\inf\big\{\int_{\{\psi_1<-t\}}|\tilde F|^2e^{-\sum_{1\le j\le n_1}\tilde{\pi}_j^*(\varphi_j)}:\tilde F$ is a holomorphic $(n_1,0)$ form on $\{\psi_1<-t\}$ satisfying $(\tilde F-w^{\alpha}d w_1\wedge\ldots\wedge dw_{n_1},z_0)\in (\mathcal{O}(K_{M'})\otimes\mathcal{I}(\psi_1))_{z_0}\big\}<+\infty$. It follows from Lemma \ref{p:exten-pro-finite} that there exists a holomorphic $(n_1,0)$ form $\tilde h_{\alpha}$ on $M'$ such that $\int_{M'}|\tilde{h}_{\alpha}|^2e^{-\sum_{1\le j\le n_1}\tilde\pi_j^*(\varphi_j)}c(-\psi_1)<+\infty$ and $(\tilde h_{\alpha}-w^{\alpha}d w_1\wedge\ldots\wedge dw_{n_1},z_0)\in (\mathcal{O}(K_{M'})\otimes\mathcal{I}(\psi_1))_{z_0}$. As $\varphi_j(z_j)>-\infty$ for any $1\le j\le n_1$, it follows from Lemma \ref{l:phi1+phi2} that there exists $t_1>t$ such that
 $$\int_{\{\psi_1<-t_1\}}|h_{t,\alpha}-\tilde h_{\alpha}|^2e^{-\sum_{1\le j\le n_1}\tilde\pi_j^*(\varphi_j)-\psi_1}<+\infty$$
 for any $\alpha\in E$.
 As $c(s)e^{-s}$ is a positive decreasing function on $(0,+\infty)$, for any $t>0$, we obtain that
\begin{equation*}
\begin{split}
		&\int_{\{\psi_1<-t\}}|h_{t,\alpha}|^2e^{-\sum_{1\le j\le n_1}\tilde\pi_j^*(\varphi_j)}c(-\psi_1)\\
		\le&C\int_{\{\psi_1<-t_1\}}|h_{t,\alpha}-\tilde h_{\alpha}|^2e^{-\sum_{1\le j\le n_1}\tilde\pi_j^*(\varphi_j)-\psi_1}\\
		&+\int_{\{\psi_1<-t_1\}}|\tilde h_{\alpha}|^2e^{-\sum_{1\le j\le n_1}\tilde\pi_j^*(\varphi_j)}c(-\psi_1)\\
		&+\sup_{s\in(t,t_1]}c(s)\times\int_{\{-t_1\le \psi_1<-t\}}|h_{t,\alpha}|^2e^{-\sum_{1\le j\le n_1}\tilde\pi_j^*(\varphi_j)}\\
		<&+\infty
	\end{split}
\end{equation*}
for any $\alpha\in E$,
which implies that
 \begin{equation}
 	\label{eq:211218c}\begin{split}
 	&\int_{\{\psi<-t\}}|F_t|^2e^{-\varphi}c(-\psi)\\
 	\le&C\sum_{\alpha\in E}\int_{\{\psi_1<-t\}}|h_{t,\alpha}|^2e^{-\sum_{1\le j\le n_1}\tilde\pi_j^*(\varphi_j)}c(-\psi_1)\times\int_Y|f_{\alpha}|^2e^{-\varphi_Y} 	\\
 	<&+\infty.
 	\end{split} \end{equation}
It follows from Lemma \ref{l:linear2} and inequality \eqref{eq:211218c} that
$$G(t;\tilde{c}\equiv1 )=\int_{\{\psi<-t\}}|F|^2e^{-\varphi}=\frac{G(0;c)}{\int_0^{+\infty}c(s)e^{-s}ds}e^{-t}$$
for any $t>0$. Theorem \ref{thm:general_concave} shows that $\lim_{t\rightarrow 0+}G(t;\tilde{c}\equiv1 )=G(0;\tilde{c}\equiv1 )$, hence we get $G(-\log r;\tilde c\equiv1 )$ is linear with respect to $r\in(0,1]$.

\

\emph{Step 3. proofs of statements $(2)$ and $(3)$ in Theorem \ref{thm:linear-fibra-single}.}

\

Denote
\begin{equation*}
\begin{split}
\inf\bigg\{\int_{\{\psi_1<-t\}}|\tilde{f}|^{2}e^{-\varphi}:(\tilde{f}-w^{\alpha}dw_1\wedge\ldots\wedge dw_{n_1}&,z_0)\in(\mathcal{O}(K_{M'})\otimes\mathcal{I}(\psi_1))_{z_0} \\&\&{\,}\tilde{f}\in H^{0}(\{\psi_1<-t\},\mathcal{O}(K_{M'}))\bigg\}
\end{split}
\end{equation*}
by $ G_{\alpha}(t)$, where $t\ge0$. Lemma \ref{l:G equal to 0} and Lemma \ref{l:0} show that $G_{\alpha}(t)\not=0$ for any $\alpha\in E$. It follows from equality \eqref{eq:211218d} that
\begin{equation}
	\label{eq:211219a}G(t,\tilde c\equiv1)=\sum_{\alpha\in E}G_{\alpha}(t)\int_Y|f_{\alpha}|^2e^{-\varphi}.
\end{equation}
  Theorem \ref{thm:general_concave} tells us that $G_{\alpha}(-\log r)$ is concave with respect to $r$. It follows from the linearity of $G(-\log r;\tilde{c})$ and equality \eqref{eq:211219a} that $G_{\alpha}(-\log r)$ is linear with respect to $r$ for any $\alpha\in E$ satisfying $f_{\alpha}\not\equiv0$.  It follows from Theorem \ref{thm:linear-2d} and the linearity of $G_{\alpha}(-\log r)$ that statements $(2)$ and $(3)$ in Theorem \ref{thm:linear-fibra-single} hold.

Thus,  the necessity part of Theorem \ref{thm:linear-fibra-single} holds.

\section{Proofs of Theorem \ref{thm:linear-fibra-finite} and Reamrk \ref{r:fibra-finite}}\label{sec:proof-2}

In this section, we prove Theorem \ref{thm:linear-fibra-finite} and Remark \ref{r:fibra-finite}.

Denote that $M':=\prod_{1\le j\le n_1}\Omega_j$, and let $K_{M'}$ be the conanical (holomorphic) line bundle on $M'$. Denote that
$$\psi_1:=\max_{1\le j\le n_1}\left\{\tilde\pi_j^*\left(2\sum_{1\le k\le m_j}p_{j,k}G_{\Omega_j}(\cdot,z_{j,k})\right)\right\}$$
 on $M'$, where $\tilde\pi_j$ is the natural projection from $M'$ to $\Omega_j$.
For any $\beta\in I_1$ and any holomorphic function $h$, it follows from Lemma \ref{l:0'} that $(h,(z_{\beta},y))\in\mathcal{I}(\psi)_{(z_{\beta},y)}$ for any $y\in Y$ if and only if $(h(\cdot,y),z_{\beta})\in\mathcal{I}(\psi_1)_{z_{\beta}}$ for any $y\in Y$.
The sufficiency part of Theorem \ref{thm:linear-fibra-finite} follows from Proposition \ref{p:fibra}, Theorem \ref{thm:prod-finite-point} and Lemma \ref{l:G1=G2}. In the following, we prove the necessity part of Theorem \ref{thm:linear-fibra-finite} and Remark \ref{r:fibra-finite}.

Following from the linearity of $G(h^{-1}(r))$ and Corollary \ref{c:linear}, there exists a holomorphic $(n,0)$ form $F$ on $M$, such that $(F-f,z)\in(\mathcal{O}(K_M)\otimes\mathcal{I}(\varphi+\psi))_{z}$ for any $z\in Z_0$ and
\begin{equation}
	\label{eq:211221a}G(t)=\int_{\{\psi<-t\}}|F|^2e^{-\varphi}c(-\psi).
\end{equation}
It follows from Lemma \ref{l:green-sup2} and Lemma \ref{l:G-compact} that there exists $t_0>0$ such that $\{\psi_1<-t_0\}\Subset\cup_{\beta\in I_1}V_{\beta}$ and $\left\{z\in\Omega_j:2\sum_{1\le k\le m_j}p_{j,k}G_{\Omega_j}(z,z_{j,k})<-t_0\right\}\cap V_{z_{j,k}}$ is simply connected for any $j\in\{1,\ldots,n_1\}$ and $k\in\{1,\ldots,m_j\}$. For any $\beta\in I_1$,
 denote
\begin{equation*}
\begin{split}
\inf\bigg\{\int_{\{\psi<-t\}\cap(V_{\beta}\times Y)}|\tilde{f}|^{2}&e^{-\varphi}c(-\psi):\tilde{f}\in H^{0}\left(\{\psi<-t\}\cap(V_{\beta}\times Y),\mathcal{O}(K_{M})\right)\\&\&\,(\tilde{f}-f,(z_{\beta},y))\in(\mathcal{O}(K_M)\otimes\mathcal{I}(\varphi+\psi))_{(z_{\beta},y)}{,}\,\forall y\in Y \bigg\}
\end{split}
\end{equation*}
by $G_{\beta}(t)$,  where $t\in[t_0,+\infty)$. Note that $\{\psi<-t\}=\cup_{\beta\in I_1}(\{\psi<-t\}\cap(V_{\beta}\times Y))$ for any $t\ge t_0$.  Following from the definition of $G(t)$ and $G_{\beta}(t)$, we have $G(t)=\sum_{\beta\in I_1}G_{\beta}(t)$ for $t\ge t_0$. Thus, we have
$$G_{\beta}(t)=\int_{\{\psi<-t\}\cap(V_{\beta}\times Y)}|F|^2e^{-\varphi}c(-\psi)$$
for any $t\ge t_0$.
 Theorem \ref{thm:general_concave} tells us that $G_{\beta}(h^{-1}(r))$ is concave with respect to $r\in(0,\int_{t_0}^{+\infty}c(s)e^{-s}ds]$. As $G(h^{-1}(r))$ is linear with respect to $r$, we have $G_{\beta}(h^{-1}(r))$ is linear with respect to $r\in(0,\int_{t_0}^{+\infty}c(s)e^{-s}ds]$.

Note that $f=\pi_1^*\left(w_{\beta^*}^{\alpha_{\beta^*}}dw_{1,1}\wedge\ldots\wedge dw_{n_1,1}\right)\wedge\pi_2^*\left(f_{\alpha_{\beta^*}}\right)+\sum_{\alpha\in E'}\pi_1^*(w^{\alpha}dw_{1,1}\wedge\ldots\wedge dw_{n_1,1})\wedge\pi_2^*(f_{\alpha})$ on $V_{\beta^*}\times Y$, where   $E'=\left\{\alpha\in\mathbb{Z}_{\ge0}^{n_1}:\sum_{j=1}^{n_1}\frac{\alpha_j+1}{p_{j,1}}>\sum_{j=1}^{n_1}\frac{\alpha_{\beta^*,j}+1}{p_{j,1}}\right\}$.  As $\frac{1}{2p_{j,1}}\left(2\sum_{1\le k\le m_j}p_{j,k}G_{\Omega_j}(\cdot,z_{j,k})+t_0\right)$ is the Green function on $\big\{z\in\Omega_j:2\sum_{1\le k\le m_j}p_{j,k}G_{\Omega_j}(z,z_{j,k})<-t_0\big\}\cap V_{z_{j,1}}$, it follows from Theorem \ref{thm:linear-fibra-single} that $\big(f-\sum_{\alpha\in E_{\beta^*}}\pi_1^*(w_{\beta^*}^{\alpha}dw_{1,1,}\wedge\ldots\wedge dw_{n_1,1})\wedge\pi_2^*(\tilde f_{\alpha}),(z_{\beta^*},y)\big)\in(\mathcal{O}(K_M)\otimes\mathcal{I}(\varphi+\psi))_{(z_{\beta^*},y)}$ for any $y\in Y$, where $E_{\beta^*}=\left\{\alpha\in\mathbb{Z}_{\ge0}^{n_1}:\sum_{1\le j\le n_1}\frac{\alpha_j+1}{p_{j,\beta^*_j}}=1\right\}$ and $\tilde f_{\alpha}$ is a holomorphic $(n_2,0)$ form on $Y$ satisfying $\int_Y|\tilde f_{\alpha}|^2e^{-\varphi_Y}<+\infty$ for any $\alpha\in E_{\beta^*}$.   Following from Lemma \ref{l:0} and Lemma \ref{l:0'}, we have $\alpha_{\beta^*}\in E_{\beta^*}$, $f_{\alpha_{\beta^*}}=\tilde f_{\alpha_{\beta^*}}$ and $\tilde f_{\alpha}\equiv0$ for any $\alpha\not=\alpha_{\beta^*}$. Using Theorem \ref{thm:linear-fibra-single} and Remark \ref{r:fibra-single}, we obtain that there exists a holomorphic $(n_1,0)$ form $h_{0}$ on $\{\psi_1<-t_0\}\cap V_{\beta^*}$ such that
$$F=\pi_1^*(h_0)\wedge\pi_2^*(f_{\alpha_{\beta^*}})$$
 on $(\{\psi_1<-t_0\}\cap V_{\beta^*})\times Y$.  It follows from Lemma \ref{l:decom-product} that there exists a holomorphic $(n_1,0)$ form $h_{1}$ on $M'$ such that
 \begin{equation}
 	\label{eq:211221b}F=\pi_1^*(h_1)\wedge\pi_2^*(f_{\alpha_{\beta^*}})
 \end{equation}
on $M$ and $h_{0}=h_{1}$ on $\{\psi_1<-t_0\}\cap V_{\beta^*}$.

Denote that $\tilde\varphi=\sum_{1\le j\le n_1}\tilde\pi_j^*(\varphi_j)$ on $M'$.
Denote
\begin{equation*}
\begin{split}
\inf\bigg\{\int_{\{\psi_1<-t\}}|\tilde{f}|^{2}e^{-\tilde\varphi}c(-\psi_1):(\tilde{f}-h_1,z_{\beta})\in&(\mathcal{O}(K_{M'})\otimes\mathcal{I}(\psi_1))_{z_{\beta}}{,}\,\forall \beta\in I_1 \\&\&{\,}\tilde{f}\in H^{0}(\{\psi_1<-t\},\mathcal{O}(K_{M'}))\bigg\}
\end{split}
\end{equation*}
by $G'(t)$,  where $t\in[0,+\infty)$.
 Note that $f_{\alpha_{\beta^*}}=\tilde f_{\alpha_{\beta^*}}$ satisfies $\int_Y|f_{\alpha_{\beta^*}}|^2e^{-\varphi_Y}<+\infty$. For any $\beta\in I_1$ and any holomorphic function $h$, note that $(h,(z_{\beta},y))\in\mathcal{I}(\psi)_{(z_{\beta},y)}$ for any $y\in Y$ if and only if $(h(\cdot,y),z_{\beta})\in\mathcal{I}(\psi_1)_{z_{\beta}}$ for any $y\in Y$.  Following from Lemma \ref{l:G1=G2}, equality \eqref{eq:211221b} and Proposition \ref{p:fibra}, we get that $G'(0)<+\infty$, $G'(h^{-1}(r))$ is linear with respect to $r\in(0,\int_0^{+\infty}c(s)e^{-s}ds]$
 and \begin{equation}
 	\label{eq:211221c}G'(t)=\int_{\{\psi_1<-t\}}|h_1|^2e^{-\tilde\varphi}c(-\psi_1)
 \end{equation}
 for any $t\ge0$.
 Theorem \ref{thm:prod-finite-point} tells us that the following statements hold:

$(1)$ $\varphi_j=2\log|g_j|+2u_j$ for any $j\in\{1,\ldots,n\}$, where $u_j$ is a harmonic function on $\Omega_j$ and $g_j$ is a holomorphic function on $\Omega_j$ satisfying $g_j(z_{j,k})\not=0$ for any $k\in\{1,\ldots,m_j\}$;
	
	$(2)$ There exists a nonnegative integer $\gamma_{j,k}$ for any $j\in\{1,\ldots,n_1\}$ and $k\in\{1,\ldots,m_j\}$, which satisfies that $\prod_{1\le k\leq m_j}\chi_{j,z_{j,k}}^{\gamma_{j,k}+1}=\chi_{j,-u_j}$ and $\sum_{1\le j\le n_1}\frac{\gamma_{j,\beta_j}+1}{p_{j,\beta_j}}=1$ for any $\beta\in I_1$;
	
	$(3)$ $h_1=\left(c_{\beta}\prod_{1\le j\le n_1}w_{j,\beta_j}^{\gamma_{j,\beta_j}}+\tilde g_{\beta}\right)dw_{1,\beta_1}\wedge\ldots\wedge dw_{n_1,\beta_{n_1}}$ on $V_{\beta}$ for any $\beta\in I_1$, where $c_{\beta}$ is a constant and $g_{\beta}$ is a holomorphic function on $V_{\beta}$ such that $(g_{\beta},z_{\beta})\in\mathcal{I}(\psi_1)_{z_{\beta}}$;
	
	$(4)$ $\lim_{z\rightarrow z_{\beta}}\frac{c_{\beta}\prod_{1\le j\le n_1}w_{j,\beta_j}^{\gamma_{j,\beta_j}}dw_{1,\beta_1}\wedge\ldots\wedge dw_{n_1,\beta_{n_1}}}{\wedge_{1\le j\le n_1}\tilde\pi_{j}^*\left(g_j(P_{j})_*\left(f_{u_j}\left(\prod_{1\le k\le m_j}f_{z_{j,k}}^{\gamma_{j,k}+1}\right)\left(\sum_{1\le k\le m_j}p_{j,k}\frac{df_{z_{j,k}}}{f_{z_{j,k}}}\right)\right)\right)}=c_0$ for any $\beta\in I_1$, where $c_0\in\mathbb{C}\backslash\{0\}$ is a constant independent of $\beta$, $f_{u_j}$ is a holomorphic function $\Delta$ such that $|f_{u_j}|=P_j^*(e^{u_j})$ and $f_{z_{j,k}}$ is a holomorphic function on $\Delta$ such that $|f_{z_{j,k}}|=P_j^*\left(e^{G_{\Omega_j}(\cdot,z_{j,k})}\right)$ for any $j\in\{1,\ldots,n\}$ and $k\in\{1,\ldots,m_j\}$.

As $\int_Y|f_{\alpha_{\beta^*}}|^2e^{-\varphi_Y}<+\infty$ and $\tilde\varphi(z_{\beta})>-\infty$ for any $\beta\in I_1$, it follows from Lemma \ref{l:phi1+phi2} that $\pi_1^*(\tilde g_{\beta}dw_{1,\beta_1}\wedge\ldots\wedge dw_{n_1,\beta_{n_1}})\wedge\pi_2^*(f_{\alpha_{\beta^*}}),z)\in(\mathcal{O}(K_M)\otimes\mathcal{I}(\varphi+\psi))_z$ for any $z\in Z_0$.  As $(F-f,z)\in(\mathcal{O}(K_M)\otimes\mathcal{I}(\varphi+\psi))_z$ for any $z\in Z_0$ and $F=\pi_1^*(h_1)\wedge\pi_2^*(f_{\alpha_{\beta^*}})$, we have
$$f=\pi_1^*\left(c_{\beta}\left(\prod_{1\le j\le n_1}w_{j,\beta_j}^{\gamma_{j,\beta_j}}\right)dw_{1,\beta_1}\wedge\ldots\wedge dw_{n,\beta_n}\right)\wedge\pi_2^*\left(f_{\alpha_{\beta^*}}\right)+g_\beta$$
 on $V_{\beta}\times Y$ for any $\beta\in I_1$, where $g_{\beta}$ is a holomorphic $(n,0)$ form on $V_{\beta}\times Y$ such that $(g_{\beta},z)\in(\mathcal{O}(K_M)\otimes\mathcal{I}(\varphi+\psi))_{z}$ for any $z\in\{z_\beta\}\times Y$. Take $f_0=f_{\alpha_{\beta^*}}$. Thus, Theorem \ref{thm:linear-fibra-finite} holds.

Note that $G'(h^{-1}(r))$ is linear with respect to $r$. Following from Theorem \ref{thm:prod-finite-point}, Remark \ref{r:1.2} and equality \eqref{eq:211221c}, we have
$$h_1=c_0\wedge_{1\le j\le n_1}\tilde\pi_{j}^*\left(g_j(P_{j})_*\left(f_{u_j}\left(\prod_{1\le k\le m_j}f_{z_{j,k}}^{\gamma_{j,k}+1}\right)\left(\sum_{1\le k\le m_j}p_{j,k}\frac{df_{z_{j,k}}}{f_{z_{j,k}}}\right)\right)\right)$$
and
\begin{displaymath}
	\begin{split}
		G'(t)&=\int_{\{\psi_1<-t\}}|h_1|^2e^{-\tilde\varphi}c(-\psi_1)\\
		&=\left(\int_{t}^{+\infty}c(s)e^{-s}ds\right)\sum_{\beta\in I_1}\frac{|c_{\beta}|^2(2\pi)^{n_1}e^{-\tilde\varphi(z_{\beta})}}{\prod_{1\le j\le n_1}(\gamma_{j,\beta_j}+1)c_{j,\beta_j}^{2\gamma_{j,\beta_j}+2}}.
	\end{split}
\end{displaymath}
Thus, we have
$$F=c_0\left(\wedge_{1\le j\le n_1}\tilde\pi_{j}^*\left(g_j(P_{j})_*\left(f_{u_j}\left(\prod_{k=1}^{m_j}f_{z_{j,k}}^{\gamma_{j,k}+1}\right)\left(\sum_{k=1}^{m_j}p_{j,k}\frac{df_{z_{j,k}}}{f_{z_{j,k}}}\right)\right)\right)\right)\wedge\pi_2^*(f_0)$$
and
\begin{displaymath}
	\begin{split}
		G(t)&=\int_{\{\psi<-t\}}|F|^2e^{-\varphi}c(-\psi)\\
		&=\left(\int_{t}^{+\infty}c(s)e^{-s}ds\right)\sum_{\beta\in I_1}\frac{|c_{\beta}|^2(2\pi)^{n_1}e^{-\tilde\varphi(z_{\beta})}}{\prod_{1\le j\le n_1}(\gamma_{j,\beta_j}+1)c_{j,\beta_j}^{2\gamma_{j,\beta_j}+2}}\int_Y|f_0|^2e^{-\varphi_Y}.
	\end{split}
\end{displaymath}
The uniqueness of $F$ follows from Corollary \ref{c:linear}. Thus, Remark \ref{r:fibra-finite} holds.

\section{Proofs of Theorem \ref{thm:linear-fibra-infinite} and Proposition \ref{p:M=M_1}}

In this section, we prove Theorem \ref{thm:linear-fibra-infinite} and Proposition \ref{p:M=M_1}.

\subsection{Proof of Theorem \ref{thm:linear-fibra-infinite}}
\

In this section, we prove Theorem \ref{thm:linear-fibra-infinite} by contradiction. Assume that $G(h^{-1}(r))$ is linear with respect to $r\in(0,\int_0^{+\infty}c(s)e^{-s}ds]$.

Denote that $M':=\prod_{1\le j\le n_1}\Omega_j$, and let $K_{M'}$ be the canonical (holomorphic) line bundle on $M'$. Denote that
$$\psi_1:=\max_{1\le j\le n_1}\left\{\tilde\pi_j^*\left(2\sum_{1\le k<\tilde m_j}p_{j,k}G_{\Omega_j}(\cdot,z_{j,k})\right)\right\}$$
 on $M'$, where $\tilde\pi_j$ is the natural projection from $M'$ to $\Omega_j$.
Following from the linearity of $G(h^{-1}(r))$ and Corollary \ref{c:linear}, there exists a holomorphic $(n,0)$ form $F$ on $M$, such that $(F-f,z)\in(\mathcal{O}(K_M)\otimes\mathcal{I}(\varphi+\psi))_{z}$ for any $z\in Z_0$ and
\begin{equation}
	\label{eq:211221d}G(t)=\int_{\{\psi<-t\}}|F|^2e^{-\varphi}c(-\psi).
\end{equation}
For any $\beta\in \tilde I_1$, it follows from Lemma \ref{l:green-sup2} and Lemma \ref{l:G-compact} that there exists $t_\beta>0$ such that $\{\psi_1<-t_{\beta}\}\cap V_{\beta}\Subset V_{\beta}$ and $\big\{z\in\Omega_j:2\sum_{1\le k<\tilde m_j}p_{j,k}G_{\Omega_j}(z,z_{j,k})<-t_0\big\}\cap V_{z_{j,k}}$ is simply connected for any $1\le j\le n_1$ and $1\le k<\tilde m_j$. For any $\beta\in\tilde I_1$,
 denote
\begin{equation*}
\begin{split}
\inf\bigg\{\int_{\{\psi<-t\}\cap(V_{\beta}\times Y)}|\tilde{f}|^{2}&e^{-\varphi}c(-\psi):\tilde{f}\in H^{0}(\{\psi<-t\}\cap(V_{\beta}\times Y),\mathcal{O}(K_{M}))\\&\&\,(\tilde{f}-f,(z_{\beta},y))\in(\mathcal{O}(K_M)\otimes\mathcal{I}(\varphi+\psi))_{(z_{\beta},y)}{,}\,\forall y\in Y\bigg\}
\end{split}
\end{equation*}
by $G_{\beta}(t)$,  where $t\in[t_\beta,+\infty)$, and  denote
\begin{equation*}
\begin{split}
\inf\bigg\{\int_{\{\psi<-t\}\backslash(V_{\beta}\times Y)}&|\tilde{f}|^{2}e^{-\varphi}c(-\psi):\tilde{f}\in H^{0}(\{\psi<-t\}\backslash(V_{\beta}\times Y),\mathcal{O}(K_{M}))\\&\&\,(\tilde{f}-f,z)\in(\mathcal{O}(K_M)\otimes\mathcal{I}(\varphi+\psi))_{z}{,}\,\forall z\in \left(\tilde I_1\backslash\{\beta\}\right)\times Y\bigg\}
\end{split}
\end{equation*}
by $\tilde G_{\beta}(t)$,  where $t\in[t_\beta,+\infty)$.
 By the definition of $G(t)$, $G_{\beta}(t)$ and $\tilde G_{\beta}(t)$, we have $G(t)=G_{\beta}(t)+\tilde G_{\beta}(t)$ for $t\ge t_\beta$. Thus, we have
$$G_{\beta}(t)=\int_{\{\psi<-t\}\cap(V_{\beta}\times Y)}|F|^2e^{-\varphi}c(-\psi)$$
for any $t\ge t_\beta$.
 Theorem \ref{thm:general_concave} tells us that $G_{\beta}(h^{-1}(r))$ and $\tilde G_{\beta}(h^{-1}(r))$ are concave with respect to $r\in(0,\int_{t_\beta}^{+\infty}c(s)e^{-s}ds]$. As $G(h^{-1}(r))$ is linear with respect to $r$, we have $G_{\beta}(h^{-1}(r))$ is linear with respect to $r\in(0,\int_{t_\beta}^{+\infty}c(s)e^{-s}ds]$.

 Following from Lemma \ref{l:green-sup} and \ref{l:green-sup2}, we know $\frac{1}{2p_{j,1}}\left(2\sum_{1\le k<\tilde m_j}p_{j,k}G_{\Omega_j}(\cdot,z_{j,k})+t_{\beta^*}\right)$ is the Green function on $\left\{z\in\Omega_j:2\sum_{1\le k<\tilde m_j}p_{j,k}G_{\Omega_j}(z,z_{j,k})<-t_{\beta^*}\right\}\cap V_{z_{j,1}}$. Note that   $f=\pi_1^*\left(w_{\beta^*}^{\alpha_{\beta^*}}dw_{1,1}\wedge\ldots\wedge dw_{n_1,1}\right)\wedge\pi_2^*\left(f_{\alpha_{\beta^*}}\right)+\sum_{\alpha\in E'}\pi_1^*(w^{\alpha}dw_{1,1}\wedge\ldots\wedge dw_{n_1,1})\wedge\pi_2^*(f_{\alpha})$ on $V_{\beta^*}\times Y$, where  $E'=\left\{\alpha\in\mathbb{Z}_{\ge0}^{n_1}:\sum_{j=1}^{n_1}\frac{\alpha_j+1}{p_{j,1}}>\sum_{j=1}^{n_1}\frac{\alpha_{\beta^*,j}+1}{p_{j,1}}\right\}$. It follows from Theorem \ref{thm:linear-fibra-single} that $\big(f-\sum_{\alpha\in E_{\beta^*}}\pi_1^*(w_{\beta^*}^{\alpha}dw_{1,1,}\wedge\ldots\wedge dw_{n_1,1})\wedge\pi_2^*(\tilde f_{\alpha}),(z_{\beta^*},y)\big)\in(\mathcal{O}(K_M)\otimes\mathcal{I}(\varphi+\psi))_{(z_{\beta^*},y)}$ for any $y\in Y$, where $E_{\beta^*}=\left\{\alpha\in\mathbb{Z}_{\ge0}^{n_1}:\sum_{1\le j\le n_1}\frac{\alpha_j+1}{p_{j,\beta^*_j}}=1\right\}$ and $\tilde f_{\alpha}$ is a holomorphic $(n_2,0)$ form on $Y$ satisfying $\int_Y|\tilde f_{\alpha}|^2e^{-\varphi_Y}<+\infty$ for any $\alpha\in E_{\beta^*}$.   Following from Lemma \ref{l:0} and Lemma \ref{l:0'}, we have $\alpha_{\beta^*}\in E_{\beta^*}$, $f_{\alpha_{\beta^*}}=\tilde f_{\alpha_{\beta^*}}$ and $\tilde f_{\alpha}\equiv0$ for any $\alpha\not=\alpha_{\beta^*}$. Using Theorem \ref{thm:linear-fibra-single} and Remark \ref{r:fibra-single}, we obtain that there exists a holomorphic $(n_1,0)$ form $h_{0}$ on $\{\psi_1<-t_{\beta^*}\}\cap V_{\beta^*}$ such that
$$F=\pi_1^*(h_0)\wedge\pi_2^*(f_{\alpha_{\beta^*}})$$
 on $(\{\psi_1<-t_{\beta^*}\}\cap V_{\beta^*})\times Y$.  It follows from Lemma \ref{l:decom-product} that there exists a holomorphic $(n_1,0)$ form $h_{1}$ on $M'$ such that
 \begin{equation}
 	\label{eq:211221e}F=\pi_1^*(h_1)\wedge\pi_2^*(f_{\alpha_{\beta^*}})
 \end{equation}
on $M$ and $h_{0}=h_{1}$ on $\{\psi_1<-t_{\beta^*}\}\cap V_{\beta^*}$.

Denote that $\tilde\varphi=\sum_{1\le j\le n_1}\tilde\pi_j^*(\varphi_j)$ on $M'$.
Denote
\begin{equation*}
\begin{split}
\inf\bigg\{\int_{\{\psi_1<-t\}}|\tilde{f}|^{2}e^{-\tilde\varphi}c(-\psi_1):(\tilde{f}-h_1,z_{\beta})\in&(\mathcal{O}(K_{M'})\otimes\mathcal{I}(\psi_1))_{z_{\beta}}{,}\,\forall \beta\in\tilde I_1 \\&\&{\,}\tilde{f}\in H^{0}(\{\psi_1<-t\},\mathcal{O}(K_{M'}))\bigg\}
\end{split}
\end{equation*}
by $G'(t)$,  where $t\in[0,+\infty)$.
 Note that $f_{\alpha_{\beta^*}}=\tilde f_{\alpha_{\beta^*}}$ satisfies $\int_Y|f_{\alpha_{\beta^*}}|^2e^{-\varphi_Y}<+\infty$. For any $\beta\in\tilde I_1$ and any holomorphic function $h$, it follows from Lemma \ref{l:0'} that $(h,(z_{\beta},y))\in\mathcal{I}(\psi)_{(z_{\beta},y)}$ for any $y\in Y$ if and only if $(h(\cdot,y),z_{\beta})\in\mathcal{I}(\psi_1)_{z_{\beta}}$ for any $y\in Y$. Following from Lemma \ref{l:G1=G2}, equality \eqref{eq:211221e} and Proposition \ref{p:fibra}, we get that $G'(0)<+\infty$ and $G'(h^{-1}(r))$ is linear with respect to $r\in(0,\int_0^{+\infty}c(s)e^{-s}ds]$, which contradicts to Theorem \ref{thm:prod-infinite-point}.

Thus, we obtain that $G(h^{-1}(r))$ is not linear.

\subsection{Proof of Proposition \ref{p:M=M_1}}
\

It follows from Corollary \ref{c:linear} that there exists a holomorphic $(n,0)$ form $F$ on $M_1$, which satisfies that $(F-f,z)\in(\mathcal{O}(K_{M_1})\otimes\mathcal{I}(\varphi+\psi))_z$ for any $z\in Z_0$ and
\begin{equation}
	\label{eq:1226a}
	G(t)=\int_{\{\psi<-t\}\cap M_1}|F|^2e^{-\varphi}c(-\psi)
\end{equation}
for any $t\ge0$.

It follows from Lemma \ref{l:green-sup} and Lemma \ref{l:green-sup2} that there exists a local coordinate $w_{j,k}$  on a neighborhood $V_{z_{j,k}}\Subset\Omega_{j}$ of $z_{j,k}\in\Omega_j$ satisfying $w_{j,k}(z_{j,k})=0$ and
$$\log|w_{j,k}|=\frac{1}{p_{j,k}}\sum_{1\le k<\tilde m_j}p_{j,k}G_{\Omega_j}(\cdot,z_{j,k})$$
 for any $j\in\{1,\ldots,n_1\}$ and $1\le k<\tilde{m}_j$, where $V_{z_{j,k}}\cap V_{z_{j,k'}}=\emptyset$ for any $j$ and $k\not=k'$. Denote that $\tilde I_1:=\{(\beta_1,\ldots,\beta_{n_1}):1\le \beta_j< \tilde m_j$ for any $j\in\{1,\ldots,n_1\}\}$, $V_{\beta}:=\prod_{1\le j\le n_1}V_{z_{j,\beta_j}}$ for any $\beta=(\beta_1,\ldots,\beta_{n_1})\in\tilde I_1$ and $w_{\beta}:=(w_{1,\beta_1},\ldots,w_{n_1,\beta_{n_1}})$ is a local coordinate on $V_{\beta}$ of $z_{\beta}:=(z_{1,\beta_1},\ldots,z_{n_1,\beta_{n_1}})\in M$.
 It follows from Lemma \ref{l:fibra-decom-2} that
$$F=\sum_{\alpha\in\mathbb{Z}_{\ge0}^{n_1}}\pi_1^*(w_{\beta}^{\alpha}dw_{1,\beta_1}\wedge\ldots\wedge dw_{n_1,\beta_{n_1}})\wedge\pi_2^*(F_{\alpha,\beta})$$
on a neighborhood $U_{\beta}\subset(V_{\beta}\times Y)\cap M_1$ of  $\{z_{\beta}\}\times Y$ for any $\beta\in\tilde I_1$, where $F_{\alpha,\beta}$ is a holomorphic $(n_2,0)$ form on $Y$.
Following from Lemma \ref{l:limit2} and equality \eqref{eq:1226a}, we obtain that
$$F_{\alpha,\beta}\equiv0$$
for any $\alpha\in\left\{\alpha\in\mathbb{Z}_{\ge0}^{n_1}:\sum_{1\le j\le n_1}\frac{\alpha_j+1}{p_{j,\beta_j}}<1\right\}$ and $\beta\in\tilde I_1$, and we have
\begin{equation}\label{eq:1226b}
	\frac{G(0)}{\int_0^{+\infty}c(s)e^{-s}ds}\ge \sum_{\beta\in\tilde I_1}\sum_{\alpha\in E_{\beta}}\frac{(2\pi)^{n_1}e^{-\sum_{1\le j\le n_1}\varphi_j(z_{j,\beta_j})}}{\prod_{1\le j\le n_1}(\alpha_j+1)}\int_{Y}|F_{\alpha,\beta}|^2e^{-\varphi_Y},
\end{equation}
where $E_{\beta}=\left\{\alpha\in\mathbb{Z}_{\ge0}^{n_1}:\sum_{1\le j\le n_1}\frac{\alpha_j+1}{p_{j,\beta_j}}=1\right\}$ for any $\beta\in \tilde I_1$.
Proposition \ref{p:exten-fibra} shows that there exists a holomorphic $(n,0)$ form $F_1$ on $M$ such that $(F_1-F,z)\in(\mathcal{O}(K_M)\otimes\mathcal{I}(\psi))_z$ for any $z\in Z_0$ and
\begin{equation}
\label{eq:1226c} \begin{split}
		&\int_{M}|F_1|^2e^{-\varphi}c(-\psi)\\
		\le&\left(\int_0^{+\infty}c(s)e^{-s}ds\right)\sum_{\beta\in\tilde I_1}\sum_{\alpha\in E_{\beta}}\frac{(2\pi)^{n_1}e^{-\sum_{1\le j\le n_1}\varphi_j(z_{j,\beta_j})}}{\prod_{1\le j\le n_1}(\alpha_j+1)}\int_{Y}|F_{\alpha,\beta}|^2e^{-\varphi_Y}.
\end{split}
\end{equation}
Denote that $\tilde E_{\beta}:=\left\{\alpha\in\mathbb{Z}_{\ge0}^{n_1}:\sum_{1\le j\le n_1}\frac{\alpha_j+1}{p_{j,\beta_j}}\ge1\right\}$ for any $\beta\in \tilde I_1$. As $(F_1-F,z)\in(\mathcal{O}(K_M)\otimes\mathcal{I}(\psi))_z$. It follows from Lemma \ref{l:fibra-decom} and Lemma \ref{l:0} that
\begin{displaymath}
	\begin{split}
	F_1=&\sum_{\alpha\in E_{\beta}}\pi_1^*(w_{\beta}^{\alpha}dw_{1,\beta_1}\wedge\ldots\wedge dw_{n_1,\beta_{n_1}})\wedge\pi_2^*(F_{\alpha,\beta})	\\
	&+\sum_{\alpha\in\tilde E_{\beta}\backslash E_{\beta}}\pi_1^*(w_{\beta}^{\alpha}dw_{1,\beta_1}\wedge\ldots\wedge dw_{n_1,\beta_{n_1}})\wedge\pi_2^*(\tilde F_{\alpha,\beta})	\end{split}
\end{displaymath}
on a neighborhood  of  $\{z_{\beta}\}\times Y$ for any $\beta\in\tilde I_1$, where $\tilde F_{\alpha,\beta}$ is a holomorphic $(n_2,0)$ form on $Y$. It follows from Lemma \ref{l:fibra-decom-2} that $(F_{\alpha,\beta},y)\in(\mathcal{O}(K_Y)\otimes\mathcal{I}(\varphi_Y))_y$ and $(\tilde F_{\alpha,\beta},y)\in(\mathcal{O}(K_Y)\otimes\mathcal{I}(\varphi_Y))_y$ for any $y\in Y$. Using Lemma \ref{l:0} and Lemma \ref{l:phi1+phi2}, we obtain that $(F_1-F,z)\in(\mathcal{O}(K_M)\otimes\mathcal{I}(\varphi+\psi))_z$ for any $z\in Z_0$. Combining inequality \eqref{eq:1226b} and \eqref{eq:1226c}, we have
\begin{displaymath}
	\begin{split}
		\frac{G(0)}{\int_0^{+\infty}c(s)e^{-s}ds}=\int_{M}|F_1|^2e^{-\varphi}c(-\psi)=\int_{M_1}|F_1|^2e^{-\varphi}c(-\psi),
	\end{split}
\end{displaymath}
which implies that $M_1=M$.

\section{Proofs of Theorem \ref{thm:exten-fibra-single} and Remark \ref{r:1.7}}

In this section, we prove Theorem \ref{thm:exten-fibra-single} and Remark \ref{r:1.7}.

\subsection{Proof of Theorem \ref{thm:exten-fibra-single}}
\

As $c(t)e^{-t}$ is decreasing and $\Psi\le0$, it follows from Proposition \ref{p:exten-fibra} that
there exists a holomorphic $(n,0)$ form $F$ on $M$, which satisfies that $(F-f,z)\in\left(\mathcal{O}(K_{M})\otimes\mathcal{I}\left(\max_{1\le j\le n_1}\left\{2p_j\pi_{1,j}^{*}(G_{\Omega_j}(\cdot,z_j))\right\}\right)\right)_{z}$ for any $z\in Z_0$ and
\begin{equation}\label{eq:211224a}
	\begin{split}
	&\int_{M_1}|F|^2e^{-\varphi}c(-\psi)\\
	\le& \int_{M}|F|^2e^{-\varphi-\pi_1^*(\Psi)}c(-\psi+\pi_1^*(\Psi))\\
	\le&\left(\int_0^{+\infty}c(s)e^{-s}ds\right)\sum_{\alpha\in E}\frac{(2\pi)^{n_1}e^{-\left(\Psi+\sum_{1\le j\le n_1}\tilde\pi_j^*(\varphi_j)\right)(z_{0})}}{\prod_{1\le j\le n_1}(\alpha_j+1)c_{j}(z_j)^{2\alpha_{j}+2}}\int_Y|f_{\alpha}|^2e^{-\varphi_Y}.	\end{split}	
\end{equation}

If $\Psi\equiv0$,
as $(F-f,z)\in\left(\mathcal{O}(K_{M})\otimes\mathcal{I}\left(\max_{1\le j\le n_1}\left\{2p_j\pi_{1,j}^{*}(G_{\Omega_j}(\cdot,z_j))\right\}\right)\right)_{z}$ for any $z\in Z_0$, it follows from Lemma \ref{l:0} and Lemma \ref{l:fibra-decom} that we have $F=\sum_{\alpha\in E}\pi_1^*(w^{\alpha}dw_1\wedge\ldots\wedge dw_{n_1})\wedge\pi_2^*(f_{\alpha})+\sum_{\alpha\in\tilde E\backslash E}\pi_1^*(w^{\alpha}dw_1\wedge\ldots\wedge dw_{n_1})\wedge\pi_2^*(\tilde f_{\alpha})$ on $V_0\times Y$, where $\tilde f_{\alpha}$ is a holomorphic $(n_2,0)$ form on $ Y$ satisfying $\int_Y|\tilde f_{\alpha}|^2e^{-\varphi_Y}<+\infty$ for any $\alpha\in\tilde E\backslash E$. Note that $\left(\Psi+\sum_{1\le j\le n_1}\tilde\pi_j^*(\varphi_j)\right)(z_0)>-\infty$. It follows from Lemma \ref{l:0}, Lemma \ref{l:phi1+phi2} and Lemma \ref{l:closedness} that $\left(\sum_{\alpha\in\tilde E\backslash E}\pi_1^*(w^{\alpha}dw_1\wedge\ldots\wedge dw_{n_1})\wedge\pi_2^*(\tilde f_{\alpha}),z\right)\in(\mathcal{O}(K_{M})\otimes\mathcal{I}(\varphi+\psi))_{z}$ for any $z\in Z_0$.

In the following, we prove the characterization of the holding of the equality in Theorem \ref{thm:exten-fibra-single}.

Firstly, we prove the necessity.
Using inequality \eqref{eq:211224a}, we have
$$\int_{M_1}|F|^2e^{-\varphi}c(-\psi)= \int_{M}|F|^2e^{-\varphi-\pi_1^*(\Psi)}c(-\psi+\pi_1^*(\Psi)).$$ Note that $c(t)e^{-t}$ is decreasing. As $F\not\equiv0$, we get that
$$M_1=M=\left(\prod_{1\le j\le n_1}\Omega_j\right)\times Y.$$ As  $\Psi\le0$, it follows from Lemma \ref{l:psi=G} that $\Psi\equiv0$, i.e.,
$$\psi=\max_{1\le j\le n_1}\left\{\pi_{1,j}^*(2p_jG_{\Omega_j}(\cdot,z_j))\right\}.$$
Denote
\begin{displaymath}
	\begin{split}
		\inf\bigg\{\int_{\{\psi<-t\}}|\tilde f|^2e^{-\varphi}&c(-\psi):\tilde f\in H^0(\{\psi<-t\},\mathcal{O}(K_M))\\
		&\&\,(\tilde f-F,z)\in(\mathcal{O}(K_M)\otimes\mathcal{I}(\varphi+\psi))_z\mbox{ for any $z\in Z_0$}\bigg\}
	\end{split}
\end{displaymath}
by $G(t)$, where $t\ge0$.
 Denote
\begin{displaymath}
	\begin{split}
		\inf\bigg\{\int_{\{\psi<-t\}}|\tilde f|^2e^{-\varphi}&c(-\psi):\tilde f\in H^0(\{\psi<-t\},\mathcal{O}(K_M))\\
		&\&\,(\tilde f-F,z)\in(\mathcal{O}(K_M)\otimes\mathcal{I}(\psi))_z\mbox{ for any $z\in Z_0$}\bigg\}
	\end{split}
\end{displaymath}
by $\tilde G(t)$, where $t\ge0$.
It follows from Lemma \ref{l:G1=G2} that $G(t)=\tilde G(t)$ for any $t\ge0$.
Let $t\ge0$. It follows from Proposition \ref{p:exten-fibra}  ($M\sim\{\psi<-t\}$, $\psi\sim\psi+t$ and $c(\cdot)\sim c(\cdot+t)$, here $\sim$ means the former replaced by the latter) that there exists a holomorphic $(n,0)$ form $F_t$ on $\{\psi<-t\}$ satisfying that $(F_t-F,z)\in(\mathcal{O}(K_M)\otimes\mathcal{I}(\psi))_{z}$ for any $z\in Z_0$ and
\begin{equation}\label{eq:211224b} \begin{split}
&\int_{\{\psi<-t\}}|F_t|^2e^{-\varphi}c(-\psi)\\
\le&\left(\int_t^{+\infty}c(s)e^{-s}ds\right)\sum_{\alpha\in E}\frac{(2\pi)^{n_1}e^{-\sum_{1\le j\le n_1}\varphi_j(z_{j})}}{\prod_{1\le j\le n_1}(\alpha_j+1)c_{j}(z_j)^{2\alpha_{j}+2}}\int_Y|f_{\alpha}|^2e^{-\varphi_Y}.	
\end{split}\end{equation}
Following from inequality \eqref{eq:211224b}, we have
\begin{equation*}
	\frac{\tilde G(t)}{\int_t^{+\infty}c(s)e^{-s}ds}\leq	\sum_{\alpha\in E}\frac{(2\pi)^{n_1}e^{-\sum_{1\le j\le n_1}\varphi_j(z_{j})}}{\prod_{1\le j\le n_1}(\alpha_j+1)c_{j}(z_j)^{2\alpha_{j}+2}}\int_Y|f_{\alpha}|^2e^{-\varphi_Y}\end{equation*}
holds for any $t\geq0$.
Note that
$$\tilde G(0)=\left(\int_0^{+\infty}c(s)e^{-s}ds\right)\sum_{\alpha\in E}\frac{(2\pi)^{n_1}e^{-\sum_{1\le j\le n_1}\varphi_j(z_{j})}}{\prod_{1\le j\le n_1}(\alpha_j+1)c_{j}(z_j)^{2\alpha_{j}+2}}\int_Y|f_{\alpha}|^2e^{-\varphi_Y}.$$
 Combining Theorem \ref{thm:general_concave}, we obtain that  $\tilde G({h}^{-1}(r))$ is linear with respect to $r$, which implies that $G({h}^{-1}(r))$ is linear with respect to $r$, where $h(t)=\int_t^{+\infty}c(s)e^{-s}ds$.
It follows from Theorem \ref{thm:linear-fibra-single} that   statements  $(2)$ and $(3)$ in Theorem \ref{thm:exten-fibra-single} hold.

Now, we prove the sufficiency. Following from Remark \ref{r:fibra-single} and $G(0)=\tilde G(0)$, we obtain that
$$\tilde G(0)=\left(\int_0^{+\infty}c(s)e^{-s}ds\right)\sum_{\alpha\in E}\frac{(2\pi)^{n_1}e^{-\sum_{1\le j\le n_1}\varphi_j(z_{j})}}{\prod_{1\le j\le n_1}(\alpha_j+1)c_{j}(z_j)^{2\alpha_{j}+2}}\int_Y|f_{\alpha}|^2e^{-\varphi_Y}.$$

Thus, Theorem \ref{thm:exten-fibra-single} holds.

\subsection{Proof of Remark \ref{r:1.7}}
\label{sec:proof-1.7}
\

Note that $\left(\Psi+\sum_{1\le j\le n_1}\tilde\pi_j^*(\varphi_j)\right)(z_{0})>-\infty$.  As $(f_{\alpha},y)\in(\mathcal{O}(K_Y)\otimes\mathcal{I}(\varphi_Y))_y$ for any $y\in Y$ and $\alpha\in\tilde E\backslash E$, following from Lemma \ref{l:phi1+phi2}, Lemma \ref{l:0} and Lemma \ref{l:closedness}, we get that $\left(\sum_{\alpha\in \tilde E\backslash E}\pi_1^*(w^{\alpha}dw_1\wedge\ldots dw_{n_1})\wedge\pi_2^*(f_{\alpha}),z\right)\in(\mathcal{O}(K_{M_1})\otimes\mathcal{I}(\varphi+\psi))_z$ for any $z\in Z_0$.

As $c(t)e^{-t}$ is decreasing and $\Psi\le0$, it follows from Proposition \ref{p:exten-fibra} that
there exists a holomorphic $(n,0)$ form $F$ on $M$, which  satisfies that $(F-f,z)\in\left(\mathcal{O}(K_{M})\otimes\mathcal{I}\left(\max_{1\le j\le n_1}\left\{2p_j\pi_{1,j}^{*}(G_{\Omega_j}(\cdot,z_j))\right\}\right)\right)_{z}$ for any $z\in Z_0$ and
\begin{equation}\label{eq:211224c}
	\begin{split}
	&\int_{M_1}|F|^2e^{-\varphi}c(-\psi)\\
	\le& \int_{M}|F|^2e^{-\varphi-\pi_1^*(\Psi)}c(-\psi+\pi_1^*(\Psi))\\
	\le&\left(\int_0^{+\infty}c(s)e^{-s}ds\right)\sum_{\alpha\in E}\frac{(2\pi)^{n_1}e^{-\left(\Psi+\sum_{1\le j\le n_1}\tilde\pi_j^*(\varphi_j)\right)(z_{0})}}{\prod_{1\le j\le n_1}(\alpha_j+1)c_{j}(z_j)^{2\alpha_{j}+2}}\int_Y|f_{\alpha}|^2e^{-\varphi_Y}.	\end{split}	
\end{equation}

If $\Psi\equiv0$,
as $(F-f,z)\in\left(\mathcal{O}(K_{M})\otimes\mathcal{I}\left(\max_{1\le j\le n_1}\left\{2p_j\pi_{1,j}^{*}(G_{\Omega_j}(\cdot,z_j))\right\}\right)\right)_{z}$ for any $z\in Z_0$, it follows from Lemma \ref{l:0} and Lemma \ref{l:fibra-decom} that  $F=\sum_{\alpha\in E}\pi_1^*(w^{\alpha}dw_1\wedge\ldots\wedge dw_{n_1})\wedge\pi_2^*(f_{\alpha})+\sum_{\alpha\in\tilde E\backslash E}\pi_1^*(w^{\alpha}dw_1\wedge\ldots\wedge dw_{n_1})\wedge\pi_2^*(\tilde f_{\alpha})$ on $V_0\times Y$, where $\tilde f_{\alpha}$ is a holomorphic $(n_2,0)$ form on $ Y$ satisfying $\int_Y|\tilde f_{\alpha}|^2e^{-\varphi_Y}<+\infty$ for any $\alpha\in\tilde E\backslash E$. Note that $\left(\Psi+\sum_{1\le j\le n_1}\tilde\pi_j^*(\varphi_j)\right)(z_0)>-\infty$. It follows from Lemma \ref{l:0}, Lemma \ref{l:phi1+phi2} and Lemma \ref{l:closedness} that $\left(\sum_{\alpha\in\tilde E\backslash E}\pi_1^*(w^{\alpha}dw_1\wedge\ldots\wedge dw_{n_1})\wedge\pi_2^*(\tilde f_{\alpha}),z\right)\in(\mathcal{O}(K_{M})\otimes\mathcal{I}(\varphi+\psi))_{z}$ for any $z\in Z_0$.
Thus, we have $(F-f,z)\in(\mathcal{O}(K_{M_1})\otimes\mathcal{I}(\varphi+\psi))_z$ for any $z\in Z_0$.

In the following, we prove the characterization of the holding of the equality (replacing  the ideal sheaf $\mathcal{I}\left(\max_{1\le j\le n_1}\left\{2p_j\pi_{1,j}^{*}(G_{\Omega_j}(\cdot,z_j))\right\}\right)$  by $\mathcal{I}(\varphi+\psi)$) in Theorem \ref{thm:exten-fibra-single}.

Firstly, we prove the necessity.
Using inequality \eqref{eq:211224c}, we have
$$\int_{M_1}|F|^2e^{-\varphi}c(-\psi)= \int_{M}|F|^2e^{-\varphi-\pi_1^*(\Psi)}c(-\psi+\pi_1^*(\Psi)).$$ Note that $c(t)e^{-t}$ is decreasing. As $F\not\equiv0$, we get that
$$M_1=M=\left(\prod_{1\le j\le n_1}\Omega_j\right)\times Y.$$ As  $\Psi\le0$, it follows from Lemma \ref{l:psi=G} that $\Psi\equiv0$, i.e.,
$$\psi=\max_{1\le j\le n_1}\left\{\pi_{1,j}^*(2p_jG_{\Omega_j}(\cdot,z_j))\right\}.$$
Denote
\begin{displaymath}
	\begin{split}
		\inf\bigg\{\int_{\{\psi<-t\}}|\tilde f|^2e^{-\varphi}&c(-\psi):\tilde f\in H^0(\{\psi<-t\},\mathcal{O}(K_M))\\
		&\&\,(\tilde f-F,z)\in(\mathcal{O}(K_M)\otimes\mathcal{I}(\varphi+\psi))_z\mbox{ for any $z\in Z_0$}\bigg\}
	\end{split}
\end{displaymath}
by $G(t)$, where $t\ge0$.
 Denote
\begin{displaymath}
	\begin{split}
		\inf\bigg\{\int_{\{\psi<-t\}}|\tilde f|^2e^{-\varphi}&c(-\psi):\tilde f\in H^0(\{\psi<-t\},\mathcal{O}(K_M))\\
		&\&\,(\tilde f-F,z)\in(\mathcal{O}(K_M)\otimes\mathcal{I}(\psi))_z\mbox{ for any $z\in Z_0$}\bigg\}
	\end{split}
\end{displaymath}
by $\tilde G(t)$, where $t\ge0$.
It follows from Lemma \ref{l:G1=G2} that $G(t)=\tilde G(t)$ for any $t\ge0$.
Let $t\ge0$. It follows from Proposition \ref{p:exten-fibra}  ($M\sim\{\psi<-t\}$, $\psi\sim\psi+t$ and $c(\cdot)\sim c(\cdot+t)$, here $\sim$ means the former replaced by the latter) that
\begin{equation*}
	\frac{\tilde G(t)}{\int_t^{+\infty}c(s)e^{-s}ds}\leq	\sum_{\alpha\in E}\frac{(2\pi)^{n_1}e^{-\sum_{1\le j\le n_1}\varphi_j(z_{j})}}{\prod_{1\le j\le n_1}(\alpha_j+1)c_{j}(z_j)^{2\alpha_{j}+2}}\int_Y|f_{\alpha}|^2e^{-\varphi_Y}.\end{equation*}
Note that
$$ G(0)=\left(\int_0^{+\infty}c(s)e^{-s}ds\right)\sum_{\alpha\in E}\frac{(2\pi)^{n_1}e^{-\sum_{1\le j\le n_1}\varphi_j(z_{j})}}{\prod_{1\le j\le n_1}(\alpha_j+1)c_{j}(z_j)^{2\alpha_{j}+2}}\int_Y|f_{\alpha}|^2e^{-\varphi_Y}.$$
 Combining Theorem \ref{thm:general_concave}, we obtain that  $ G({h}^{-1}(r))$ is linear with respect to $r$,  where $h(t)=\int_t^{+\infty}c(s)e^{-s}ds$.
It follows from Theorem \ref{thm:linear-fibra-single} that   statements  $(2)$ and $(3)$ in Theorem \ref{thm:exten-fibra-single} hold.

Now, we prove the sufficiency. Following from Remark \ref{r:fibra-single}, we obtain that
$$G(0)=\left(\int_0^{+\infty}c(s)e^{-s}ds\right)\sum_{\alpha\in E}\frac{(2\pi)^{n_1}e^{-\sum_{1\le j\le n_1}\varphi_j(z_{j})}}{\prod_{1\le j\le n_1}(\alpha_j+1)c_{j}(z_j)^{2\alpha_{j}+2}}\int_Y|f_{\alpha}|^2e^{-\varphi_Y}.$$

Thus, Remark \ref{r:1.7} holds.

\section{Proofs of Theorem \ref{thm:exten-fibra-finite} and Remark \ref{r:1.8}}

In this section, we prove Theorem \ref{thm:exten-fibra-finite} and Remark \ref{r:1.8}.

\subsection{Proof of Theorem \ref{thm:exten-fibra-finite}}
\

As $c(t)e^{-t}$ is decreasing and $\Psi\le0$, it follows from Proposition \ref{p:exten-fibra} that
there exists a holomorphic $(n,0)$ form $F$ on $M$, which  satisfies that $(F-f,z)\in\left(\mathcal{O}(K_{M})\otimes\mathcal{I}\left(\max_{1\le j\le n_1}\left\{2\sum_{1\le k\le m_j}p_{j,k}\pi_{1,j}^{*}(G_{\Omega_j}(\cdot,z_{j,k}))\right\}\right)\right)_{z}$ for any $z\in Z_0$ and
\begin{equation}\label{eq:1225a}
	\begin{split}
	&\int_{M_1}|F|^2e^{-\varphi}c(-\psi)\\
	\le& \int_{M}|F|^2e^{-\varphi-\pi_1^*(\Psi)}c(-\psi+\pi_1^*(\Psi))\\
	\le&\left(\int_0^{+\infty}c(s)e^{-s}ds\right)\sum_{\beta\in I_1}\sum_{\alpha\in E_{\beta}}\frac{(2\pi)^{n_1}e^{-\left(\Psi+\sum_{1\le j\le n_1}\tilde\pi_j^*(\varphi_j)\right)(z_{\beta})}}{\prod_{1\le j\le n_1}(\alpha_j+1)c_{j}(z_j)^{2\alpha_{j}+2}}\int_Y|f_{\alpha,\beta}|^2e^{-\varphi_Y}.	
	\end{split}	
\end{equation}

If $\Psi\equiv0$,
as $(F-f,z)\in\left(\mathcal{O}(K_{M})\otimes\mathcal{I}\left(\max_{1\le j\le n_1}\left\{2\sum_{1\le k\le m_j}p_{j,k}\pi_{1,j}^{*}(G_{\Omega_j}(\cdot,z_{j,k}))\right\}\right)\right)_{z}$ for any $z\in Z_0$, it follows from Lemma \ref{l:0} and Lemma \ref{l:fibra-decom} that we have $F=\sum_{\alpha\in E_{\beta}}\pi_1^*(w_{\beta}^{\alpha}dw_{1,\beta_1}\wedge\ldots\wedge dw_{n_1,\beta_{n_1}})\wedge\pi_2^*(f_{\alpha,\beta})+\sum_{\alpha\in\tilde E_{\beta}\backslash E_{\beta}}\pi_1^*(w_{\beta}^{\alpha}dw_{1,\beta_1}\wedge\ldots\wedge dw_{n_1,\beta_{n_1}})\wedge\pi_2^*(\tilde f_{\alpha,\beta})$ on $V_\beta\times Y$, where $\tilde f_{\alpha,\beta}$ is a holomorphic $(n_2,0)$ form on $Y$ satisfying $\int_Y|\tilde f_{\alpha,\beta}|^2e^{-\varphi_Y}<+\infty$ for any $\alpha\in\tilde E_{\beta}\backslash E_{\beta}$ and $\beta\in I_1$. Note that $\left(\Psi+\sum_{1\le j\le n_1}\tilde\pi_j^*(\varphi_j)\right)(z_\beta)>-\infty$. It follows from Lemma \ref{l:0}, Lemma \ref{l:phi1+phi2} and Lemma \ref{l:closedness} that $\left(\sum_{\alpha\in\tilde E_{\beta}\backslash E_{\beta}}\pi_1^*(w_{\beta}^{\alpha}dw_{1,\beta_1}\wedge\ldots\wedge dw_{n_1,\beta_{n_1}})\wedge\pi_2^*(\tilde f_{\alpha,\beta}),z\right)\in(\mathcal{O}(K_{M_1})\otimes\mathcal{I}(\varphi+\psi))_{z}$ for any $z\in \{z_\beta\}\times Y$, where $\beta\in I_1$.

In the following, we prove the characterization of the holding of the equality in Theorem \ref{thm:exten-fibra-finite}.

Firstly, we prove the necessity.
Using inequality \eqref{eq:1225a}, we have
$$\int_{M_1}|F|^2e^{-\varphi}c(-\psi)= \int_{M}|F|^2e^{-\varphi-\pi_1^*(\Psi)}c(-\psi+\pi_1^*(\Psi)).$$ Note that $c(t)e^{-t}$ is decreasing. As $F\not\equiv0$, we get that
$$M_1=M=\left(\prod_{1\le j\le n_1}\Omega_j\right)\times Y.$$ As  $\Psi\le0$, it follows from Lemma \ref{l:psi=G} that $\Psi\equiv0$, i.e.,
$$\psi=\max_{1\le j\le n_1}\left\{2\sum_{1\le k\le m_j}p_{j,k}\pi_{1,j}^{*}(G_{\Omega_j}(\cdot,z_{j,k}))\right\}.$$
Denote
\begin{displaymath}
	\begin{split}
		\inf\bigg\{\int_{\{\psi<-t\}}|\tilde f|^2e^{-\varphi}&c(-\psi):\tilde f\in H^0(\{\psi<-t\},\mathcal{O}(K_M))\\
		&\&\,(\tilde f-F,z)\in(\mathcal{O}(K_M)\otimes\mathcal{I}(\varphi+\psi))_z\mbox{ for any $z\in Z_0$}\bigg\}
	\end{split}
\end{displaymath}
by $G(t)$, where $t\ge0$.
 Denote
\begin{displaymath}
	\begin{split}
		\inf\bigg\{\int_{\{\psi<-t\}}|\tilde f|^2e^{-\varphi}&c(-\psi):\tilde f\in H^0(\{\psi<-t\},\mathcal{O}(K_M))\\
		&\&\,(\tilde f-F,z)\in(\mathcal{O}(K_M)\otimes\mathcal{I}(\psi))_z\mbox{ for any $z\in Z_0$}\bigg\}
	\end{split}
\end{displaymath}
by $\tilde G(t)$, where $t\ge0$.
It follows from Lemma \ref{l:G1=G2} that $G(t)=\tilde G(t)$ for any $t\ge0$.
Let $t\ge0$. It follows from Proposition \ref{p:exten-fibra}  ($M\sim\{\psi<-t\}$, $\psi\sim\psi+t$ and $c(\cdot)\sim c(\cdot+t)$, here $\sim$ means the former replaced by the latter) that
\begin{equation*}
	\frac{\tilde G(t)}{\int_t^{+\infty}c(s)e^{-s}ds}\leq	\sum_{\beta\in I_1}\sum_{\alpha\in E_{\beta}}\frac{(2\pi)^{n_1}e^{-\sum_{1\le j\le n_1}\varphi_j(z_{j,\beta_j})}}{\prod_{1\le j\le n_1}(\alpha_j+1)c_{j}(z_j)^{2\alpha_{j}+2}}\int_Y|f_{\alpha,\beta}|^2e^{-\varphi_Y}.\end{equation*}
Note that
$$\tilde G(0)=\left(\int_0^{+\infty}c(s)e^{-s}ds\right)\sum_{\beta\in I_1}\sum_{\alpha\in E_{\beta}}\frac{(2\pi)^{n_1}e^{-\sum_{1\le j\le n_1}\varphi_j(z_{j,\beta_j})}}{\prod_{1\le j\le n_1}(\alpha_j+1)c_{j}(z_j)^{2\alpha_{j}+2}}\int_Y|f_{\alpha,\beta}|^2e^{-\varphi_Y}.$$
 Combining Theorem \ref{thm:general_concave}, we obtain that  $\tilde G({h}^{-1}(r))$ is linear with respect to $r$, which implies that $G({h}^{-1}(r))$ is linear with respect to $r$, where $h(t)=\int_t^{+\infty}c(s)e^{-s}ds$.  As $f_{\alpha,\beta^*}\equiv0$ for any $\alpha\not=\alpha_{\beta^*}$ satisfying $\sum_{1\le j\le n_1}\frac{\alpha_j+1}{p_{j,1}}=1$, where $\beta^*=(1,\ldots,1)\in I_1$,
it follows from Theorem \ref{thm:linear-fibra-finite} that   statements  $(2)$, $(3)$, $(4)$ and $(5)$ in Theorem \ref{thm:exten-fibra-finite} hold.

Now, we prove the sufficiency.  Following from Remark \ref{r:fibra-finite} and $G(0)=\tilde G(0)$, we obtain that
$$\tilde G(0)=\left(\int_0^{+\infty}c(s)e^{-s}ds\right)\sum_{\beta\in I_1}\sum_{\alpha\in E_{\beta}}\frac{(2\pi)^{n_1}e^{-\sum_{1\le j\le n_1}\varphi_j(z_{j,\beta_j})}}{\prod_{1\le j\le n_1}(\alpha_j+1)c_{j}(z_j)^{2\alpha_{j}+2}}\int_Y|f_{\alpha,\beta}|^2e^{-\varphi_Y}.$$

Thus, Theorem \ref{thm:exten-fibra-finite} holds.

\subsection{Proof of Remark \ref{r:1.8}}\label{sec:proof-1.8}
\

Note that $\left(\Psi+\sum_{1\le j\le n_1}\tilde\pi_j^*(\varphi_j)\right)(z_{\beta})>-\infty$ for any $\beta\in I_1$.  As $(f_{\alpha,\beta},y)\in(\mathcal{O}(K_Y)\otimes\mathcal{I}(\varphi_Y))_y$ for any $y\in Y$, $\alpha\in\tilde E_{\beta}\backslash E_{\beta}$ and $\beta\in I_1$, following from Lemma \ref{l:phi1+phi2}, Lemma \ref{l:0} and Lemma \ref{l:closedness}, we get that $\big(\sum_{\alpha\in \tilde E_{\beta}\backslash E_{\beta}}\pi_1^*(w_{\beta}^{\alpha}dw_{1,\beta_1}\wedge\ldots dw_{n_1,\beta_{n_1}})\wedge\pi_2^*(f_{\alpha,\beta}),z\big)\in(\mathcal{O}(K_{M_1})\otimes\mathcal{I}(\varphi+\psi))_z$ for any $z\in \{z_{\beta}\}\times Y$, where $\beta\in I_1$.

As $c(t)e^{-t}$ is decreasing and $\Psi\le0$, it follows from Proposition \ref{p:exten-fibra} that
there exists a holomorphic $(n,0)$ form $F$ on $M$, which satisfies that $(F-f,z)\in\left(\mathcal{O}(K_{M})\otimes\mathcal{I}\left(\max_{1\le j\le n_1}\left\{2\sum_{1\le k\le m_j}p_{j,k}\pi_{1,j}^{*}(G_{\Omega_j}(\cdot,z_{j,k}))\right\}\right)\right)_{z}$ for any $z\in Z_0$ and
\begin{equation}\label{eq:1225b}
	\begin{split}
	&\int_{M_1}|F|^2e^{-\varphi}c(-\psi)\\
	\le& \int_{M}|F|^2e^{-\varphi-\pi_1^*(\Psi)}c(-\psi+\pi_1^*(\Psi))\\
	\le&\left(\int_0^{+\infty}c(s)e^{-s}ds\right)\sum_{\beta\in I_1}\sum_{\alpha\in E_{\beta}}\frac{(2\pi)^{n_1}e^{-\left(\Psi+\sum_{1\le j\le n_1}\tilde\pi_j^*(\varphi_j)\right)(z_{\beta})}}{\prod_{1\le j\le n_1}(\alpha_j+1)c_{j}(z_j)^{2\alpha_{j}+2}}\int_Y|f_{\alpha,\beta}|^2e^{-\varphi_Y}.	
	\end{split}	
\end{equation}

If $\Psi\equiv0$, as $(F-f,z)\in\left(\mathcal{O}(K_{M})\otimes\mathcal{I}\left(\psi\right)\right)_{z}$ for any $z\in Z_0$, it follows from Lemma \ref{l:0} and Lemma \ref{l:fibra-decom} that we have $F=\sum_{\alpha\in E_{\beta}}\pi_1^*(w_{\beta}^{\alpha}dw_{1,\beta_1}\wedge\ldots\wedge dw_{n_1,\beta_{n_1}})\wedge\pi_2^*(f_{\alpha,\beta})+\sum_{\alpha\in\tilde E_{\beta}\backslash E_{\beta}}\pi_1^*(w_{\beta}^{\alpha}dw_{1,\beta_1}\wedge\ldots\wedge dw_{n_1,\beta_{n_1}})\wedge\pi_2^*(\tilde f_{\alpha,\beta})$ on $V_{\beta}\times Y$, where $\beta\in I_1$ and $\tilde f_{\alpha,\beta}$ is a holomorphic $(n_2,0)$ form on $ Y$ satisfying $\int_Y|\tilde f_{\alpha,\beta}|^2e^{-\varphi_Y}<+\infty$ for any $\alpha\in\tilde E_{\beta}\backslash E_{\beta}$. Note that $(\Psi+\sum_{1\le j\le n_1}\tilde\pi_j^*(\varphi_j))(z_\beta)>-\infty$ for any $\beta\in I_1$. Following from Lemma \ref{l:0}, Lemma \ref{l:phi1+phi2} and Lemma \ref{l:closedness}, we obtain that $\left(\sum_{\alpha\in\tilde E_{\beta}\backslash E_{\beta}}\pi_1^*(w_{\beta}^{\alpha}dw_{1,\beta_1}\wedge\ldots\wedge dw_{n_1,\beta_{n_1}})\wedge\pi_2^*(\tilde f_{\alpha,\beta}),z\right)\in(\mathcal{O}(K_{M})\otimes\mathcal{I}(\varphi+\psi))_{z}$ for any $z\in \{z_{\beta}\}\times Y$.
Thus, we have $(F-f,z)\in(\mathcal{O}(K_{M})\otimes\mathcal{I}(\varphi+\psi))_z$ for any $z\in Z_0$.

In the following, we prove the characterization of the holding of the equality (replacing  the ideal sheaf $\mathcal{I}\left(\max_{1\le j\le n_1}\left\{2\sum_{1\le k\le m_j}p_{j,k}\pi_{1,j}^{*}(G_{\Omega_j}(\cdot,z_{j,k}))\right\}\right)$  by $\mathcal{I}(\varphi+\psi)$) in Theorem \ref{thm:exten-fibra-finite}.

Firstly, we prove the necessity.
Using inequality \eqref{eq:1225b}, we have
$$\int_{M_1}|F|^2e^{-\varphi}c(-\psi)= \int_{M}|F|^2e^{-\varphi-\pi_1^*(\Psi)}c(-\psi+\pi_1^*(\Psi)).$$ Note that $c(t)e^{-t}$ is decreasing. As $F\not\equiv0$, we get that
$$M_1=M=\left(\prod_{1\le j\le n_1}\Omega_j\right)\times Y.$$ As  $\Psi\le0$, it follows from Lemma \ref{l:psi=G} that $\Psi\equiv0$, i.e.,
$$\psi=\max_{1\le j\le n_1}\left\{2\sum_{1\le k\le m_j}p_{j,k}\pi_{1,j}^{*}(G_{\Omega_j}(\cdot,z_{j,k}))\right\}.$$
Denote
\begin{displaymath}
	\begin{split}
		\inf\bigg\{\int_{\{\psi<-t\}}|\tilde f|^2e^{-\varphi}&c(-\psi):\tilde f\in H^0(\{\psi<-t\},\mathcal{O}(K_M))\\
		&\&\,(\tilde f-F,z)\in(\mathcal{O}(K_M)\otimes\mathcal{I}(\varphi+\psi))_z\mbox{ for any $z\in Z_0$}\bigg\}
	\end{split}
\end{displaymath}
by $G(t)$, where $t\ge0$.
 Denote
\begin{displaymath}
	\begin{split}
		\inf\bigg\{\int_{\{\psi<-t\}}|\tilde f|^2e^{-\varphi}&c(-\psi):\tilde f\in H^0(\{\psi<-t\},\mathcal{O}(K_M))\\
		&\&\,(\tilde f-F,z)\in(\mathcal{O}(K_M)\otimes\mathcal{I}(\psi))_z\mbox{ for any $z\in Z_0$}\bigg\}
	\end{split}
\end{displaymath}
by $\tilde G(t)$, where $t\ge0$.
It follows from Lemma \ref{l:G1=G2} that $G(t)=\tilde G(t)$ for any $t\ge0$.
Let $t\ge0$. It follows from Proposition \ref{p:exten-fibra}  ($M\sim\{\psi<-t\}$, $\psi\sim\psi+t$ and $c(\cdot)\sim c(\cdot+t)$, here $\sim$ means the former replaced by the latter) that
\begin{equation*}
	\frac{\tilde G(t)}{\int_t^{+\infty}c(s)e^{-s}ds}\leq	\sum_{\beta\in I_1}\sum_{\alpha\in E_{\beta}}\frac{(2\pi)^{n_1}e^{-\sum_{1\le j\le n_1}\varphi_j(z_{j,\beta_j})}}{\prod_{1\le j\le n_1}(\alpha_j+1)c_{j}(z_j)^{2\alpha_{j}+2}}\int_Y|f_{\alpha,\beta}|^2e^{-\varphi_Y}.\end{equation*}
Note that
$$ G(0)=\left(\int_0^{+\infty}c(s)e^{-s}ds\right)\sum_{\beta\in I_1}\sum_{\alpha\in E_{\beta}}\frac{(2\pi)^{n_1}e^{-\sum_{1\le j\le n_1}\varphi_j(z_{j,\beta_j})}}{\prod_{1\le j\le n_1}(\alpha_j+1)c_{j}(z_j)^{2\alpha_{j}+2}}\int_Y|f_{\alpha,\beta}|^2e^{-\varphi_Y}.$$
 Combining Theorem \ref{thm:general_concave}, we obtain that  $ G({h}^{-1}(r))$ is linear with respect to $r$,  where $h(t)=\int_t^{+\infty}c(s)e^{-s}ds$.
It follows from Theorem \ref{thm:linear-fibra-single} that   statements  $(2)$, $(3)$, $(4)$ and $(5)$ in Theorem \ref{thm:exten-fibra-finite} hold.

Now, we prove the sufficiency. Following from Remark \ref{r:fibra-finite}, we obtain that
$$G(0)=\left(\int_0^{+\infty}c(s)e^{-s}ds\right)\sum_{\beta\in I_1}\sum_{\alpha\in E_{\beta}}\frac{(2\pi)^{n_1}e^{-\sum_{1\le j\le n_1}\varphi_j(z_{j,\beta_j})}}{\prod_{1\le j\le n_1}(\alpha_j+1)c_{j}(z_j)^{2\alpha_{j}+2}}\int_Y|f_{\alpha,\beta}|^2e^{-\varphi_Y}.$$

Thus, Remark \ref{r:1.8} holds.

\section{Proofs of Theorem \ref{thm:exten-fibra-infinite} and Remark \ref{r:1.9}}

In this section, we prove Theorem \ref{thm:exten-fibra-infinite} and Remark \ref{r:1.9}.

\subsection{Proof of Theorem \ref{thm:exten-fibra-infinite}}
\

As $c(t)e^{-t}$ is decreasing and $\Psi\le0$, it follows from Proposition \ref{p:exten-fibra} that
there exists a holomorphic $(n,0)$ form $F$ on $M$, which satisfies that $(F-f,z)\in\left(\mathcal{O}(K_{M})\otimes\mathcal{I}\left(\max_{1\le j\le n_1}\left\{2\sum_{1\le k<\tilde m_j}p_{j,k}\pi_{1,j}^{*}(G_{\Omega_j}(\cdot,z_{j,k}))\right\}\right)\right)_{z}$ for any $z\in Z_0$ and
\begin{equation}\label{eq:1225c}
	\begin{split}
	&\int_{M_1}|F|^2e^{-\varphi}c(-\psi)\\
	\le& \int_{M}|F|^2e^{-\varphi-\pi_1^*(\Psi)}c(-\psi+\pi_1^*(\Psi))\\
	\le&\left(\int_0^{+\infty}c(s)e^{-s}ds\right)\sum_{\beta\in\tilde I_1}\sum_{\alpha\in E_{\beta}}\frac{(2\pi)^{n_1}e^{-\left(\Psi+\sum_{1\le j\le n_1}\tilde\pi_j^*(\varphi_j)\right)(z_{\beta})}}{\prod_{1\le j\le n_1}(\alpha_j+1)c_{j}(z_j)^{2\alpha_{j}+2}}\int_Y|f_{\alpha,\beta}|^2e^{-\varphi_Y}.	
	\end{split}	
\end{equation}

If $\Psi\equiv0$, as $(F-f,z)\in(\mathcal{O}(K_{M_1})\otimes\mathcal{I}(\psi))_{z}$ for any $z\in Z_0$, it follows from Lemma \ref{l:0} and Lemma \ref{l:fibra-decom} that we have $F=\sum_{\alpha\in E_{\beta}}\pi_1^*(w_{\beta}^{\alpha}dw_{1,\beta_1}\wedge\ldots\wedge dw_{n_1,\beta_{n_1}})\wedge\pi_2^*(f_{\alpha,\beta})+\sum_{\alpha\in\tilde E_{\beta}\backslash E_{\beta}}\pi_1^*(w_{\beta}^{\alpha}dw_{1,\beta_1}\wedge\ldots\wedge dw_{n_1,\beta_{n_1}})\wedge\pi_2^*(\tilde f_{\alpha,\beta})$ on $V_\beta\times Y$, where $\tilde f_{\alpha,\beta}$ is a holomorphic $(n_2,0)$ form on $Y$ satisfying $\int_Y|\tilde f_{\alpha,\beta}|^2e^{-\varphi_Y}<+\infty$ for any $\alpha\in\tilde E_{\beta}\backslash E_{\beta}$ and $\beta\in\tilde I_1$. Note that $\left(\Psi+\sum_{1\le j\le n_1}\tilde\pi_j^*(\varphi_j)\right)(z_\beta)>-\infty$. For any $\beta\in\tilde I_1$, it follows from Lemma \ref{l:0}, Lemma \ref{l:phi1+phi2} and Lemma \ref{l:closedness} that $\left(\sum_{\alpha\in\tilde E_{\beta}\backslash E_{\beta}}\pi_1^*(w_{\beta}^{\alpha}dw_{1,\beta_1}\wedge\ldots\wedge dw_{n_1,\beta_{n_1}})\wedge\pi_2^*(\tilde f_{\alpha,\beta}),z\right)\in(\mathcal{O}(K_{M})\otimes\mathcal{I}(\varphi+\psi))_{z}$ for any $z\in \{z_\beta\}\times Y$.

Denote that
$\tilde \psi:=\max_{1\le j\le n_1}\left\{2\sum_{1\le k<\tilde m_j}p_{j,k}\pi_{1,j}^{*}(G_{\Omega_j}(\cdot,z_{j,k}))\right\}.$
Now, we assume  $\left(\int_0^{+\infty}c(s)e^{-s}ds\right)\sum_{\beta\in\tilde I_1}\sum_{\alpha\in E_{\beta}}\frac{(2\pi)^{n_1}e^{-\left(\Psi+\sum_{1\le j\le n_1}\tilde\pi_j^*(\varphi_j)\right)(z_{\beta})}}{\prod_{1\le j\le n_1}(\alpha_j+1)c_{j}(z_j)^{2\alpha_{j}+2}}\int_Y|f_{\alpha,\beta}|^2e^{-\varphi_Y}= \inf\big\{\int_{M_1}|\tilde{F}|^2e^{-\varphi}c(-\psi):\tilde{F}$ is a holomorphic $(n,0)$ form on $M_1$ such that $(\tilde{F}-f,z)\in(\mathcal{O}(K_{M_1})\otimes\mathcal{I}(\tilde\psi))_{z}$ for any $z\in Z_0\big\}$ to get a contradiction.	

Using inequality \eqref{eq:1225c}, we have
$$\int_{M_1}|F|^2e^{-\varphi}c(-\psi)= \int_{M}|F|^2e^{-\varphi-\pi_1^*(\Psi)}c(-\psi+\pi_1^*(\Psi)).$$ Note that $c(t)e^{-t}$ is decreasing. As $F\not\equiv0$, we get that
$$M_1=M=\left(\prod_{1\le j\le n_1}\Omega_j\right)\times Y.$$ As  $\Psi\le0$, it follows from Lemma \ref{l:psi=G} that $\Psi\equiv0$, i.e.,
$$\psi=\max_{1\le j\le n_1}\left\{2\sum_{1\le k<\tilde m_j}p_{j,k}\pi_{1,j}^{*}(G_{\Omega_j}(\cdot,z_{j,k}))\right\}.$$
Denote
\begin{displaymath}
	\begin{split}
		\inf\bigg\{\int_{\{\psi<-t\}}|\tilde f|^2e^{-\varphi}&c(-\psi):\tilde f\in H^0(\{\psi<-t\},\mathcal{O}(K_M))\\
		&\&\,(\tilde f-F,z)\in(\mathcal{O}(K_M)\otimes\mathcal{I}(\varphi+\psi))_z\mbox{ for any $z\in Z_0$}\bigg\}
	\end{split}
\end{displaymath}
by $G(t)$, where $t\ge0$.
 Denote
\begin{displaymath}
	\begin{split}
		\inf\bigg\{\int_{\{\psi<-t\}}|\tilde f|^2e^{-\varphi}&c(-\psi):\tilde f\in H^0(\{\psi<-t\},\mathcal{O}(K_M))\\
		&\&\,(\tilde f-F,z)\in(\mathcal{O}(K_M)\otimes\mathcal{I}(\psi))_z\mbox{ for any $z\in Z_0$}\bigg\}
	\end{split}
\end{displaymath}
by $\tilde G(t)$, where $t\ge0$.
It follows from Lemma \ref{l:G1=G2} that $G(t)=\tilde G(t)$ for any $t\ge0$.
Let $t\ge0$. It follows from Proposition \ref{p:exten-fibra}  ($M\sim\{\psi<-t\}$, $\psi\sim\psi+t$ and $c(\cdot)\sim c(\cdot+t)$, here $\sim$ means the former replaced by the latter) that
\begin{equation*}
	\frac{\tilde G(t)}{\int_t^{+\infty}c(s)e^{-s}ds}\leq	\sum_{\beta\in\tilde I_1}\sum_{\alpha\in E_{\beta}}\frac{(2\pi)^{n_1}e^{-\sum_{1\le j\le n_1}\varphi_j(z_{j,\beta_j})}}{\prod_{1\le j\le n_1}(\alpha_j+1)c_{j}(z_j)^{2\alpha_{j}+2}}\int_Y|f_{\alpha,\beta}|^2e^{-\varphi_Y}.\end{equation*}
Note that
$$\tilde G(0)=\left(\int_0^{+\infty}c(s)e^{-s}ds\right)\sum_{\beta\in\tilde I_1}\sum_{\alpha\in E_{\beta}}\frac{(2\pi)^{n_1}e^{-\sum_{1\le j\le n_1}\varphi_j(z_{j,\beta_j})}}{\prod_{1\le j\le n_1}(\alpha_j+1)c_{j}(z_j)^{2\alpha_{j}+2}}\int_Y|f_{\alpha,\beta}|^2e^{-\varphi_Y}.$$
 Combining Theorem \ref{thm:general_concave}, we obtain that  $\tilde G({h}^{-1}(r))$ is linear with respect to $r$, which implies that $G({h}^{-1}(r))$ is linear with respect to $r$, where $h(t)=\int_t^{+\infty}c(s)e^{-s}ds$.  As $f_{\alpha,\beta^*}\equiv0$ for any $\alpha\not=\alpha_{\beta^*}$ satisfying $\sum_{1\le j\le n_1}\frac{\alpha_j+1}{p_{j,1}}=1$, where $\beta^*=(1,\ldots,1)\in\tilde I_1$,
the linearity of $G({h}^{-1}(r))$ contradicts to Theorem \ref{thm:linear-fibra-infinite}.
 Thus, we obtain that there exists a holomorphic $(n,0)$ form $\tilde F$ on $M_1$, which satisfies that $(\tilde{F}-f,z)\in\left(\mathcal{O}(K_{M})\otimes\mathcal{I}\left(\max_{1\le j\le n_1}\left\{2\sum_{1\le k<\tilde m_j}p_{j,k}\pi_{1,j}^{*}(G_{\Omega_j}(\cdot,z_{j,k}))\right\}\right)\right)_{z}$ for any $z\in Z_0$ and
 \begin{displaymath}
 	\begin{split}
 		&\int_{M}|\tilde F|^2e^{-\varphi}c(-\psi)\\
 		<&\left(\int_0^{+\infty}c(s)e^{-s}ds\right)\sum_{\beta\in\tilde I_1}\sum_{\alpha\in E_{\beta}}\frac{(2\pi)^{n_1}e^{-\left(\Psi+\sum_{1\le j\le n_1}\tilde\pi_j^*(\varphi_j)\right)(z_{\beta})}}{\prod_{1\le j\le n_1}(\alpha_j+1)c_{j}(z_j)^{2\alpha_{j}+2}}\int_Y|f_{\alpha,\beta}|^2e^{-\varphi_Y}.
 	\end{split}
 \end{displaymath}

\subsection{Proof of Remark \ref{r:1.9}}\label{sec:proof-1.9}
\

Note that $\left(\Psi+\sum_{1\le j\le n_1}\tilde\pi_j^*(\varphi_j)\right)(z_{\beta})>-\infty$ for any $\beta\in \tilde I_1$.  As $(f_{\alpha,\beta},y)\in(\mathcal{O}(K_Y)\otimes\mathcal{I}(\varphi_Y))_y$ for any $y\in Y$, $\alpha\in\tilde E_{\beta}\backslash E_{\beta}$ and $\beta\in\tilde I_1$, following from Lemma \ref{l:phi1+phi2}, Lemma \ref{l:0} and Lemma \ref{l:closedness}, we get that $\big(\sum_{\alpha\in \tilde E_{\beta}\backslash E_{\beta}}\pi_1^*(w_{\beta}^{\alpha}dw_{1,\beta_1}\wedge\ldots dw_{n_1,\beta_{n_1}})\wedge\pi_2^*(f_{\alpha,\beta}),z\big)\in(\mathcal{O}(K_{M_1})\otimes\mathcal{I}(\varphi+\psi))_z$ for any $z\in \{z_{\beta}\}\times Y$, where $\beta\in\tilde I_1$.

As $c(t)e^{-t}$ is decreasing and $\Psi\le0$, it follows from Proposition \ref{p:exten-fibra} that
there exists a holomorphic $(n,0)$ form $F$ on $M$, which  satisfies that $(F-f,z)\in\left(\mathcal{O}(K_{M})\otimes\mathcal{I}\left(\max_{1\le j\le n_1}\left\{2\sum_{1\le k<\tilde m_j}p_{j,k}\pi_{1,j}^{*}(G_{\Omega_j}(\cdot,z_{j,k}))\right\}\right)\right)_{z}$ for any $z\in Z_0$ and
\begin{equation}\label{eq:1225d}
	\begin{split}
	&\int_{M_1}|F|^2e^{-\varphi}c(-\psi)\\
	\le& \int_{M}|F|^2e^{-\varphi-\pi_1^*(\Psi)}c(-\psi+\pi_1^*(\Psi))\\
	\le&\left(\int_0^{+\infty}c(s)e^{-s}ds\right)\sum_{\beta\in\tilde I_1}\sum_{\alpha\in E_{\beta}}\frac{(2\pi)^{n_1}e^{-\left(\Psi+\sum_{1\le j\le n_1}\tilde\pi_j^*(\varphi_j)\right)(z_{\beta})}}{\prod_{1\le j\le n_1}(\alpha_j+1)c_{j}(z_j)^{2\alpha_{j}+2}}\int_Y|f_{\alpha,\beta}|^2e^{-\varphi_Y}.	
	\end{split}	
\end{equation}

If $\Psi\equiv0$, as $(F-f,z)\in\left(\mathcal{O}(K_{M})\otimes\mathcal{I}\left(\psi\right)\right)_{z}$ for any $z\in Z_0$, it follows from Lemma \ref{l:0} and Lemma \ref{l:fibra-decom} that we have $F=\sum_{\alpha\in E_{\beta}}\pi_1^*(w_{\beta}^{\alpha}dw_{1,\beta_1}\wedge\ldots\wedge dw_{n_1,\beta_{n_1}})\wedge\pi_2^*(f_{\alpha,\beta})+\sum_{\alpha\in\tilde E_{\beta}\backslash E_{\beta}}\pi_1^*(w_{\beta}^{\alpha}dw_{1,\beta_1}\wedge\ldots\wedge dw_{n_1,\beta_{n_1}})\wedge\pi_2^*(\tilde f_{\alpha,\beta})$ on $V_\beta\times Y$, where $\tilde f_{\alpha,\beta}$ is a holomorphic $(n_2,0)$ form on $Y$ satisfying $\int_Y|\tilde f_{\alpha,\beta}|^2e^{-\varphi_Y}<+\infty$ for any $\alpha\in\tilde E_{\beta}\backslash E_{\beta}$ and $\beta\in\tilde I_1$. Note that $\left(\Psi+\sum_{1\le j\le n_1}\tilde\pi_j^*(\varphi_j)\right)(z_\beta)>-\infty$. Following from Lemma \ref{l:0}, Lemma \ref{l:phi1+phi2} and Lemma \ref{l:closedness}, we obtain that $\left(\sum_{\alpha\in\tilde E_{\beta}\backslash E_{\beta}}\pi_1^*(w_{\beta}^{\alpha}dw_{1,\beta_1}\wedge\ldots\wedge dw_{n_1,\beta_{n_1}})\wedge\pi_2^*(\tilde f_{\alpha,\beta}),z\right)\in(\mathcal{O}(K_{M})\otimes\mathcal{I}(\varphi+\psi))_{z}$ for any $z\in \{z_\beta\}\times Y$, where $\beta\in\tilde I_1$. Hence, we have $(F-f,z)\in(\mathcal{O}(K_{M_1})\otimes\mathcal{I}(\varphi+\psi))_z$ for any $z\in Z_0$.

In the following, we assume that $\inf\big\{\int_{M_1}|\tilde{F}|^2e^{-\varphi}c(-\psi):\tilde{F}$ is a holomorphic $(n,0)$ form on $M_1$ such that $(\tilde{F}-f,z)\in(\mathcal{O}(K_{M_1})\otimes\mathcal{I}(\varphi+\psi))_{z}$ for any $z\in Z_0\big\}=\left(\int_0^{+\infty}c(s)e^{-s}ds\right)\sum_{\beta\in\tilde I_1}\sum_{\alpha\in E_{\beta}}\frac{(2\pi)^{n_1}e^{-\left(\Psi+\sum_{1\le j\le n_1}\tilde\pi_j^*(\varphi_j)\right)(z_{\beta})}}{\prod_{1\le j\le n_1}(\alpha_j+1)c_{j}(z_j)^{2\alpha_{j}+2}}\int_Y|f_{\alpha,\beta}|^2e^{-\varphi_Y}$ to get a contradiction.	

Using inequality \eqref{eq:1225d}, we have
$$\int_{M_1}|F|^2e^{-\varphi}c(-\psi)= \int_{M}|F|^2e^{-\varphi-\pi_1^*(\Psi)}c(-\psi+\pi_1^*(\Psi)).$$ Note that $c(t)e^{-t}$ is decreasing. As $F\not\equiv0$, we get that
$$M_1=M=\left(\prod_{1\le j\le n_1}\Omega_j\right)\times Y.$$ As  $\Psi\le0$, it follows from Lemma \ref{l:psi=G} that $\Psi\equiv0$, i.e.,
$$\psi=\max_{1\le j\le n_1}\left\{2\sum_{1\le k<\tilde m_j}p_{j,k}\pi_{1,j}^{*}(G_{\Omega_j}(\cdot,z_{j,k}))\right\}.$$
Denote
\begin{displaymath}
	\begin{split}
		\inf\bigg\{\int_{\{\psi<-t\}}|\tilde f|^2e^{-\varphi}&c(-\psi):\tilde f\in H^0(\{\psi<-t\},\mathcal{O}(K_M))\\
		&\&\,(\tilde f-F,z)\in(\mathcal{O}(K_M)\otimes\mathcal{I}(\varphi+\psi))_z\mbox{ for any $z\in Z_0$}\bigg\}
	\end{split}
\end{displaymath}
by $G(t)$, where $t\ge0$.
 Denote
\begin{displaymath}
	\begin{split}
		\inf\bigg\{\int_{\{\psi<-t\}}|\tilde f|^2e^{-\varphi}&c(-\psi):\tilde f\in H^0(\{\psi<-t\},\mathcal{O}(K_M))\\
		&\&\,(\tilde f-F,z)\in(\mathcal{O}(K_M)\otimes\mathcal{I}(\psi))_z\mbox{ for any $z\in Z_0$}\bigg\}
	\end{split}
\end{displaymath}
by $\tilde G(t)$, where $t\ge0$.
It follows from Lemma \ref{l:G1=G2} that $G(t)=\tilde G(t)$ for any $t\ge0$.
Let $t\ge0$. It follows from Proposition \ref{p:exten-fibra}  ($M\sim\{\psi<-t\}$, $\psi\sim\psi+t$ and $c(\cdot)\sim c(\cdot+t)$, here $\sim$ means the former replaced by the latter) that
\begin{equation*}
	\frac{\tilde G(t)}{\int_t^{+\infty}c(s)e^{-s}ds}\leq	\sum_{\beta\in\tilde I_1}\sum_{\alpha\in E_{\beta}}\frac{(2\pi)^{n_1}e^{-\sum_{1\le j\le n_1}\varphi_j(z_{j,\beta_j})}}{\prod_{1\le j\le n_1}(\alpha_j+1)c_{j}(z_j)^{2\alpha_{j}+2}}\int_Y|f_{\alpha,\beta}|^2e^{-\varphi_Y}.\end{equation*}
Note that
$$G(0)=\left(\int_0^{+\infty}c(s)e^{-s}ds\right)\sum_{\beta\in\tilde I_1}\sum_{\alpha\in E_{\beta}}\frac{(2\pi)^{n_1}e^{-\sum_{1\le j\le n_1}\varphi_j(z_{j,\beta_j})}}{\prod_{1\le j\le n_1}(\alpha_j+1)c_{j}(z_j)^{2\alpha_{j}+2}}\int_Y|f_{\alpha,\beta}|^2e^{-\varphi_Y}.$$
 Combining Theorem \ref{thm:general_concave}, we obtain that  $ G({h}^{-1}(r))$ is linear with respect to $r$, which implies that $G({h}^{-1}(r))$ is linear with respect to $r$, where $h(t)=\int_t^{+\infty}c(s)e^{-s}ds$.  As $f_{\alpha,\beta^*}\equiv0$ for any $\alpha\not=\alpha_{\beta^*}$ satisfying $\sum_{1\le j\le n_1}\frac{\alpha_j+1}{p_{j,1}}=1$, where $\beta^*=(1,\ldots,1)\in\tilde I_1$,
the linearity of $G({h}^{-1}(r))$ contradicts to Theorem \ref{thm:linear-fibra-infinite}.
 Thus, we obtain that there exists a holomorphic $(n,0)$ form $\tilde F$ on $\Omega$ such that $(\tilde{F}-f,z)\in(\varphi+\psi)_{z}$ for any $z\in Z_0$ and
 \begin{displaymath}
 	\begin{split}
 		&\int_{M}|\tilde F|^2e^{-\varphi}c(-\psi)\\
 		<&\left(\int_0^{+\infty}c(s)e^{-s}ds\right)\sum_{\beta\in\tilde I_1}\sum_{\alpha\in E_{\beta}}\frac{(2\pi)^{n_1}e^{-\left(\Psi+\sum_{1\le j\le n_1}\tilde\pi_j^*(\varphi_j)\right)(z_{\beta})}}{\prod_{1\le j\le n_1}(\alpha_j+1)c_{j}(z_j)^{2\alpha_{j}+2}}\int_Y|f_{\alpha,\beta}|^2e^{-\varphi_Y}.
 	\end{split}
 \end{displaymath}

\section{Proofs of Theorem \ref{thm:suita}, Remark \ref{r:suita},  Theorem \ref{thm:extend} and Remark \ref{r:extend}}
In this section, we prove Theorem \ref{thm:suita}, Remark \ref{r:suita},  Theorem \ref{thm:extend} and Remark \ref{r:extend}.

\subsection{Proofs of Theorem \ref{thm:suita} and Remark \ref{r:suita}}
\

Let $f_1=dw_1\wedge\ldots\wedge dw_{n_1}\wedge d\tilde w_1\wedge\ldots\wedge d\tilde w_{n_2}$ on $V_0\times U_0$, and let $f_2=d\tilde w_1\wedge\ldots\wedge d\tilde w_{n_2}$ on $U_0$. Let $\psi=\max_{1\le j\le n_1}\left\{\pi_{1,j}^*(2n_1G_{\Omega_j}(\cdot,z_j))\right\}$. Following from Lemma \ref{l:0}, we get that $(H_1-H_2,(z_0,y))\in\mathcal{I}(\psi)_{(z_0,y)}$ for any $y\in Y$ if and only if $(H_1-H_2)|_{\{z_0\}\times Y}=0$, where $H_1$ and $H_2$ are holomorphic $(n,0)$ form on a neighborhood of $\{z_0\}\times Y$. Let $f$ be a holomorphic $(n_2,0)$ form on $Y$ satisfying $\int_Y|f|^2<+\infty$. It follows from Proposition \ref{p:exten-fibra} that there exists a holomorphic $(n,0)$ form $F$ on $M$ such that
$F|_{\{z_0\}\times Y}=\pi_1^*(dw_1\wedge\ldots\wedge dw_{n_1})\wedge\pi_2^*(f)$
 and
$$\int_{M}|F|^2\le \frac{(2\pi)^{n_1}}{\prod_{1\le j\le n_1}c_j(z_j)^2}\int_Y|f|^2.$$
Note that
$$B_Y(y_0)=\frac{2^{n_2}}{\inf\left\{\int_Y|f|^2:f\in H^0(Y,\mathcal{O}(K_Y))\,\&\,f(y_0)=f_2(y_0)\right\}}$$
and
$$B_M((z_0,y_0))=\frac{2^{n}}{\inf\left\{\int_M|F|^2:F\in H^0(M,\mathcal{O}(K_M))\,\&\,F((z_0,y_0))=f_1((z_0,y_0))\right\}}.$$
Thus, we have $\prod_{1\le j\le n_1}c_j(z_j)^2 B_Y(y_0)\le \pi^{n_1}B_M((z_0,y_0)).$

In the following, we prove the characterization of the holding of the  equality $\prod_{1\le j\le n_1}c_j(z_j)^2 B_Y(y_0)= \pi^{n_1}B_M((z_0,y_0))$.

There exists a holomorphic $(n_2,0)$ form $f_0$ on $Y$ such that $f_0(y_0)=f_2(y_0)$ and
$$B_Y(y_0)=\frac{2^{n_2}}{\int_Y|f_0|^2}>0.$$
It follows from Proposition \ref{p:exten-fibra} that there exists a holomorphic $(n,0)$ form $F_0$ on $M$ such that $F_0=\pi_1^*(dw_1\wedge\ldots\wedge dw_{n_1})\wedge\pi_2^*(f_0)$ and
\begin{equation}
	\label{eq:1226d}\int_{M}|F_0|^2\le \frac{(2\pi)^{n_1}}{\prod_{1\le j\le n_1}c_j(z_j)^2}\int_Y|f_0|^2.
\end{equation}

Firstly, we prove the necessity. Note that $B_M((z_0,y_0))\ge\frac{2^n}{\int_M|\tilde F|^2}$ for any holomorphic $(n,0)$ form  $\tilde F$  on $M$ satisfying that  $\tilde F=\pi_1^*(dw_1\wedge\ldots\wedge dw_{n_1})\wedge\pi_2^*(f_0)$ on $\{z_0\}\times Y$. Combining $\prod_{1\le j\le n_1}c_j(z_j)^2 B_Y(y_0)= \pi^{n_1}B_M((z_0,y_0))$, $B_Y(y_0)=\frac{2^{n_2}}{\int_Y|f_0|^2}$ and inequality \eqref{eq:1226d}, we obtain that  $\frac{(2\pi)^{n_1}}{\prod_{1\le j\le n_1}c_j(z_j)^2}\int_Y|f_0|^2=\inf\big\{\int_{M}|\tilde F|^2:\tilde F\in H^0(M,\mathcal{O}(K_M))\,\&\, \tilde F|_{\{z_0\}\times Y}=\pi_1^*(dw_1\wedge\ldots\wedge dw_{n_1})\wedge\pi_2^*(f_0)\big\}.$
It follows from Theorem \ref{thm:exten-fibra-single} that $\chi_{j,z_j}=1$ for any $1\le j\le n_1$. $\chi_{j,z_j}=1$ implies that there exists a holomorphic function $f_j$ on $\Omega_j$ such that $|f_j|=e^{G_{\Omega_j}(\cdot,z_j)}$, thus $\Omega_j$ is conformally equivalent to the unit disc less a (possible) closed set of inner capacity zero (see \cite{suita72}, see also \cite{Yamada} and \cite{GZ15}).

Now, we prove the sufficiency. As $\Omega_j$ is conformally equivalent to the unit disc less a (possible) closed set of inner capacity zero, we have $\chi_{j,z_j}=1$. We prove $\prod_{1\le j\le n_1}c_j(z_j)^2 B_Y(y_0)= \pi^{n_1}B_M((z_0,y_0))$ by contradiction: if not, there exists a holomorphic $(n,0)$ form $\tilde F_0$ on $M$ such that $\tilde F_0((z_0,y_0))=f_1((z_0,y_0))$ and
\begin{equation}
	\label{eq:1226e}\int_M|\tilde F_0|^2<\frac{(2\pi)^{n_1}}{\prod_{1\le j\le n_1}c_j(z_j)^2}\int_Y|f_0|^2.
\end{equation}
There exists a holomorphic $(n_2,0)$ form $\tilde f_0$ on $Y$ such that $\tilde F_0=\pi_1^*(dw_1\wedge\ldots\wedge dw_{n_1})\wedge\pi_2^*(\tilde f_0)$ on $\{z_0\}\times Y$. Hence $\tilde f_0(y_0)=f_2(y_0)=f_0(y_0)$, which implies that $\int_Y|\tilde f_0|^2\ge \int_Y|f_0|^2$. Combining inequality \eqref{eq:1226e}, we have $\inf\big\{\int_M|\tilde F|^2:F\in H^0(M,\mathcal{O}(K_M))\,\&\,\tilde F|_{\{z_0\}\times Y}=\pi_1^*(dw_1\wedge\ldots\wedge dw_{n_1})\wedge\pi_2^*(\tilde f_0)\big\}<\frac{(2\pi)^{n_1}}{\prod_{1\le j\le n_1}c_j(z_j)^2}\int_Y|\tilde f_0|^2$, which contradicts to Theorem \ref{thm:exten-fibra-single}, hence $\prod_{1\le j\le n_1}c_j(z_j)^2 B_Y(y_0)= \pi^{n_1}B_M((z_0,y_0))$.

Thus, Theorem \ref{thm:suita} holds.

Note that $B_{M_1}((z_0,y_0))\ge B_M((z_0,y_0))>0$ and $B_{M_1}((z_0,y_0))= B_M((z_0,y_0))$ if and only if $M=M_1$, thus Theorem \ref{thm:suita} shows Remark \ref{r:suita} holds.

\subsection{Proofs of Theorem \ref{thm:extend} and Remark \ref{r:extend}}
\

Let $f_1=dw_1\wedge\ldots\wedge dw_{n_1}\wedge d\tilde w_1\wedge\ldots\wedge d\tilde w_{n_2}$ on $V_0\times U_0$, and let $f_2=d\tilde w_1\wedge\ldots\wedge d\tilde w_{n_2}$ on $U_0$. Let $\psi=\max_{1\le j\le n_1}\left\{\pi_{1,j}^*(2n_1G_{\Omega_j}(\cdot,z_j))\right\}$. Following from Lemma \ref{l:0}, we get that $(H_1-H_2,(z_0,y))\in\mathcal{I}(\psi)_{(z_0,y)}$ for any $y\in Y$ if and only if $(H_1-H_2)|_{\{z_0\}\times Y}=0$, where $H_1$ and $H_2$ are holomorphic $(n,0)$ form on a neighborhood of $\{z_0\}\times Y$. Let $f$ be a holomorphic $(n_2,0)$ form on $Y$ satisfying $\int_Y|f|^2<+\infty$. It follows from Proposition \ref{p:exten-fibra} that there exists a holomorphic $(n,0)$ form $F$ on $M$ such that
$F|_{\{z_0\}\times Y}=\pi_1^*(dw_1\wedge\ldots\wedge dw_{n_1})\wedge\pi_2^*(f)$
 and
$$\int_{M}|F|^2\rho\le \frac{(2\pi)^{n_1}\rho(z_0)}{\prod_{1\le j\le n_1}c_j(z_j)^2}\int_Y|f|^2.$$
Note that
$$B_Y(y_0)=\frac{2^{n_2}}{\inf\left\{\int_Y|f|^2:f\in H^0(Y,\mathcal{O}(K_Y))\,\&\,f(y_0)=f_2(y_0)\right\}}$$
and
$$B_{M,\rho}((z_0,y_0))=\frac{2^{n}}{\inf\left\{\int_M|F|^2\rho:F\in H^0(M,\mathcal{O}(K_M))\,\&\,F((z_0,y_0))=f_1((z_0,y_0))\right\}}.$$
Thus, we have $\prod_{1\le j\le n_1}c_j(z_j)^2 B_Y(y_0)\le \pi^{n_1}\rho(z_0)B_{M,\rho}((z_0,y_0)).$

In the following, we prove the characterization of the holding of the  equality $\prod_{1\le j\le n_1}c_j(z_j)^2 B_Y(y_0)= \pi^{n_1}\rho(z_0)B_{M,\rho}((z_0,y_0))$.

There exists a holomorphic $(n_2,0)$ form $f_0$ on $Y$ such that $f_0(y_0)=f_2(y_0)$ and
$$B_Y(y_0)=\frac{2^{n_2}}{\int_Y|f_0|^2}>0.$$
It follows from Proposition \ref{p:exten-fibra} that there exists a holomorphic $(n,0)$ form $F_0$ on $M$ such that $F_0=\pi_1^*(dw_1\wedge\ldots\wedge dw_{n_1})\wedge\pi_2^*(f_0)$ and
\begin{equation}
	\label{eq:1226f}\int_{M}|F_0|^2\rho\le \frac{(2\pi)^{n_1}\rho(z_0)}{\prod_{1\le j\le n_1}c_j(z_j)^2}\int_Y|f_0|^2.
\end{equation}

Firstly, we prove the necessity. Note that $B_{M,\rho}((z_0,y_0))\ge\frac{2^n}{\int_M|\tilde F|^2\rho}$ for any holomorphic $(n,0)$ form $\tilde F$  on $M$ satisfying that  $\tilde F=\pi_1^*(dw_1\wedge\ldots\wedge dw_{n_1})\wedge\pi_2^*(f_0)$ on $\{z_0\}\times Y$. Combining $\prod_{1\le j\le n_1}c_j(z_j)^2 B_Y(y_0)= \pi^{n_1}\rho(z_0)B_{M,\rho}((z_0,y_0))$, $B_Y(y_0)=\frac{2^{n_2}}{\int_Y|f_0|^2}$ and inequality \eqref{eq:1226f}, we obtain that  $\frac{(2\pi)^{n_1}\rho(z_0)}{\prod_{1\le j\le n_1}c_j(z_j)^2}\int_Y|f_0|^2=\inf\big\{\int_{M}|\tilde F|^2\rho:\tilde F\in H^0(M,\mathcal{O}(K_M))\,\&\, \tilde F|_{\{z_0\}\times Y}=\pi_1^*(dw_1\wedge\ldots\wedge dw_{n_1})\wedge\pi_2^*(f_0)\big\}.$
It follows from Theorem \ref{thm:exten-fibra-single} that $\chi_{j,z_j}=\chi_{j,-u_j}$ for any $1\le j\le n_1$.

Now, we prove $\prod_{1\le j\le n_1}c_j(z_j)^2 B_Y(y_0)= \pi^{n_1}\rho(z_0)B_{M,\rho}((z_0,y_0))$ by contradiction: if not, there exists a holomorphic $(n,0)$ form $\tilde F_0$ on $M$ such that $\tilde F_0((z_0,y_0))=f_1((z_0,y_0))$ and
\begin{equation}
	\label{eq:1226g}\int_M|\tilde F_0|^2\rho<\frac{(2\pi)^{n_1}\rho(z_0)}{\prod_{1\le j\le n_1}c_j(z_j)^2}\int_Y|f_0|^2.
\end{equation}
There exists a holomorphic $(n_2,0)$ form $\tilde f_0$ on $Y$ such that $\tilde F_0=\pi_1^*(dw_1\wedge\ldots\wedge dw_{n_1})\wedge\pi_2^*(\tilde f_0)$ on $\{z_0\}\times Y$. Hence $\tilde f_0(y_0)=f_2(y_0)=f_0(y_0)$, which implies that $\int_Y|\tilde f_0|^2\ge \int_Y|f_0|^2$. Combining inequality \eqref{eq:1226g}, we have $\inf\big\{\int_M|\tilde F|^2\rho:F\in H^0(M,\mathcal{O}(K_M))\,\&\,\tilde F|_{\{z_0\}\times Y}=\pi_1^*(dw_1\wedge\ldots\wedge dw_{n_1})\wedge\pi_2^*(\tilde f_0)\big\}<\frac{(2\pi)^{n_1}\rho(z_0)}{\prod_{1\le j\le n_1}c_j(z_j)^2}\int_Y|\tilde f_0|^2$, which contradicts to Theorem \ref{thm:exten-fibra-single}, hence $\prod_{1\le j\le n_1}c_j(z_j)^2 B_Y(y_0)= \pi^{n_1}\rho(z_0)B_{M,\rho}((z_0,y_0))$.

Thus, Theorem \ref{thm:extend} holds.

Note that $B_{M_1,\rho}((z_0,y_0))\ge B_{M,\rho}((z_0,y_0))>0$ and $B_{M_1,\rho}((z_0,y_0))= B_{M,\rho}((z_0,y_0))$ if and only if $M=M_1$, thus Theorem \ref{thm:extend} shows Remark \ref{r:extend} holds.

\


\vspace{.1in} {\em Acknowledgements}.
The second named author was supported by NSFC-11825101, NSFC-11522101 and NSFC-11431013.

\bibliographystyle{references}
\bibliography{xbib}

\end{document}